\documentclass[12pt,reqno]{amsart}

\usepackage{mathrsfs}
\usepackage{amsmath,amsthm,amssymb,amsfonts,amscd,amstext}
\usepackage{latexsym}
\usepackage{color}
\usepackage[dvips]{epsfig}
\usepackage[english]{babel}
\usepackage{xspace}
\usepackage{caption}
\usepackage{subcaption}
\usepackage{hyperref}
\usepackage{textcomp}
\usepackage{tikz}
\usepackage{rotating}
\usepackage[utf8]{inputenc}
\usepackage{cite}
\usepackage{bm}
\usepackage{enumerate}
\usepackage{bbm}
\usepackage{verbatim}

\theoremstyle{plain}
\newtheorem{theorem}{Theorem}[section]

\newtheorem{proposition}[theorem]{Proposition}
\newtheorem{lemma}[theorem]{Lemma}
\theoremstyle{definition}
\newtheorem{definition}[theorem]{Definition}
\newtheorem{remark}[theorem]{Remark}
\newtheorem{example}[theorem]{Example}
\newtheorem{examples}[theorem]{Examples}
\numberwithin{theorem}{section}
\numberwithin{equation}{section}

\newcommand{\average}{{\mathchoice {\kern1ex\vcenter{\hrule height.4pt
width 6pt depth0pt} \kern-9.7pt} {\kern1ex\vcenter{\hrule
height.4pt width 4.3pt depth0pt} \kern-7pt} {} {} }}

\def\R{\mathbb{R}}
\def\C{\mathbb{C}}
\def\N{\mathbb{N}}

\newcommand{\scrS }{\mathscr{S}}

\renewcommand{\k}{\kappa}
\newcommand{\tk}{\tilde{\kappa}}

\renewcommand{\H}{{\mathcal H}}

\newcommand{\Bquad}{{Q}}
\newcommand{\bquad}{{q}}

\newcommand{\eps}{\varepsilon}

\let \Re \relax
\DeclareMathOperator{\Re}{Re}
\let \Im \relax
\DeclareMathOperator{\Im}{Im}

\newcommand{\LS}{Lopatinski\u{\i}-\v{S}apiro\xspace}

\newcommand{\y}{\varrho}
\newcommand{\tchi}{\tilde{\chi}}
\newcommand{\nhd}{neighborhood\xspace}

\newcommand{\bld}[1]{\mbox{\boldmath $#1$}}

\newcommand{\U}{\mathscr U}

\newcommand{\br}{\ensuremath{_{|x_d=0^+}}}
\newcommand{\bd}{\ensuremath{_{|\partial\Omega}}}

\newcommand{\ovl}[1]{\overline{#1}}

\newcommand{\Rdp}{\R^d_+}

\newcommand{\Op}{\ensuremath{\mathrm{Op}}}

\newcommand{\Opt}{\ensuremath{\mathrm{Op}_{\mbox{\tiny ${\mathsf T}$}}}}

\newcommand{\suff}{sufficiently\xspace}

\DeclareMathOperator{\trace}{tr}

\DeclareMathOperator{\range}{Ran}

\newcommand{\E}{\ensuremath{\mathcal E}}

\newcommand{\transp}{\ensuremath{\phantom{}^{t}}}

\newcommand{\Norm}[2]{{\| #1 \|}_{#2}}
\newcommand{\bigNorm}[2]{{\big\| #1 \big\|}_{#2}}

\newcommand{\norm}[2]{{|#1|}_{#2}}

\newcommand{\Normsc}[2]{{\| #1 \|}_{#2, \tau}}
\newcommand{\normsc}[2]{{| #1 |}_{#2, \tau}}

\newcommand{\bignormsc}[2]{{\big| #1 \big|}_{#2, \tau}}

\newcommand{\inp}[2]{( #1, #2 )} 
\newcommand{\biginp}[2]{\big( #1, #2 \big)}

\newcommand{\dup}[2]{\langle #1 , #2 \rangle}

\newcommand{\lsc}{\lambda_\tau}
\newcommand{\lsct}{\lambda_{\mbox{\tiny ${\mathsf T}$},\tau}}

\newcommand{\Lsct}{\Lambda_{\mbox{\tiny ${\mathsf T}$},\tau}}

\newcommand{\Ssc}{S_\tau}
\newcommand{\Psisc}{\Psi_\tau}
\newcommand{\Dsc}{\mathscr{D}_{\tau}}

\newcommand{\Ssct}{S_{\mbox{\tiny ${\mathsf T}$},\tau}}
\newcommand{\Psisct}{\Psi_{\mbox{\tiny ${\mathsf T}$},\tau}}
\newcommand{\Dsct}{\mathscr{D}_{\mbox{\tiny ${\mathsf T}$},\tau}}

\newcommand{\Psig}{P_{\sigma}}
\newcommand{\psig}{p_{\sigma}}
\newcommand{\Qsig}{Q_{\sigma}}
\newcommand{\qsig}{q_{\sigma}}

\newcommand{\Pconj}{P_{\sigma,\varphi}}
\newcommand{\pconj}{p_{\sigma,\varphi}}
\newcommand{\Qconj}{Q_{\sigma,\varphi}}
\newcommand{\qconj}{q_{\sigma,\varphi}}
\newcommand{\Qs}{Q_{s}}
\newcommand{\qs}{q_{s}}
\newcommand{\Qa}{Q_{a}}
\newcommand{\qa}{q_{a}}

\newcommand{\un}[1]{\Sigma_{#1}}
\newcommand{\z}{\mathbf{z}}

\newcommand{\Con}{\ensuremath{\mathscr C}}

\newcommand{\Cinf}{\ensuremath{\Con^\infty}}
\newcommand{\Cinfc}{\ensuremath{\Con^\infty_c}}
\newcommand{\Cbarc}{\ensuremath{\overline \Con^\infty}_c}
\newcommand{\Sbarp}{\ensuremath{\overline{\mathscr{S}}(\Rdp)}}

\newcommand{\Pell}{\mathsf P_0}

\newcommand{\csp}{\gamma}

\newcommand{\scrO}{\mathscr O}

\DeclareMathOperator{\id}{Id}
\DeclareMathOperator{\supp}{supp}
\DeclareMathOperator{\rank}{rank}
\DeclareMathOperator{\dist}{dist}
\DeclareMathOperator{\ind}{ind}

\DeclareMathOperator{\Span}{Span}

\begin{document}
\title[]{Stabilization of the damped plate equation under general
  boundary conditions}

\author{Jérôme Le Rousseau}
\address{Universit\'e Sorbonne Paris Nord, Laboratoire Analyse, G\'eom\'etrie et Applications, LAGA, CNRS, UMR 7539, F-93430, Villetaneuse, France.}
\email{jlr@math.univ-paris13.fr}

\author{Emmanuel Wend-Benedo Zongo}
\address{
Dip. di Matematica, Via Saldini 50, 20133 Milano, Italy
and
 Universit\'e Sorbonne Paris Nord, Laboratoire Analyse, G\'eom\'etrie et Applications, LAGA, CNRS, UMR 7539, F-93430, Villetaneuse, France.}
\email{wend.zongo@unimi.it}

\date{today}
\begin{abstract}
We consider a damped plate equation on an open bounded
subset of $\R^d$, or a smooth manifold, with boundary, along with
general boundary operators fulfilling the \LS condition. 
The damping term acts on a region without
imposing a geometrical condition. We derive a
resolvent estimate for the generator of the damped plate semigroup
that yields a logarithmic decay of the energy of the solution to the
plate equation. The resolvent estimate is a consequence of a Carleman
inequality obtained for the bi-Laplace operator involving a spectral
parameter under the considered boundary conditions. The derivation
goes first though microlocal estimates, then local estimates, and
finally a global estimate.\\

\end{abstract}

\maketitle
\tableofcontents

\section{Introduction}

 Let $\Omega$ be a bounded connected open subset in $\R^d$, or a smooth bounded connected $d$-dimensional manifold, with smooth
boundary $\partial\Omega$, where we  consider a damped plate equation 
\begin{equation}\label{eq: damped plate equation - intro}
\begin{cases}
\partial^2_t y +\Delta^2y+\alpha(x)\partial_t y =0 
&(t,x) \in  \R_+ \times\Omega,\\
B_1 y_{|\R_+ \times\partial \Omega}= B_2y_{|\R_+ \times\partial \Omega}=0,\\
y_{|t=0}=y^0, \ \ \partial_t y_{|t=0}=y^1,
\end{cases}
\end{equation}
where $\alpha\geq 0$ and 
where $B_1$ and $B_2$ denote two boundary differential operators.
The damping property is provided by $+\alpha(x)\partial_t$ thus
referred as the {\em damping term}.
As introduced below $\Delta^2$ is the bi-Laplace operator, that is,
the square of the Laplace operator. Here, it is associated with a
smooth metric $g$ to be introduced below; it is thus rather the
bi-Laplace-Beltrami operator. This equation appears in models for the description of mechanical vibrations of
thin domains. The two boundary  operators are of $k_j$, $j=1,2$ respectively, yet at most of order $3$ in the direction
 normal to the boundary. They 
are chosen such that the two following properties are fulfilled: 
\begin{enumerate}
  \item the \LS boundary condition holds (this condition is fully
    described in what follows);
  \item along with the homogeneous boundary conditions given above the
    bi-Laplace operator is self-adjoint and nonnegative. This guarantees the
   conservation of the energy of the solution in the case of a
    damping free equation, that is, if $\alpha=0$.  
\end{enumerate}
We are concerned with the decay of the energy of the solution in the
case $\alpha$ is not identically zero. We shall prove that the damping
term yields a stabilization property: the energy decays to zero as
time $t \to \infty$ and we shall prove that the decay rate is at least logarithmic.

\subsection{On the stabilization of waves and plates}
If no geometrical condition is imposed on the damping region, here as
given by the
support of the function $\alpha$ one cannot expect a exponential decay
rate as observed in the case this region fulfills the celebrated
Rauch-Taylor condition, often coined GCC for geometrical control
condition \cite{RT:74,BLR:92}. The GCC expresses that all rays of
geometrical optics reach the damping region $\{\alpha>0\}$ in a finite time. 
Here, with no such condition, a logarithmic decay rate is quite
natural if one has in mind the equivalent result obtained for the wave
equation in the works of G.~Lebeau~\cite{Lebeau:96} and G.~Lebeau and
L.~Robbiano~\cite{LR:97}. We also refer to Chapter~6 in
\cite{JGL-vol1} and Chapters~10 and 11 in \cite{JGL-vol2} where the
result of \cite{Lebeau:96,LR:97} are reviewed and generalized in
particular on the framework of general boundary conditions as those we
consider here. See also the work of P.~Cornilleau and L.~Robbiano for
a quite exotic boundary condition, namely the Zaremba condition.

\bigskip
Among the existing results available in the literature for plate type
equations, many of them concern the ``hinged'' boundary conditions,
that is, $u_{|\partial \Omega}=0$ and $\Delta u_{|\partial \Omega}=0$.
We first mention these result. An important result obtained in
\cite{Jaffard:90} on the controllability of the plate equation on a
rectangle domain with an arbitrarily small control domain. The method
relies on the generalization of Ingham type inequalities in
\cite{Kahane:62}. An exponential stabilization result, in the same
geometry, can be found in \cite{RTT:06}, using similar techniques. In
\cite{RTT:06} the localized damping term involves the time derivative
$\partial_ty$ as in \eqref{eq: damped plate equation - intro}. Interior nonlinear
feedbacks can be used for exponential stabilization
\cite{Tebou:09}. There, feedbacks are localized in a neighborhood of
part of the boundary that fulfills multiplier-type conditions.  A
general analysis of nonlinear damping that includes the plate equation
is provided in \cite{AA:11} under multiplier-type conditions.  For
``hinged'' boundary conditions also, with a boundary damping term, we
cite \cite{ATT:07} where, on a square domain, a necessary and
sufficient condition is provided for exponential stabilization.

Note that under ``hinged'' boundary conditions the bi-Laplace operator
is precisely the square of  the Dirichlet-Laplace operator. This makes
its mathematical analysis much easier, in particular where using
spectral properties, and this explains why this type of boundary
conditions appears very frequently in the mathematical literature.  

A more challenging type of boundary condition is the so-called 
``clamped'' boundary conditions,
that is, $u_{|\partial \Omega}=0$ and $\partial_\nu  u_{|\partial \Omega}=0$,
for which few results are available. We
cite \cite{Alabau:06}, where a general analysis of nonlinearly damped
systems that includes the plate equation 
under multiplier-type conditions is provided. In \cite{APT:17}, the
analysis of discretized general nonlinearly damped system is also
carried out, with the plate equation as an application.  In \cite{Tebou:12}, a nonlinear
damping involving the p-$Laplacian$ is used also under multiplier-type
conditions. In \cite{DS:15}, an
exponential decay is obtained in the case of ``clamped'' boundary
conditions, yet with a damping term of the Kelvin-Voigt type, that is
of the form $\partial_t \Delta y$, that acts over the whole domain.
In the case of the ``clamped'' boundary conditions,  the logarithmic-type stabilization
result we obtain here was proven in \cite{JL}. The present article
thus stands as a generalization of the stabilization result of
\cite{JL} if considering a whole class of boundary condition instead
of specializing to a certain type. The present work contains in
particular also the case of ``hinged'' boundary conditions.

\subsection{Method}
Following the works of \cite{Lebeau:96,LR:97,JL} we obtain a
logarithmic decay rate for the energy of the solution to~\eqref{eq: damped plate equation - intro} 
by means of a resolvent estimate
for the generator of the semigroup associated with the damped plate
equation~\eqref{eq: damped plate equation - intro}. This estimate follows from a Carleman inequality
derived for the operator $\Psig=\Delta^2 - \sigma^4$ where $\sigma$ is a
spectral parameter for the   generator of the semigroup.

Our first goal is thus the derivation of the Carleman inequality for
the operator $\Psig$ near the boundary under the boundary conditions
given by $B_1$ and $B_2$.

Then, from the Carleman estimate one deduces an observation inequality
for the operator $\Psig$ in the case of the prescribed boundary
conditions.  The resolvent estimate then follows from this observation inequality.

\subsection{On Carleman estimates}
A Carleman estimate is  a weighted {\em a priori} inequality for the solutions of a partial partial differential equation, where the weight is of exponential type. 
For instance, for a partial differential operator $P$ away from the boundary, it takes the form 
\begin{align*}
\Norm{e^{\tau\varphi}u}{L^2(\Omega)}
  \leq  C \Norm{e^{\tau\varphi} P u}{L^2(\Omega)},
\end{align*}
for $u \in \Cinfc(\Omega)$ and $\tau\geq\tau_0$ for $\varphi$ well
chosen and  some $\tau_0$ chosen \suff large. The exponential weight
function involves a parameter $\tau$ that can be taken as large as
desired, making Carleman inequalities very powerful estimates.
Additional terms on the left-hand side of the inequality can be
obtained, including higher-order derivatives of the function $u$,
depending of course of the order of the operator $P$ itself. 
For a second-order elliptic operator such as the Laplace operator one
has   
\begin{align*}
 \tau^{3/2}\Norm{e^{\tau\varphi}u}{L^2(\Omega)}
  + \tau^{1/2}\Norm{e^{\tau\varphi} D u}{L^2(\Omega)}
  + \tau^{-1/2}\sum_{|\beta|=2}\Norm{e^{\tau\varphi} D^\beta u}{L^2(\Omega)}
  \leq  C \Norm{e^{\tau\varphi} \Delta u}{L^2(\Omega)},
\end{align*}
under the so-called sub-ellipticity condition; see Chapter~3 in
\cite{JGL-vol1}. Note that the power of the large parameter $\tau$
adds to $3/2$ with the order of the derivative in each term on the
left-hand side. In fact, in the calculus used to derive such estimates
one power of $\tau$ is equivalent to a derivative of order one. Thus
with this $3/2$ compared with the order two of the operator one says
that one looses a half-derivative in the estimate.

This type of estimate was used for the first time by T.~Carleman
\cite{Carleman:39} to achieve uniqueness properties for the Cauchy
problem of an elliptic operator. Later, A.-P.~Calder\'on and
L.~H\"ormander further developed Carleman's method
\cite{Calderon:58,Hoermander:58}.  To this day, the method based on Carleman estimates
remains essential to prove unique continuation properties;
see for instance \cite{Zuily:83} for an overview.  On such questions,
more recent advances have been concerned with differential operators
with singular potentials, starting with the contribution of D.~Jerison
and C.~Kenig \cite{JK:85}. There, Carleman estimates rely on
$L^p$-norms rather than $L^2$-norms as in the estimates
above. The proof of such $L^p$ Carleman estimates is very
delicate. The reader is also referred to \cite{Sogge:89,KT:01,KT:02,DDSF:05,KT:05}.
In more recent years, the field of applications of Carleman estimates
has gone beyond the original domain; they are also used in the study
of:
\begin{itemize}
\item Inverse problems, where Carleman estimates are used to obtain
  stability estimates for the unknown sought quantity (for instance
  coefficient, source term) with respect to norms on measurements
  performed on the solution of the PDE, see for instance
  \cite{BK:81,Isakov:98,Kubo:00,IIY:03}; Carleman estimates are also
  fundamental in the construction of complex geometrical optic solutions
  that lead to the resolution of inverse problems such as the Calder\'on
  problem with partial data \cite{KSU:07,DKSU:09}.
  
\item Control theory for PDEs;  Carleman estimates yield the null controllability of linear parabolic equations
\cite{LR:95} and the null controllability of classes of semi-linear
parabolic equations \cite{FI:96,Barbu:00,FZ:00}.
They can also be used to prove unique continuation properties, that in
turn are crucial for the treatment of low frequencies for  exact
controllability results for  hyperbolic equations as in \cite{BLR:92}.
\end{itemize}

For a function supported near a point at the boundary, in normal geodesic coordinates where $\Omega$ is locally given by $\{
x_d >0\}$ (see Section~\ref{sec:
  Geometrical setting} below) the estimate
can take the form
\begin{align*}
  \sum_{|\beta|\leq 2} \tau^{3/2-|\beta|}
  \Norm{e^{\tau\varphi} D^\beta u}{L^2(\Omega)}
  + \sum_{|\beta|\leq 1} \tau^{3/2-|\beta|}
  \norm{e^{\tau\varphi} D^\beta u\br}{L^2(\Omega)}
  \leq  C \Norm{e^{\tau\varphi} \Delta u}{L^2(\Omega)}.
\end{align*}
This is the type of estimate we seek here for the operator $\Psig$,
with some uniformity with respect to $\sigma$.

\subsection{Geometrical setting}
\label{sec: Geometrical setting}

On $\Omega$ we consider a Riemannian metric $g_x = (g_{ij}(x))$, with
associated cometric $(g^{ij}(x))  = (g_x)^{-1}$.
It stands as a bilinear form that act on vector fields,
\begin{align*}
  g_x (u_x, v_x) = g_{ij}(x) u_x^i v_x^j,
  \quad u_x = u_x^i \partial_{x_i}, \
  v_x = v_x^i \partial_{x_i}.
\end{align*}

\medskip
For $x \in \partial \Omega$ we denote by $\nu_x$ the unit outward
pointing {\em normal vector} at $x$, unitary in the sense of the
metric $g$, that is
\begin{align*}
  g_x(\nu_x, \nu_x)=1
  \ \ \text{and} \ \
  g_x(\nu_x, u_x)=1 \quad \forall u_x \in T_x \partial \Omega.
  \end{align*}
We denote by $\partial_\nu$ the associated derivative at the boundary,
that is, $\partial_\nu f (x) = \nu_x (f)$. 
We also denote by $n_x$ the unit outward
pointing {\em conormal vector} at $x$, that is,
$n_x = \nu_x^\flat$, that is, $(n_x)_i = g_{i j} \nu_x^j$. 

Near a boundary point we shall often use normal geodesic coordinates
where $\Omega$ is locally given by $\{ x_d >0\}$ and the  
metric $g$ takes the form
\begin{align*}
  g=dx^d\otimes dx^d
  +\sum\limits_{1\leq i,j\leq d-1} g_{ij} dx^i\otimes  dx^j.
\end{align*}
Then, the vector field $\nu_x$ is locally given by $(0, \dots, 0,
-1)$. The same for the  one form $n_x$.  

Normal geodesic coordinates allow us to locally formulate 
boundary problems in a half-space  geometry. We write 
$$\Rdp:=\{ x\in\R^d,\ x_d>0\}\qquad \text{where}\ 
x=(x',x_d) \ \text{with}\ x'\in\R^{d-1}, x_d\in\R.$$ 
We shall naturally denote its closure by $\ovl{\Rdp}$, that is,
$\ovl{\Rdp} = \{ x\in \R^d; x_d\geq 0\}$.

The Laplace-Beltrami operator
is given by
\begin{align}
\label{eq: laplace-Beltrami}
  (\Delta_g f) (x) =(\det g_x)^{-1/2} 
  \sum_{1\leq i, j\leq d} \partial_{x_i}\big(
  (\det g_x)^{1/2} g^{i j}(x)
  \partial_{x_j}  f \big)(x).
\end{align}
in local coordinates. Its principal part is given by
$\sum_{1\leq i, j\leq d} g^{i j}(x) \partial_{x_i} \partial_{x_j}$ and
its principal symbol by $\sum_{1\leq i, j\leq d} g^{i j}(x) \xi_i
\xi_j$.

The bi-Laplace operator is $P= \Delta^2_g$. In the main text of the
article we shall write $\Delta$, $\Delta^2$ in place of $\Delta_g$, $\Delta_g^2$.

\subsection{Main results}

\subsubsection{Carleman estimate}
We state the main Carleman estimate for the operator $\Psig$ in
normal geodesic coordinates as presented in Section~\ref{sec:
  Geometrical setting}. A point $x^0 \in \partial \Omega$ is
considered and a weight function $\varphi$ is assumed to be defined
locally and 
such that 
\begin{enumerate}
  \item $\partial_d \varphi \geq C>0$ locally.
  \item $(\Delta \pm \sigma^2, \varphi)$ satisfies the sub-ellipticity
    condition of Definition~\ref{de1} locally. This is a necessary and
    sufficient condition for a Carleman estimate to hold for a
    second-order operator $\Delta \pm \sigma^2$, regardless of
    boundary conditions \cite[Chapters 3 and 4]{JGL-vol1}.
  \item $(\Psig, B_1, B_2,\varphi)$ satisfies the
  \LS condition of Definition~\ref{def: LS after conjugation} at  $\y'=(x^0,\xi',\tau,\sigma)$
  for all $(\xi',\tau,\sigma) \in \R^{d-1} \times [0,+\infty) \times
  [0,+\infty)$ such that $\tau \geq \k_0\sigma$, for some $\k_0>0$. This means that the
  \LS condition holds after the conjugation of the operator $\Psig$
  and the boundary operators $B_1$ and $B_2$ by the weight function
  $\exp(\tau \varphi)$.  
\end{enumerate}

\begin{theorem}[Carleman estimate]
  \label{theorem: local Carleman estimate-intro}
  Let $\k_0' > \k_0>0$. Let $x^0 \in \partial \Omega$.
  Let $\varphi$ be such that the properties above hold locally. 
  Then, there exists $W^0$ a \nhd of $x^0$, $C>0$, $\tau_0>0$ such that
  \begin{align}\label{eq: local estimate Pconj-final-intro}
    \tau^{-1/2} \Normsc{e^{\tau \varphi} u}{4}
    + \normsc{ \trace( e^{\tau \varphi} u)}{3,1/2}
    \leq
    C \big(
    \Norm{e^{\tau \varphi} \Psig u}{+}
    + \sum\limits_{j=1}^{2}
    \normsc{e^{\tau \varphi} B_{j}v\br}{7/2-k_{j}}
    \big),
\end{align}
for $\tau\geq \tau_0$,
$\k_0 \sigma \leq \tau \leq \k_0' \sigma $, and $u\in \Cbarc(W^0_+)$.
\end{theorem}
The volume norm is given by
\begin{align*}
  \Normsc{e^{\tau \varphi} u}{4} = \sum_{|\beta| \leq 4}
  \tau^{4 - |\beta|}\Norm{e^{\tau\varphi} D^\beta u}{L^2(\Omega)}.
\end{align*}
The trace norm is given by
\begin{align*}
  \normsc{ \trace( e^{\tau \varphi} u)}{3,1/2}
  = \sum_{0\leq n \leq 3} \normsc{ \partial_\nu^n (e^{\tau \varphi} u)\br}{7/2 -n},
\end{align*}
where the norm $\normsc{.}{7/2 -n}$ is the $L^2$-norm in $\R^{d-1}$
after applying the Fourier multiplier $(\tau^2 + |\xi'|^2)^{7/4
  -n/2}$. These norms are well described in Section~\ref{sec: Function norms}. 

Observe that the Carleman estimate of Theorem~\ref{theorem: local
  Carleman estimate-intro} exhibits a loss of a half-derivative. A
more precise statement is given in Theorem~\ref{theorem: main Carleman estimate} in Section~\ref{sec: Carleman final estimate}.

\subsubsection{Stabilization result} 

Let $(\Pell, D(\Pell))$ be the unbounded operator on $L^2(\Omega)$ given by the domain
\begin{align}
  \label{eq: domain P intro}
  D(\Pell) = \big\{ u \in H^4(\Omega); \ B_1
  u_{|\partial \Omega}
  = B_2 u_{|\partial \Omega}=0\big\},
\end{align}
given by $\Pell u = \Delta^2 u$ for $u \in D(\Pell)$. As written above
the two boundary differential operators are such that $(\Pell, D(\Pell))$
  is self-adjoint and nonnegative.

  Let $y(t)$ be a strong solution of the plate equation~\eqref{eq:
    damped plate equation - intro}. A precise definition of strong
  solutions is given in Section~\ref{sec: Strong and weak
    solutions}. One has $y^0 \in D(\Pell)$ and
  $y^1 \in D(\Pell^{1/2})$. Its energy is defined as
  \begin{align*}
    \E( y) (t) = \frac12 \big(
    \Norm{\partial_t y (t)}{L^2(\Omega)}^2
    + \inp{\Pell y (t)}{y (t)}_{L^2(\Omega)}
    \big).
  \end{align*}
\begin{theorem}[logarithmic stabilization for the damped plate equation]
  \label{theorem: stabilisation theorem-intro}
  There exists $C>0$ such that for any such strong solution to the
  damped plate equation~\eqref{eq: damped plate equation - intro}  one
  has
  \begin{align*}
    \E(y)(t) \leq  \frac{C}{\big(\log(2+t)\big)^{4}}
    \big(\Norm{\Pell y^0+ \alpha y^1}{L^2(\Omega)}^2
    + \Norm{\Pell^{1/2} y^1}{L^2(\Omega)}^2
    \big).
\end{align*}  
\end{theorem}
A more precise and more general statement is given in Theorem~\ref{theorem:
  stabilisation theorem} in Section~\ref{sec: Stabilization result}.

\subsection{Some open questions}

\subsubsection{Boundary damping}
Here, we have considered a damping that acts in the interior of the
domain $\Omega$. The study of boundary damping, as in \cite{LR:97} for
the wave equation, is also of relevance.
Yet we foresee that it requires to specify more the used boundary
operators. This was not our goal here as we wished to treat general
boundary operators here.  

\subsubsection{Spectral inequality}
If the boundary operators $B_1$ and $B_2$ are well chosen, the
bi-Laplace operator $\Delta^2$ can be selfadjoint on $L^2(\Omega)$;
see Section~\ref{sec: unbounded bi-Laplace operator}. Associated with
the operator is then a Hilbert basis  $(\phi_j)_{j \in \N}$ of
$L^2(\Omega)$. In the case of ``clamped'' boundary condition the
following spectral inequality was proven in \cite{JL}. 
    \begin{theorem}[Spectral inequality for the ``clamped'' bi-Laplace
      operator] 
      \label{theorem: spectral inequality} 
      Let $\O$ be an open subset of $\Omega$.
      There exists $C>0$ such that 
      \begin{align*}
        \Norm{u}{L^2(\Omega)} \leq C e^{C \mu^{1/4}}
        \Norm{u}{L^2(\O)}, \qquad \mu >0, \quad   u \in \Span \{\phi_j; \ \mu_j \leq \mu\}.
      \end{align*}
\end{theorem}
The proof of this theorem is based on a Carleman inequality for the
fourth-order elliptic operator $D_s^4 + \Delta^2$, that is, after the
addition of a variable $s$. Extending this strategy to the type of
boundary conditions treated here was not successful because it is not
guaranteed that having the \LS condition for $\Delta^2$, $B_1$, and
$B_2$ implies that the  \LS condition holds for $D_s^4 +\Delta^2$, $B_1$, and
$B_2$. Yet, the \LS condition is at the heart of the proof of our
Carleman estimate. Proving a spectral estimate as in the above
statement for the general boundary conditions considered here is 
an open question.

\subsubsection{Quantification of the unique continuation property}

For a second-order operator like the Laplace operator $\Delta$ and a
boundary operator $B$ of order $k$ (yet of order at most one in the
normal direction) such that the \LS holds one can derive the
following inequality that locally quantifies the unique continuation property
up to the boundary; see \cite[Lemma~9.2]{JGL-vol1}.
\begin{proposition}
  \label{prop: unique continuation boundary}
  Let $x^0\in\partial \Omega$ and $V$ be a \nhd of $x^0$ where the \LS
  condition holds. There exist $W$ an open  \nhd of $x^0$, $\eps\in
  (0,1)$, $\delta \in (0,1)$, and $C>0$, such that 
  \begin{align}
    \label{eq: local interpolation at the boundary-manifold1}
    \Norm{u}{H^1(W)} 
      \leq C \Norm{u}{H^1(V)}^{1-\delta} 
    \Big(\Norm{\Delta u}{L^2(V)} 
    + \norm{B u}{H^{1-k}(V \cap \partial \Omega)} 
    + \Norm{u}{H^1(V_\eps)}
    \Big)^\delta, 
  \end{align}  
  for  $u \in H^2(V)$, with
  $V_\eps = \{ x\in V; \ \dist(x, \partial\Omega) > \eps \}$.
\end{proposition}
This inequality can be obtained from a Carleman estimate as in
Theorem~\ref{theorem: local Carleman estimate-intro} for the Laplace
operator, yet with the large parameter $\tau$ allowed to be chosen as
large as desired. This is exploited in an optimization procedure on
the parameter $\tau$ in $[\tau_0,+\infty[$ for some $\tau_0$. Note
however that in the statement of Theorem~\ref{theorem: local Carleman
  estimate-intro}  one has $\sigma \lesssim \tau \lesssim
\sigma$. Thus, the
optimization procedure cannot be carried out in $[\tau_0,+\infty[$. While having a
result quantifying the unique continuation under \LS-type conditions
for the bi-Laplace operator
similar to that of Proposition~\ref{prop: unique continuation
  boundary} can be 
expected, the Carleman estimate we obtain in the present article cannot be used, at
least directly, for a proof as carried out in  \cite{JGL-vol1} or in
former sources such as \cite{LR:95}.

\subsection{Outline}
This article is organized as follows.
In Section~\ref{sec: Psido} we recall some basic aspects of
pseudo-differential operators with a large parameters $\tau>0$ and
some positivity inequality of G{\aa}rding type in particular for
quadratic forms in a half-space or at the boundary. Associated with
the large parameters are Sobolev like norms, also in a half-space or
at the boundary.

In Section~\ref{sec: LS condition}, the \LS boundary condition are
properly defined for an elliptic operator, we give examples focusing
on the Laplace and bi-Laplace operator and we give a formulation in
local  normal geodesic coordinates that we shall mostly use throughout
the article. For the bi-Laplace operator we provide a series of
examples of boundary operators for which the \LS boundary conditions
holds and moreover the resulting operator is symmetric. We also show
that the algebraic conditions that characterize the \LS condition are
robust under perturbation. This last aspect is key in the
understanding of how the \LS condition get preserved under conjugation
and the introduction of a spectral parameter. This is done in
Section~\ref{sec: LS condition for conjugated bilaplace}, where an
analysis of the configuration of the  roots of the conjugated
bi-Laplace operator is performed. In Section~\ref{sec: Symbol postivity at the boundary} the \LS condition for the conjugated
operator is exploited to obtain a symbol positivity for a quadratic
form to prepare for the derivation of a Carleman estimate.

In Section~\ref{sec: boundary norm under LS} we derive a estimation of
the boundary traces. This is precisely where the \LS condition is
used. The result is first obtained microlocally and we then apply a
patching procedure.

To obtain the Carleman estimate for the bi-Laplace operator with
spectral parameter $\Delta^2 - \sigma^4$ in Section~\ref{sec: Microlocal estimate for each second-order factors} we first
derive microlocal estimates for  the operators $\Delta \pm  \sigma^2$.
Imposing $\sigma$ to be non-zero, in the sense that $\sigma \gtrsim
\tau$, the previous estimates exhibits losses in different microlocal
regions. Thus concatenating the two estimates one derives an estimate
for $\Delta^2 - \sigma^4$ where losses do not accumulates. A local  Carleman
estimate with only a loss of a half-derivative is obtained.  This is
done in Section~\ref{sec: Local Carleman estimate for the fouth-order operator}. With the traces estimation obtained in Section~\ref{sec:
  boundary norm under LS} one obtains the local Carleman estimate of
Theorem~\ref{theorem: local Carleman estimate-intro}.

For the application to stabilization we have in mind, in
Section~\ref{sec: Global Carleman estimate} we use a global weight
function and derive a global version of the Carleman estimate for
$\Delta^2 - \sigma^4$ on the whole $\Omega$.
This leads to an observability inequality.

In Section~\ref{sec: Solutions to the damped plate equations} we
recall aspects of strong and weak solutions to the damped plate
equation, in particular through a semigroup formulation.
With the observability inequality obtained in Section~\ref{sec: Global
  Carleman estimate} we derive in Section~\ref{sec:resolvent, stab} a
resolvent estimate for the generator of the plate semigroup that in
turn implies the stabilization result of Theorem~\ref{theorem: stabilisation theorem-intro}.

\subsection{Some notation}

The canonical inner product in $\C^m$ is denoted by $(z,z')_{\C^m}=\sum\limits_{k=0}^{m-1}z_k
\bar{z'}_k,$ for $z=(z_0,\cdots,z_{m-1})\in\C^m,$
$z'=(z'_0,\cdots,z'_{m-1})\in\C^m.$ The associated norm will be
denoted $|z|^2_{\C^m}=\sum\limits_{k=0}^{m-1}|z_k|^2.$

We shall use the notations $a\lesssim b$ for $a\leq C b$ and
$a\gtrsim b$ for $a\geq C b$, with a constant $C>0$ that may change
from one line to another. We also write $a\asymp b$ to denote
$a\lesssim b\lesssim a.$

For an open set $U$ of $\R^d$ we set $U_+ = U \cap \Rdp$ and 
\begin{equation}
   \label{eq: notation Cbarc}
 \Cbarc  (U_+)=\{ u=v_{|\Rdp}; v\in\Cinfc(\R^d) \text{ and }\supp (v)\subset U\}.
\end{equation}

We set $\Sbarp=\{u_{|\Rdp};\ u\in\scrS(\R^d)\}$ with
$\scrS(\R^d)$ the usual Schwartz space in $\R^d$:
\begin{align*}
  u \in \scrS(\R^d)
  \ \ \Leftrightarrow \ \ 
  u\in \Cinf(\R^d) \ \text{and} \ 
  \forall\alpha,\beta\in\N^d 
  \sup\limits_{x\in\R^d} |x^{\alpha}D_x^{\beta} u(x)|<\infty.
\end{align*}

We recall that the Poisson bracket of two smooth functions is given by $$\{f,g\}=\sum\limits_{j=1}^{d}\left(
\partial_{\xi_j}f\partial_{x_j}g-
\partial_{x_j}f\partial_{\xi_j}g\right).$$

\section{Pseudo-differential calculus: notation, definitions, and some
  properties}
\label{sec: Psido}

In a half-space geometry motivated by the normal geodesic coordinates
introduced in Section~\ref{sec: Geometrical setting}
we shall use $\xi=(\xi',\xi_d)\in\R^{d-1}\times\R$ and we shall
consider the operators $D_d=-i\partial_d$ and $D'=-i\partial'$, with
$\partial' = (\partial_{x_1}, \dots, \partial_{x_{d-1}})$.

\subsection{Pseudo-differential operators with a large a parameter on $\R^d$}\label{ss1}
In this subsection we recall some notions on semi-classical
pseudo-differential operators with large parameter $\tau\geq 1.$ We
denote by $\Ssc^m$ the space of smooth functions $a(x,\xi,\tau)$
defined on $\R^d\times\R^d,$ with $\tau\geq 1$ as a large parameter,
that satisfies the following : for all multi-indices
$\alpha,\beta\in\N^d$ and $m\in\R$, there exists
$C_{\alpha,\beta}>0$ such
that
\begin{align*}
  |\partial_x^{\alpha}\partial^{\beta}_{\xi}a(x,\xi,\tau)|
  \leq C_{\alpha,\beta}\lsc^{m-|\beta|},
  \ \ \text{where} \ \ \lsc^2=\tau^2+|\xi|^2,
\end{align*}
for all
$(x,\xi,\tau)\in\R^d\times\R^d\times [1, \infty).$ For $a\in \Ssc^m,$
one defines the associated pseudo-differential operator of order $m$, denoted by
$A=\Op (a)$, 
\begin{align*}
  Au(x):=\frac{1}{(2\pi)^d}
  \int_{\R^d}e^{ix\cdot\xi}a(x,\xi,\tau)\hat{u}(\xi)d\xi,
  \qquad u\in \mathscr{S}(\R^d).
\end{align*}
One says that $a$ is the symbol of $A$. We denote $\Psisc^m$ the set
of pseudo-differential operators of order $m.$ We shall denote by
$\Dsc^m$ the space of semi-classical differential operators, i.e, the
case when the symbol $a(x,\xi,\tau)$ is a polynomial function of order
$m$ in $(\xi,\tau).$

\subsection{Tangential pseudo-differential operators with a large a parameter}\label{ss2}
Here, we consider pseudo-differential operators that only act in the
tangential direction $x'$ with $x_d$ behaving as a parameter. We shall
denote by $\Ssct^m$, the set of smooth functions $b(x,\xi',\tau)$
defined for $\tau\geq 1$ as a large parameter and satisfying the
following: for all multi-indices $\alpha\in\N^d$,
$\beta\in\N^{d-1}$ and $m\in\R$, there exists
$C_{\alpha,\beta}>0$ such
that
\begin{align*}
  |\partial_x^{\alpha}\partial^{\beta}_{\xi'}b(x,\xi',\tau)|\leq
  C_{\alpha,\beta}\lsct^{m-|\beta|}, \  \
  \text{where} \ \ \lsct^2=\tau^2+|\xi'|^2,
\end{align*}
for all
$(x,\xi',\tau)\in\R^d\times\R^{d-1}\times [1, \infty).$ For $b\in \Ssct^m$,
we define the associated tangential pseudo-differential operator
$B=\Opt (b)$ of order $m$ by
  \begin{align*}
    Bu(x):=\frac{1}{(2\pi)^{d-1}}
    \int_{\R^{d-1}}e^{i x'\cdot \xi'}b(x,\xi',\tau)\hat{u}(\xi',x_d)d{\xi'},
    \qquad u\in\Sbarp.
    \end{align*}
We define $\Psisct^m$ as the set of tangential
pseudo-differential operators of order $m,$ and
$\Dsct^m$ the set of tangential differential
operators of order $m.$ We also set
\begin{align*}
  \Lsct^m=\Opt (\lsct^m).
\end{align*}

\bigskip
Let $m\in\N$ and $m'\in\R.$ If we consider $a$ of the form
\begin{align*}
  a(x,\xi,\tau)=\sum\limits_{j=0}^{m}a_j(x,\xi',\tau)\xi_d^j,
  \quad a_j\in \Ssct^{m+m'-j},
  \end{align*}
  we define $\Op (a)=\sum\limits_{j=0}^{m}\Opt (a_j)D^j_{x_d}.$ We
  write $a\in \Ssc^{m,m'}$ and $\Op (a)\in \Psisc^{m,m'}.$

\subsection{Function norms}
\label{sec: Function norms}
For functions norms we use the notation $\|.\|$ for functions
defined in the interior of the domain and $|.|$ for functions defined on the boundary.
In that spirit, we shall use the notation
\begin{align*}
  \Norm{u}{+}=\Norm{u}{L^2(\Rdp)}, \quad
  \inp{u}{\tilde{u}}_+=\inp{u}{\tilde{u}}_{L^2(\Rdp)},
\end{align*}
for functions defined in $\Rdp$ and
\begin{align*}
  \norm{w}{\partial}=\Norm{w}{L^2(\R^{d-1})}, \quad
  \inp{w}{\tilde{w}}_{\partial}=\inp{w}{\tilde{w}}_{L^2(\R^{d-1})},
\end{align*}
for functions defined on $\{ x_d =0\}$, such as traces.

We introduce the following norms, for $m\in\N$ and $m'\in\R,$
\begin{align*}
  & \Normsc{u}{m,m'}
    \asymp \sum\limits_{j=0}^{m}
    \Norm{\Lsct^{m+m'-j} D_{x_d}^j u}{+},\\
  & \Normsc{u}{m}
    =\Normsc{u}{m,0}
    \asymp\sum\limits_{j=0}^{m}
    \Norm{\Lsct^{m-j} D_{x_d}^j u}{+},
\end{align*}
for $u\in\Sbarp$. One has 
\begin{align*}
  \Normsc{u}{m}
  \asymp
  \sum\limits_{|\alpha|\leq m}\tau^{m-|\alpha|}
  \Norm{D^{\alpha}u}{+},
\end{align*}
and in the case $m'\in\N$ one has
\begin{align*}
  \Normsc{u}{m,m'}
  \asymp\sum\limits_{\substack{\alpha_d\leq m\\ |\alpha|\leq m+m'}}
  \tau^{m+m'-|\alpha|}
  \Norm{D^{\alpha}u}{+},
\end{align*}
  with $\alpha=(\alpha',\alpha_d)\in\N^d$.

  \medskip
  The following argument will be used on numerous occasions: for $m\in\N,$ $m',\ell\in\R,$ with $\ell>0,$
  \begin{align*}
    \Normsc{u}{m,m'}
    \ll \Normsc{u}{m,m'+\ell},
  \end{align*}
  if $\tau $ is chosen
\suff large.

\medskip
At the boundary $\{x_d=0\}$ we define the following norms, for $m\in\N$ and $m'\in\R$,
\begin{align*}
  \normsc{\trace(u)}{m,m'}^2
  =\sum\limits_{j=0}^{m}
  \norm{\Lsct^{m+m'-j} D^j_{x_d}u\br}{\partial}^2,
  \qquad u\in\Sbarp.
\end{align*}

\subsection{Differential quadratic forms}
\label{sec: Differential quadratic forms}

\subsubsection{Differential quadratic forms in a half-space}
\begin{definition}[interior differential  quadratic form]
  \label{def: interior quadratic form}
  Let $u\in \Sbarp$. We say that
\begin{equation}
  \label{eq: quadratic form}
  Q(u)=\sum_{s=1}^N\inp{A^su}{B^su}_+, 
  \qquad A^s = \Op(a^s), \ 
  B^s= \Op(b^s), 
\end{equation}
is an interior differential quadratic form of type $(m,r)$ with smooth
coefficients, if for each $s=1,\dots N$, we have $a^s(\y)\in
\Ssc^{m,r'}$ and $b^s(\y)\in
\Ssc^{m,r''}$, with $r' + r'' = 2
r$, $\y = (x,\xi,\tau)$.
  
The principal symbol of the quadratic form $Q$ is defined as the
class of 
\begin{equation}
  \label{eq: symbol quadratic form}
  q(\y)=\sum_{s=1}^N a^{s}(\y)\ovl{b^s}(\y)
\end{equation}
in $\Ssc^{2m,2 r} / \Ssc^{2m,2 r-1}$.
\end{definition}
A result we shall use is the following microlocal G{\aa}rding inequality.
\begin{proposition}[microlocal G{\aa}rding inequality]
  \label{prop: Gaarding quadratic forms}
  Let $K$ be a compact set of $\ovl{\Rdp}$ and let 
$\U$ be an conic open set of $\ovl{\Rdp} \times \R^{d-1} \times
  \R_+$ contained in $K \times \R^{d-1} \times
  \R_+$. 
  Let also $\chi \in \Ssct^0$ be homogeneous of degree $0$, be such
  that 
  $\supp (\chi) \subset \U$.  Let $Q$ be an interior differential quadratic
  form of type $(m,r)$ with homogeneous
  principal symbol $q\in \Ssc^{2m,2 r}$ satisfying, for some $C_0>0$
  and $\tau_0>0$,
  \begin{equation*}
    \Re q(\y)\geq C_0 \lsc^{2m} \lsct^{2r},\quad 
    \text{for} \ \tau \geq \tau_0,\ \ \y = (\y',\xi_d),\
    \y'=(x,\xi',\tau)\in\U,\ \xi_d\in\R.
  \end{equation*}
  
  For $0< C_1 < C_0$ and $N \in \N$ there exist $\tau_\ast$, $C>0$,
  and $C_N>0$ such that  
  \begin{align*}
    \Re Q(\Opt(\chi) u)
    &\geq
    C_1\Normsc{\Opt(\chi) u}{m,r}^2
    -C\normsc{\trace(\Opt(\chi) u)}{m-1,r+1/2}^2
    - C_N \Normsc{u}{m,-N}^2,
  \end{align*}
  for $u\in \Sbarp$ and  $\tau\geq \tau_\ast$.
\end{proposition}
We refer to \cite[Proposition~3.5]{ML} and \cite[Theorem~6.17]{JGL-vol2} for a proof. 
A local version of the result is the following one that follows from Proposition~\ref{prop: Gaarding quadratic forms}.
\begin{proposition}[G{\aa}rding inequality]
  \label{prop: Gaarding quadratic forms local}
  Let $U_0$ be a bounded open subset
of $\ovl{\Rdp}$ and let $Q$ be an interior differential quadratic
  form of type $(m,r)$ with homogeneous
  principal symbol $q\in \Ssc^{2m,2 r}$ satisfying, for some $C_0>0$
  and $\tau_0>0$,
  \begin{equation*}
    \Re q(\y)\geq C_0 \lsc^{2m} \lsct^{2r},\quad 
    \text{for} \ \tau \geq \tau_0,\ \ \y = (\y',\xi_d),\
    \y'=(x,\xi',\tau)\in U_0 \times \R^{d-1} \times \R_+,\ \xi_d\in\R.
  \end{equation*}
  
  For $0< C_1 < C_0$ there exist $\tau_\ast$, $C>0$such that  
  \begin{align*}
    \Re Q(u)
    &\geq
    C_1\Normsc{ u}{m,r}^2
    -C\normsc{\trace(u)}{m-1,r+1/2}^2,
  \end{align*}
  for $u\in \Sbarp$ and  $\tau\geq \tau_\ast$.
\end{proposition}

\subsubsection{Boundary  differential quadratic forms}
\label{sec: Boundary  differential quadratic forms}
\begin{definition}
\label{def: boundary quadratic form}
Let $u\in  \Sbarp$. We say that
\begin{equation*}
  \Bquad(u)=\sum_{s=1}^N\inp{A^s u\br}{B^s u\br}_\partial, 
  \quad A^s = a^s(x,D,\tau),  \ B^s = b^s(x,D,\tau),  
\end{equation*}
is a boundary differential quadratic form of type $(m-1,r)$ with $\Cinf$
coefficients, if for each $s=1,\dots N$, we have $a^s(\y)\in
\Ssc^{m-1,r'}(\ovl{\Rdp} \times\R^d)$, $b^s(\y)\in
\Ssc^{m-1,r''}(\ovl{\Rdp}  \times\R^d)$ with $r'+r''=2
r$, $\y=(\y',\xi_d)$ with $\y'=(x,\xi',\tau)$. The symbol of the
boundary differential quadratic form $\Bquad$ is defined by
\begin{equation*}
  \bquad(\y',\xi_d,\tilde{\xi}_d)=\sum_{s=1}^N
  a^{s}(\y',\xi_d)\ovl{b^s}(\y',\tilde{\xi}_d). 
\end{equation*}
\end{definition}

For $\z=(z_0,\dots,z_{\ell-1})\in\C^\ell$ and $a(\y)\in
\Ssc^{\ell-1,t}$, of the form 
$a(\y',\xi_d) = \sum_{j=0}^{\ell-1} a_j(\y') \xi_d^j$ with $a_j(\y') \in \Ssct^{\ell-1+t-j}$
we set
\begin{equation}
  \label{eq: symbol z}
  \un{a}(\y',\z)=\sum_{j=0}^{\ell-1}a_j(\y')z_j.
\end{equation}
From the boundary differential quadratic form $\Bquad$ we introduce the  following
bilinear symbol $\un{\Bquad}:\C^m\times\C^m\to \C $
\begin{equation}
  \label{eq: bilinear symbol-boundary quadratic form}
  \un{\Bquad}(\y',\z,\z')
  =\sum_{s=1}^N \un{a^s}(\y',\z)\ovl{\un{b^s}}(\y',\ovl{\z}'),
  \quad \z,\z'\in\C^m.
\end{equation}

We let $\mathscr W$ be an open conic set in $\R^{d-1} \times \R^{d-1} \times \R_+$.
\begin{definition}
  Let $\Bquad$ be a boundary differential quadratic form of type $(m-1,r)$ with
  homogeneous principal symbol and 
  associated with the bilinear symbol $\un{\Bquad}(\y',\z,\z')$. We say that $\Bquad$ is
  positive definite in $\mathscr W$ if there exist $C>0$ and $R>0$ such that 
\begin{equation*}
  \Re \un{\Bquad}(\y'', x_d = 0^+, \z,\z)\geq
  C\sum_{j=0}^{m-1}  \lsct^{2(m-1-j+r)}|z_j|^2,
\end{equation*}
for $\y''=(x',\xi',\tau) \in \mathscr W$, with $|(\xi', \tau)| \geq
R$, and
$\z=(z_0,\dots,z_{m-1})\in \C^m$.
\end{definition}

We have the following G{\aa}rding inequality.
\begin{proposition}
  \label{prop: boundary form -Gaarding tangentiel}
  Let $\Bquad$ be a boundary differential quadratic form of type $(m-1,r)$,
  positive definite in $\mathscr W$, an open conic set in $\R^{d-1}
  \times \R^{d-1} \times \R_+$, with bilinear symbol
  $\un{\Bquad}(\y',\z,\z')$. Let $\chi \in \Ssct^0$ be homogeneous of
  degree 0, with $\supp(\chi\br)\subset \mathscr W$ and let $N\in \N$. Then
  there exist $\tau_\ast\geq 1$, $C>0$, $C_N >0$ such that 
  \begin{equation*}
    \Re \Bquad(\Opt(\chi) u)
    \geq C\normsc{\trace(\Opt(\chi) u)}{m-1,r}^2
    -C_N\normsc{\trace(u)}{m-1,-N}^2,
  \end{equation*}
  for $u\in \Sbarp$ and $\tau\geq\tau_\ast$.
\end{proposition}

\subsection{Symbols and operators with an additional large parameter}
\label{sec: calculus with spectral parameter}
\bigskip In this article, we shall use operators with a  symbol that depends
on an additional large parameter $\sigma$,  say $a(x,\xi,\tau,\sigma)$. They
will satisfy estimate of the form
\begin{align*}
  |\partial^{\alpha}_x\partial^{\beta}_{\xi}a(x,\xi,\tau,\sigma)|\leq
  C_{\alpha,\beta}(\tau^2+|\xi|^2+\sigma^2)^{(m-|\beta|)/2}.
\end{align*}
We observe that  if  $\tau\gtrsim \sigma$ one has
\begin{align*}
 \lsc^2 \leq  \tau^2+|\xi|^2+\sigma^2  \lesssim \lsc^2.
\end{align*}
Thus, as far as pseudo-differential calculus is concerned it is as if
$a\in \Ssc^m$ and this property will be
exploited in what follows.

Similarly if $a = a(x,\xi',\tau,\sigma)$ fulfills a tangential-type estimate of the form
\begin{align*}
  |\partial^{\alpha}_x\partial^{\beta}_{\xi'} a(x,\xi',\tau,\sigma)|\leq
  C_{\alpha,\beta}(\tau^2+|\xi'|^2+\sigma^2)^{(m-|\beta|)/2},
\end{align*}
if one has $\tau\gtrsim \sigma$ one will be able to apply techniques
adapted to symbols in $\Ssct^m$ and associated operators, like for instance the results on differential quadratic forms listed in
Section~\ref{sec: Differential quadratic forms}.

\section{Lopatinski\u{\i}-\v{S}apiro boundary conditions for an elliptic operator}
\label{sec: LS condition}
Let $P$ be an elliptic differential operator of order $2k$ on
$\Omega$, $(k\geq 1)$, with principal symbol $p(x,\omega)$ for
$(x,\omega)\in T^*\Omega$. One defines the following polynomial in
$z,$
$$\tilde{p}(x,\omega',z)=p(x,\omega'-z n_x),$$ for
$x\in\partial\Omega,$ $\omega'\in T^*_x\partial\Omega,$ $z\in\R$ and
where $n_x$ denotes the outward unit  pointing conormal vector at $x$
(see Section~\ref{sec: Geometrical setting}). Here $x$ and $\omega'$
are considered to act as parameters.  We denote by $\rho_j(x,\omega'),$ $1\leq j\leq 2k$ the complex roots of $\tilde{p}.$ One sets $$\tilde{p}^+(x,\omega',z)=\prod\limits_{\Im\rho_j(x,\omega')\geq 0}(z-\rho_j(x,\omega')).$$ Given boundary operators $B_1,\cdots, B_k$ in a neighborhood of $\partial\Omega,$ with principal symbols $b_j(x,\omega)$, $j=1,\cdots,k,$ one also sets $\tilde{b}_j(x,\omega',z)=b_j(x,\omega'-zn_x).$
\begin{definition}[\LS boundary condition]\label{def: LS}
Let $(x,\omega')\in T^*\partial\Omega$ with $\omega'\neq 0.$ One says that the \LS condition holds for $(P, B_1,\cdots, B_k)$ at $(x,\omega')$ if for any polynomial function $f(z)$ with complex coefficients, there exists $c_1,\cdots,c_k\in\C$ and a polynomial function $g(z)$ with complex coefficients such that, for all $z\in\C,$
$$f(z)=\sum\limits_{1\leq j\leq k}c_j\tilde{b_j}(x,\omega',z)+g(z)\tilde{p}^+(x,\omega',z).$$
We say that the \LS condition holds for $(P, B_1, \cdots, B_k)$ at $x\in\partial\Omega$ if it holds at $(x,\omega')$ for all $\omega'\in T^*_x\partial\Omega$ with $\omega'\neq 0.$
\end{definition}

\subsection{Some examples}

For instance the \LS condition holds in the following cases
\begin{itemize}
\item $P=-\Delta$ on $\Omega,$ with the Dirichlet boundary condition, $Bu_{|\partial\Omega}=u_{|\partial\Omega}.$
\item $P=\Delta^2$ on $\Omega,$ along  with the so-called clamped
  boundary conditions, i.e,
  $B_1u_{|\partial\Omega}=u_{|\partial\Omega}$ and
  $B_2u_{|\partial\Omega}=\partial_{\nu}u_{|\partial\Omega},$ where
  $\nu$ is the normal outward pointing unit vector to
  $\partial\Omega$; see Section~\ref{sec: Geometrical setting}.
\item $P=\Delta^2$ on $\Omega,$ along with the so-called hinged boundary conditions, i.e, $B_1u_{|\partial\Omega}=u_{|\partial\Omega}$ and  $B_2u_{|\partial\Omega}=\Delta u_{|\partial\Omega}.$
\end{itemize}

\subsection{Case of the bi-Laplace operator}
\label{sec: LS for bi-Laplace}
With $P=\Delta^2$ on $\Omega$,  along with the general boundary operators $B_1$ and $B_2$ of orders $k_1$ and $k_2$ respectively,  we give a matrix criterion of the \LS condition.  The general boundary operators $B_1$ and $B_2$ are then given by
\begin{equation*}
  B_\ell(x,D)
  = \sum\limits_{0\leq j\leq\min(3,k_\ell)}
  B_\ell^{k_\ell-j}(x,D') (i \partial_\nu) ^j,
  \qquad \ell=1,2,
\end{equation*}
with $B_\ell^{k_\ell-j}(x,D')$ differential operators acting in the
tangential variables.  
We denote by $b_1(x,\omega)$ and $b_2(x,\omega)$ the principal symbols of $B_1$ and $B_2$ respectively. For $(x,\omega')\in T^*\partial\Omega,$ we set 
\begin{equation*}
  \tilde{b}_\ell(x,\omega',z)
  =\sum\limits_{0\leq j\leq \min(3,k_\ell)}
  b_\ell^{k_\ell-j}(x,\omega')z^j,
  \qquad \ell=1,2.
\end{equation*}
We recall that the principal symbol of $P$ is given by
$p(x,\omega)=|\omega|^4_g.$ One thus has 
\begin{equation*}
  \tilde{p}(x,\omega',z)=p(x,\omega'-zn_x) =  \big( z^2 + |\omega'|_g^2\big)^2.
\end{equation*}
Therefore
$\tilde{p}(x,\omega',z)=(z-i|\omega'|_g)^2(z+i|\omega'|_g)^2$. According
to the above definition we set 
$\tilde{p}^+(x,\omega',z)=(z-i|\omega'|_g)^2.$ Thus, the
\LS condition holds at $(x,\omega')$ with $\omega'\neq 0$ if
and only if for any function $f(z)$ the complex number $i|\omega'|_g$
is a root of the polynomial function $z\mapsto
f(z)-c_1\tilde{b}_1(x,\omega',z)-c_2\tilde{b}_2(x,\omega',z)$ and its
derivative for some $c_1,c_2\in\C$. This leads to the
following determinant condition.
\begin{lemma}\label{lemma: Lopatinskii bi-laplacian op}  
  Let $P=\Delta^2$ on $\Omega$, $B_1$ and $B_2$ be two boundary
  operators. If $x \in \partial\Omega$, $\omega' \in T_x^* \partial\Omega$, with
  $\omega' \neq 0$, the \LS
   condition holds at $(x,\omega')$ if and only if
  \begin{equation}\label{cond1}
    \det
    \begin{pmatrix}
      \tilde{b}_1 &\tilde{b}_2 \\[2pt]
      \partial_z\tilde{b}_1& \partial_z\tilde{b}_2
    \end{pmatrix} (x,\omega',z=i|\omega'|_g)
    \neq 0.
\end{equation}
\end{lemma}
\begin{remark}\label{RR0}
With the determinant condition and homogeneity, we note that if the
\LS condition holds for $(P,B_1, B_2)$ at $(x,\omega')$ it also holds
in a conic neighborhood of $(x,\omega')$ by continuity. If it holds at $x\in\Omega,$ it also holds in a neighborhood of $x$.
\end{remark}

\subsection{Formulation in normal geodesic coordinates}
Near a boundary point $x\in\partial\Omega$, we shall use normal
geodesic coordinates. These coordinates are recalled at the beginning
of Section \ref{sec: Psido}. Then the principal symbols of $\Delta$ and
$\Delta^2$ are given by $\xi_d^2+r(x,\xi')$ and
$(\xi_d^2+r(x,\xi'))^2$ respectively, where $r(x,\xi')$ is the
principal symbol of a tangential differential elliptic operator
$R(x,D')$ of order 2, with
$$r(x,\xi')=\sum\limits_{1\leq i,j\leq
  d-1}g^{ij}(x)\xi'_i\xi'_j \ \ \text{and} \ \  r(x,\xi')\geq C|\xi'|^2.$$
Here $g^{ij}$ is the inverse of the metric $g_{ij}.$ Below, we shall
often write $|\xi'|^2_x=r(x,\xi')$ and we shall also write
$|\xi|^2_x=\xi_d^2+r(x,\xi'),$ for $\xi=(\xi',\xi_d).$

If $b_1(x,\xi)$ and $b_2(x,\xi)$ are the principal symbols of the
boundary operators $B_1$ and $B_2$ in the normal geodesic
coordinates then the \LS condition for $(P, B_1, B_2)$ with $P =
\Delta^2$ at $(x,\xi')$ reads
\begin{equation*}
    \det
    \begin{pmatrix}
      b_1&b_2\\[2pt]
      \partial_{\xi_d}b_1& \partial_{\xi_d}b_2
    \end{pmatrix}(x,\xi',\xi_d=i |\xi'|_x)
    \neq 0,
  \end{equation*}
  if $|\xi'|_x \neq 0$ according to Lemma~\ref{lemma: Lopatinskii bi-laplacian op}. If the \LS condition holds at some $x^0$,
  because of homogeneity, there exists $C_0>0$ such that
  \begin{equation}\label{eq: homog formul LS condition}
    \left|\det
    \begin{pmatrix}
      b_1 &b_2\\[2pt]
      \partial_{\xi_d}b_1& \partial_{\xi_d}b_2 
    \end{pmatrix}\right|(x^0,\xi',i |\xi'|_x)
    \geq C_0 |\xi'|_x^{k_1+k_2 -1},
    \qquad \xi' \in \R^{d-1}.
  \end{equation}

  \subsection{Stability of the \LS condition under perturbation}
  \label{sec: Stability LS}
To prepare for the study of how the \LS condition behaves under
conjugation with Carleman exponential weight and the addition of a
spectral parameter, we introduce some perturbations in the formulation
of the \LS condition for $(P, B_1, B_2)$ as written in \eqref{eq:
  homog formul LS condition}.
\begin{lemma}
  \label{lemma: perturbation LS}
  Let $V^0$ be a compact set of $\partial\Omega$ be such that the \LS
  condition holds for $(P, B_1, B_2)$ at every point $x$ of $V^0$. There exist $C_1>0$
  and $\eps>0$ such that
   \begin{equation}\label{eq: LS + perturbation-1}
    \left|\det
    \begin{pmatrix}
      b_1
      &b_2\\[2pt]
      \partial_{\xi_d}b_1
      & \partial_{\xi_d}b_2 
    \end{pmatrix}\right|(x,\xi' + \zeta', \xi_d = i |\xi'|_x+ \delta)
    \geq C_1 |\xi'|_x^{k_1+k_2 -1},
  \end{equation}
  for $x \in V^0$, $\xi' \in \R^{d-1}$, $\zeta' \in \C^{d-1}$, and $\delta \in \C$, if $|\zeta'| + |\delta| \leq \eps |\xi'|_x$. 
  Moreover one has
  \begin{equation}\label{eq: LS + perturbation-2}
    \left|\det
    \begin{pmatrix}
      b_1 (x,\xi' + \zeta', \xi_d = i |\xi'|_x+ \delta)
      &b_2 (x,\xi' + \zeta', \xi_d = i |\xi'|_x+ \delta)\\[4pt]
      b_1 (x,\xi' + \zeta', \xi_d = i |\xi'|_x+ \tilde{\delta})
      & b_2 (x,\xi' + \zeta', \xi_d = i |\xi'|_x+ \tilde{\delta})
    \end{pmatrix}\right|
    \geq C_1 |\delta - \tilde{\delta}|\,  |\xi'|_x^{k_1+k_2 -1},  
  \end{equation}
  for  $x \in V^0$, $\xi' \in \R^{d-1}$, $\zeta' \in \C^{d-1}$, and $\delta, \tilde{\delta}\in \C$,
  if $|\zeta'| + |\delta| + |\tilde{\delta}|\leq \eps  |\xi'|_x$. 
\end{lemma}
\begin{proof}

  From \eqref{eq: homog formul LS condition}, since $V^0$ is compact
  having the \LS condition holding at every point $x$ of $V^0$ means
  there exists $C_0>0$ such that
  \begin{equation}
    \label{eq: homog formul LS condition local}
    \left|\det
    \begin{pmatrix}
      b_1 &b_2\\[2pt]
      \partial_{\xi_d}b_1& \partial_{\xi_d}b_2 
    \end{pmatrix}\right|(x,\xi',i |\xi'|_x)
    \geq C_0 |\xi'|_x^{k_1+k_2 -1},
    \qquad x \in V^0, \xi' \in \R^{d-1}.
  \end{equation}
  The first part is a consequence of the mean value theorem,
  homogeneity and \eqref{eq: homog formul LS
  condition local} with say $C_1 = C_0 /2$. 
  
For the second part it is sufficient to assume that
$\delta \neq \tilde{\delta}$ since the result is obvious otherwise.
For $j=1,2$ one writes the Taylor formula
  \begin{align*}
    b_j (x,\xi' + \zeta', i |\xi'|_x+ \tilde{\delta})
    &= b_j (x,\xi' + \zeta', i |\xi'|_x+ \delta)
    + ( \tilde{\delta} - \delta)
    \partial_{\xi_d} b_j (x,\xi' + \zeta', i |\xi'|_x+ \delta)\\
    &\quad + ( \tilde{\delta} - \delta)^2
    \int_0^1 (1- s) \partial_{\xi_d}^2 b_j
    (x,\xi' + \zeta', i |\xi'|_x+ \delta_s )
    \, ds,
  \end{align*}
   with $\delta_s =  (1-s)\delta +s \tilde{\delta}$,
  yielding
  \begin{align*}
    \frac{1}{\tilde{\delta} - \delta}
    &\det \begin{pmatrix}
      b_1 (x,\xi' + \zeta', i |\xi'|_x+ \delta)
      &b_2 (x,\xi' + \zeta', i |\xi'|_x+ \delta)\\[4pt]
      b_1 (x,\xi' + \zeta', i |\xi'|_x+ \tilde{\delta})
      & b_2 (x,\xi' + \zeta', i |\xi'|_x+ \tilde{\delta})
    \end{pmatrix}\\
    &\qquad \quad     = \det \begin{pmatrix}
      b_1  &b_2 \\[2pt]
      \partial_{\xi_d} b_1 &\partial_{\xi_d} b_2
    \end{pmatrix}(x,\xi' + \zeta', i |\xi'|_x+ \delta)\\
    &\qquad \quad    \quad  +
     (\tilde{\delta} - \delta) \int_0^1 (1-s)
      \det \begin{pmatrix}
        b_1 (x,\xi' + \zeta', i |\xi'|_x+ \delta)
        &b_2 (x,\xi' + \zeta', i |\xi'|_x+ \delta)\\[2pt]
        \partial_{\xi_d}^2 b_1 (x,\xi' + \zeta', i |\xi'|_x+\delta_s )
        &\partial_{\xi_d}^2 b_2 (x,\xi' + \zeta', i |\xi'|_x+\delta_s )
    \end{pmatrix}
      d s.
  \end{align*}
  With homogeneity, if $|\zeta'| + |\delta| + |\tilde{\delta}| \lesssim  |\xi'|_x$ one
  finds
  \begin{align*}
    \left| \det \begin{pmatrix}
        b_1 (x,\xi' + \zeta', i |\xi'|_x+ \delta)
        &b_2 (x,\xi' + \zeta', i |\xi'|_x+ \delta)\\[2pt]
        \partial_{\xi_d}^2 b_1 (x,\xi' + \zeta', i |\xi'|_x+\delta_s )
        &\partial_{\xi_d}^2 b_2 (x,\xi' + \zeta', i |\xi'|_x+\delta_s )
    \end{pmatrix}\right| 
     \lesssim |\xi'|_x^{k_1 + k_2 -2},
  \end{align*}
  Thus with $|\delta - \tilde{\delta}|\leq \eps  |\xi'|_x$, for
  $\eps>0$ chosen \suff small, using the first part of the lemma one obtains
  the second result. 
\end{proof}

\subsection{Examples  of boundary operators yielding symmetry}
\label{sec: Examples  of boundary operators yielding symmetry}

We give some examples of pairs of boundary operators $B_1, B_2$ that
fulfill (1) the
\LS condition and (2) yield symmetry for the bi-Lalace operator $P = \Delta^2$, that is, 
\begin{align*}
  \inp{P u }{v}_{L^2(\Omega)} =  \inp{u }{P v}_{L^2(\Omega)} 
\end{align*}
for $u,v\in H^4(\Omega)$ such that $B_j u_{|\partial\Omega} = B_j
v_{|\partial\Omega} =0$, $j=1,2$.

We first recall that following Green formula
\begin{align}
\label{eq: Green formula}
  \inp{\Delta u }{v}_{L^2(\Omega)} 
  =  \inp{u }{\Delta  v}_{L^2(\Omega)} 
  + \inp{\partial_{n} u\bd}{v\bd}_{L^2(\partial\Omega)}  
  - \inp{u\bd}{\partial_{n} v\bd}_{L^2(\partial\Omega)}, 
\end{align}
which applied twice gives
$\inp{P u }{v}_{L^2(\Omega)} 
  =  \inp{u }{P  v}_{L^2(\Omega)} + T(u,v)$
with \begin{align}
  \label{eq: Green formula biLaplace}
  T(u,v) &= \inp{\partial_{n} \Delta  u\bd}{v\bd}_{L^2(\partial\Omega)}  
  - \inp{\Delta  u\bd}{\partial_{n} v\bd}_{L^2(\partial\Omega)}\notag\\
  &\quad+ \inp{\partial_{n} u\bd}{\Delta  v\bd}_{L^2(\partial\Omega)}  
  - \inp{u\bd}{\partial_{n} \Delta  v\bd}_{L^2(\partial\Omega)}.
\end{align}
Using normal geodesic coordinates in a \nhd of the whole boundary
$\partial\Omega$ allows one to write $\Delta = \partial_n^2 + \Delta'$
where $\Delta'$ is the resulting Laplace operator on the boundary,
that is, associated with the trace of the metric on
$\partial\Omega$. Since $\Delta'$ is selfadjoint on $\partial\Omega$ 
this allows one to write formula \eqref{eq: Green formula biLaplace}
in the alternative forms
\begin{align}
  \label{eq: Green formula biLaplace-bis}
  T(u,v) &= \inp{\partial_{n}^3 u\bd}{v\bd}_{L^2(\partial\Omega)}  
  - \inp{(\partial_{n}^2 + 2 \Delta') u\bd}{\partial_{n}v\bd}_{L^2(\partial\Omega)}
\notag\\
  &\quad 
    + \inp{\partial_{n} u\bd}{ (\partial_{n}^2 + 2 \Delta')v\bd}_{L^2(\partial\Omega)}  
  - \inp{u\bd}{\partial_{n}^3  v\bd}_{L^2(\partial\Omega)},
\end{align}
or 
\begin{align}
  \label{eq: Green formula biLaplace-ter}
  T(u,v) &= 
  \inp{(\partial_{n}^3 
  + 2 \Delta'\partial_{n})  u\bd}{v\bd}_{L^2(\partial\Omega)}  
  - \inp{\partial_{n}^2  u\bd}{\partial_{n}v\bd}_{L^2(\partial\Omega)}
  \notag\\
  &\quad 
   + \inp{\partial_{n} u\bd}{ \partial_{n}^2 v\bd}_{L^2(\partial\Omega)}  
  - \inp{u\bd}{(\partial_{n}^3 
           + 2 \Delta'\partial_{n})  v\bd}_{L^2(\partial\Omega)}.
\end{align}

\bigskip
We start our list of examples with the most basics ones.
\begin{example}[Hinged boundary conditions]
\label{ex: hinged boundary conditions}
This type of conditions refers to $B_1 u\bd = u\bd$ and $B_2 u\bd
= \Delta u\bd$. With \eqref{eq: Green formula biLaplace} one finds
$T(u,v) =0$ in the case of homogeneous conditions, hence symmetry.

Note that the hinged boundary conditions are equivalent to having 
$B_1 u\bd = u\bd$ and $B_2 u\bd
=  \partial_n^2 u\bd$. 
With the notation of Section~\ref{sec: LS condition} this gives
$\tilde{b}_1(x,\omega',z)=1$ and $\tilde{b}_2(x,\omega',z)=(-iz)^2 =
-z^2$. It follows that 
\begin{equation*}
    \det
    \begin{pmatrix}
      \tilde{b}_1 &\tilde{b}_2 \\[2pt]
      \partial_z\tilde{b}_1& \partial_z\tilde{b}_2
    \end{pmatrix} (x,\omega',z=i|\omega'|_g)
    =
     \det
    \begin{pmatrix}
     1 &|\omega'|^2_g \\[2pt]
     0&-2i|\omega'|_g
    \end{pmatrix} =-2i|\omega'|_g \neq 0,
\end{equation*}
if $\omega'\neq 0$ and thus the \LS condition holds by Lemma~\ref{lemma: Lopatinskii bi-laplacian op}.

With the hinged boundary conditions observe that the bi-Laplace operator is precisely
the square of the Dirichlet-Laplace operator. This makes its analysis
quite simple and this explains why this type of boundary condition is
often chosen in the mathematical literature. Observe that symmetry is
then obvious without invoking the above formulae.
\end{example}
\begin{example}[Clamped  boundary conditions]
This type of conditions refers to $B_1 u\bd = u\bd$ and $B_2 u\bd
= \partial_n u\bd$. With \eqref{eq: Green formula biLaplace-bis}  one finds
$T(u,v) =0$ in the case of homogeneous conditions, hence symmetry.
With the notation of Section~\ref{sec: LS condition} this gives
$\tilde{b}_1(x,\omega',z)=1$ and $\tilde{b}_2(x,\omega',z)=-iz$. It follows that 
\begin{equation*}
    \det
    \begin{pmatrix}
      \tilde{b}_1 &\tilde{b}_2 \\[2pt]
      \partial_z\tilde{b}_1& \partial_z\tilde{b}_2
    \end{pmatrix} (x,\omega',z=i|\omega'|_g)
    =
     \det
    \begin{pmatrix}
     1 &|\omega'|_g \\[2pt]
     0&-i
    \end{pmatrix} =-i \neq 0.
\end{equation*}
Thus the \LS condition holds by Lemma~\ref{lemma: Lopatinskii bi-laplacian op}.

Note that with the clamped boundary conditions the bi-Laplace operator
cannot be seen as the square 
of the Laplace operator with some well chosen boundary condition as
opposed to the case of the hinged boundary conditions displayed
above. 
\end{example}
\begin{examples}[More examples]
$\phantom{-}$
\begin{enumerate}
\item Take  $B_1 u\bd =\partial_n  u\bd$ and $B_2 u\bd
= \partial_n \Delta u\bd$. With these boundary conditions the bi-Laplace operator is precisely
the square of the Neumann-Laplace operator. The symmetry property is
immediate and so is the \LS condition.

\item Take $B_1 u\bd = (\partial_n^2 + 2 \Delta') u\bd$ and $B_2 u\bd
= \partial_n^3 u\bd$. With \eqref{eq: Green formula biLaplace-bis} one finds
$T(u,v) =0$ in the case of homogeneous conditions, hence symmetry.

We have $\tilde{b}_1(x,\omega',z)=-z^2 - 2 |\omega'|_g^2$ and
$\tilde{b}_2(x,\omega',z)=iz^3$ and
\begin{equation*}
    \det
    \begin{pmatrix}
      \tilde{b}_1 &\tilde{b}_2 \\[2pt]
      \partial_z\tilde{b}_1& \partial_z\tilde{b}_2
    \end{pmatrix} (x,\omega',z=i|\omega'|_g)
    =
     \det
    \begin{pmatrix}
     -|\omega'|_g^2 &|\omega'|^3_g \\[2pt]
     -2 i |\omega'|_g&-3i|\omega'|^2_g
    \end{pmatrix} 
    =5 i |\omega'|^4_g \neq 0,
\end{equation*}
if $\omega'\neq 0$ and thus the \LS condition holds by Lemma~\ref{lemma: Lopatinskii bi-laplacian op}.
\item Take $B_1 u\bd = \partial_n u\bd$ and $B_2 u\bd
= (\partial_n^3+ A')  u\bd$, with $A'$ a symmetric differential
operator of order less than or equal to three on
$\partial \Omega$, with homogeneous principal symbol $a'(x,\omega')$ such that
$a'(x,\omega') \neq 2 |\omega'|_g ^3$ for $\omega' \neq 0$, that
is, $a'(x,\omega')\neq 2$ for $|\omega'|_g=1$.

With \eqref{eq: Green formula biLaplace-bis} one
finds
\begin{align*}
T(u,v)  =  \inp{-A' u\bd}{v\bd}_{L^2(\partial\Omega)}  
  + \inp{u\bd}{A' v\bd}_{L^2(\partial\Omega)} =0,
\end{align*}
in the case of homogeneous conditions, hence symmetry for $P$.

We have $\tilde{b}_1(x,\omega',z)=-i z$ and
$\tilde{b}_2(x,\omega',z)=iz^3 + a'(x,\omega')$ with $a'$ the
principal symbol of $A'$.
\begin{equation*}
    \det
    \begin{pmatrix}
      \tilde{b}_1 &\tilde{b}_2 \\[2pt]
      \partial_z\tilde{b}_1& \partial_z\tilde{b}_2
    \end{pmatrix} (x,\omega',z=i|\omega'|_g)
    =
     \det
    \begin{pmatrix}
     |\omega'|_g &|\omega'|^3_g + a'(x,\omega')\\[2pt]
    - i &-3i|\omega'|^2_g
    \end{pmatrix} 
    = i \big( a'(x,\omega') -2  |\omega'|^3_g\big)\neq 0,
\end{equation*}
if $\omega'\neq 0$ since $a'(x,\omega') \neq 2 |\omega'|^3_g$  by
assumption implying
that the \LS condition holds by Lemma~\ref{lemma: Lopatinskii bi-laplacian op}.
\item 
 Take  $B_1 u\bd =u\bd$ and $B_2 u\bd
=(\partial_n^2+ A' \partial_n) u\bd$  with $A'$ a symmetric  differential
operator of order less than or equal to one on
$\partial \Omega$, with homogeneous principal symbol $a'(x,\omega')$ such that
$a'(x,\omega') \neq - 2 |\omega'|_g$ for $\omega' \neq 0$, that
is, $a'(x,\omega')\neq -2$ for $|\omega'|_g=1$.
This is  a refinement of the case of hinged boundary conditions given
in Example~\ref{ex: hinged boundary conditions} above.

With \eqref{eq: Green formula biLaplace-bis} one
finds
\begin{align*}
  T(u,v)  
  =  \inp{A' \partial_n u\bd}{\partial_n v\bd}_{L^2(\partial\Omega)} 
  + \inp{\partial_n u\bd}{- A' \partial_n v\bd}_{L^2(\partial\Omega)}  
  =0,
\end{align*}
in the case of homogeneous conditions, hence symmetry for $P$.

We have $\tilde{b}_1(x,\omega',z)=1$ and
$\tilde{b}_2(x,\omega',z)=-z^2 -i z a'(x,\omega')$ with $a'$ the
principal symbol of $A'$.
\begin{equation*}
    \det
    \begin{pmatrix}
      \tilde{b}_1 &\tilde{b}_2 \\[2pt]
      \partial_z\tilde{b}_1& \partial_z\tilde{b}_2
    \end{pmatrix} (x,\omega',z=i|\omega'|_g)
    =
     \det
    \begin{pmatrix}
     1&|\omega'|^2_g + |\omega'|_g   a'(x,\omega')\\[2pt]
    0 &-2i|\omega'|_g - i a'(x,\omega')
    \end{pmatrix} 
    = - i \big( a'(x,\omega') + 2  |\omega'|_g\big)\neq 0,
\end{equation*}
if $\omega'\neq 0$ since $a'(x,\omega') \neq - 2 |\omega'|_g$ by
assumption implying
that the \LS condition holds by Lemma~\ref{lemma: Lopatinskii bi-laplacian op}.
\item Take $B_1 u\bd = (\partial_n^2  + A' \partial_n) u\bd$ and $B_2 u\bd
= (\partial_n^3+ 2 \partial_n\Delta')  u\bd$,  with $A'$ a symmetric  differential
operator of order less than or equal to one on
$\partial \Omega$, with homogeneous principal symbol $a'(x,\omega')$ such that
$2 a'(x,\omega') \neq - 3 |\omega'|_g $ for $\omega' \neq 0$, that
is, $a'(x,\omega')\neq -3/2$ for $|\omega'|_g=1$.
With \eqref{eq: Green formula biLaplace-ter} one
finds
\begin{align*}
  T(u,v)  
  =  \inp{A' \partial_n u\bd}{\partial_n v\bd}_{L^2(\partial\Omega)} 
  + \inp{\partial_n u\bd}{- A' \partial_n v\bd}_{L^2(\partial\Omega)}  
  =0,
\end{align*}
in the case of homogeneous conditions, hence symmetry for $P$.

We have $\tilde{b}_1(x,\omega',z)=-z^2 -i z a'(x,\omega')$ and
$\tilde{b}_2(x,\omega',z)=iz^3 + 2 i z |\omega'|_g^2$ and
\begin{align*}
    \det
    \begin{pmatrix}
      \tilde{b}_1 &\tilde{b}_2 \\[2pt]
      \partial_z\tilde{b}_1& \partial_z\tilde{b}_2
    \end{pmatrix} (x,\omega',z=i|\omega'|_g)
    &=
     \det
    \begin{pmatrix}
     |\omega'|^2_g + |\omega'|_g   a'(x,\omega')&-|\omega'|^3_g \\[2pt]
     -2i|\omega'|_g - i a'(x,\omega')&-i|\omega'|^2_g
    \end{pmatrix} \\
    &=- i |\omega'|^3_g \big( 2 a'(x,\omega') + 3 |\omega'|_g\big) \neq 0,
\end{align*}
if $\omega'\neq 0$ since $2 a'(x,\omega') + 3 |\omega'|_g  \neq 0$ by
assumption implying that  the \LS condition holds by Lemma~\ref{lemma: Lopatinskii bi-laplacian op}.
\end{enumerate}
\end{examples}

\section{Lopatinski\u{\i}-\v{S}apiro condition for the conjugated
  bi-Laplacian with spectral parameter}
\label{sec: LS condition for conjugated bilaplace}

Set $\Psig=\Delta^2-\sigma^4$ with $\sigma\in[0,+\infty) $ and denote by $\Pconj=e^{\tau\varphi}\Psig e^{-\tau\varphi}$ the conjugate operator of $\Psig$ with $\tau\geq 0$ a large parameter and $\varphi\in\Cinf(\R^d, \R)$. We shall refer to $\varphi$ as the weight function. The principal symbol of $\Psig$ in normal geodesic coordinates is given by
$$\psig(x,\xi)=(\xi_d^2+r(x,\xi'))^2-\sigma^4.$$ Observe that $e^{\tau\varphi}D_je^{-\tau\varphi}=D_j+
i\tau\partial_j\varphi\in\Dsc^1.$ So, after conjugation, the principal symbol becomes
\begin{align*}
  \pconj(x,\xi,\tau) &= \psig(x,\xi + i \tau d_x \varphi)\\
  &=\left((\xi_d+i\tau\partial_d\varphi)^2
    +r(x,\xi'+i\tau d_{x'}\varphi)\right)^2
    -\sigma^4\\
  &= \left((\xi_d+i\tau\partial_d\varphi)^2
    +r(x,\xi'+i\tau d_{x'}\varphi)-\sigma^2\right)\left((\xi_d+i\tau\partial_d\varphi)^2+r(x,\xi'+i\tau d_{x'}\varphi)+\sigma^2\right)
\end{align*}
We write
$\pconj(x,\xi,\tau)=\qconj^1(x,\xi,\tau)\qconj^2(x,\xi,\tau)$
with
\begin{align*}
  \qconj^j(x,\xi,\tau)
  =(\xi_d+i\tau\partial_d\varphi)^2+r(x,\xi'+i\tau
  d_{x'}\varphi)+(-1)^j\sigma^2,
  \ \
  j=1,2.
\end{align*} 

We consider  two boundary operators
$B_1$ and $B_2$ of order $k_1$ and $k_2$ with $b_j(x,\xi)$ for
principal symbol, $j=1,2$. The associated conjugated operators 
\begin{align*}
  B_{j,\varphi} = e^{\tau \varphi} B_j e^{-\tau \varphi},
\end{align*}
have respective principal symbols 
\begin{align*}
  b_{j,\varphi} (x,\xi, \tau) = b_j(x,\xi+i\tau d\varphi),\quad j=1,2.
\end{align*}

We assume that the \LS condition holds for $(P_0, B_1, B_2)$ as in
Definition \ref{def: LS} for any point $(x,\omega')\in
T^*_x\partial\Omega.$ We wish to know if the \LS condition can hold
for $(\Psig , B_1, B_2,\varphi),$ as given by the following
definition (in local coordinates for simplicity).
\begin{definition}\label{def: LS after conjugation}
  Let
  $(x,\xi',\tau,\sigma)\in
  \partial\Omega\times\R^{d-1}\times[0,+\infty) \times
   [0,+\infty)$ with
  $(\xi',\tau,\sigma)\neq0$. One says that the \LS condition holds for
  $(\Psig ,B_1, B_2, \varphi)$ at $(x,\xi',\tau,\sigma)$ if for
  any polynomial function $f(\xi_d)$ with complex coefficients there
  exist $c_1, c_2\in\C$ and a polynomial function
  $\ell(\xi_d)$ with complex coefficients such that, for all
  $\xi_d\in\C$
  \begin{equation*}
    f(\xi_d)=c_1 b_{1,\varphi}(x,\xi',\xi_d,\tau)
    +c_2b_{2,\varphi}(x,\xi',\xi_d,\tau)
    +\ell(\xi_d) \pconj^+(x,\xi',\xi_d,\tau),
\end{equation*}
with
\begin{equation*}
 \pconj^+ (x,\xi',\xi_d,\tau)
  =\prod\limits_{\Im\rho_j(\xi',\tau,\sigma)\geq
    0}(\xi_d-\rho_j(\xi',\tau,\sigma)),
\end{equation*}
where $\rho_j(x,\xi',\tau,\sigma),$ $j=1,\cdots,4,$ denote the complex
roots of $\pconj(x,\xi',\xi_d,\tau)$ viewed as a polynomial in
$\xi_d$.
%
\end{definition}
In what follows, we shall assume that $\partial_d\varphi>0.$
Locally, one has $\partial_d\varphi\geq C_1>0$, for some $C_1>0$.

\subsection{Discussion on the Lopatinski\u{\i}-\v{S}apiro condition
  according to the position of the roots}
\label{sec: discussion LS root positions}
With the assumption that $\partial_d\varphi>0,$ for any point
$(x,\xi',\tau,\sigma)$ at most two roots lie in the upper complex closed half-plane (this is explained below). We then enumerate the following cases.
\begin{enumerate}
\item[•] Case 1 : No root lying in the upper complex closed
  half-plane, then $\pconj^+ (x,\xi',\xi_d,\tau)=1$ and the
  \LS condition of Definition~\ref{def: LS after conjugation} holds trivially at $(x,\xi',\tau,\sigma)$.
\item[•] Case  2 : One root lying in the upper complex closed
  half-plane. Let us denote by  $\rho^+$ that root, then
  $\pconj^+ (x,\xi',\xi_d,\tau)=\xi_d-\rho^+.$ With
  Definition~\ref{def: LS after conjugation}, for any choice of $f$, the polynomial
  function $\xi_d\mapsto
  f(\xi_d)-c_1b_{1,\varphi}(x,\xi',\xi_d,\tau)-c_2b_{2,\varphi}(x,\xi',\xi_d,\tau)$
  admits $\rho^+$ as a root for $c_1,c_2\in\C$ well
  chosen. Hence, the \LS condition holds  at
  $(x,\xi',\tau,\sigma)$ if and only
  if
  \begin{align*}
    b_{1,\varphi}(x,\xi',\xi_d=\rho^+,\tau) \neq 0
    \ \ \text{or} \ \
    b_{2,\varphi}(x,\xi',\xi_d=\rho^+,\tau)\neq 0.
  \end{align*}
  Note that it then suffices to have
  \begin{equation*}
    \det
    \begin{pmatrix}
      b_{1,\varphi} &b_{2,\varphi}\\[2pt]
      \partial_{\xi_d}b_{1,\varphi}& \partial_{\xi_d}b_{2,\varphi} 
    \end{pmatrix}(x,\xi',\xi_d=\rho^+,\tau) \neq 0.
\end{equation*}
\item[•] Case 3 : Two different roots lying in the upper complex
  closed half-plane. Let denote by $\rho^+_1$ and $\rho^+_2$ these
  roots. One has
  $\pconj^+ (x,\xi',\xi_d,\tau)=(\xi_d-\rho^+_1)\xi_d-\rho^+_2).$
  The \LS condition holds at $(x,\xi',\tau,\sigma)$ if and only
  if the complex numbers $\rho^+_1$ and $\rho^+_2$ are the roots of
  the polynomial function
  $\xi_d\mapsto
  f(\xi_d)-c_1b_{1,\varphi}(x,\xi',\xi_d,\tau)-c_2b_{2,\varphi}(x,\xi',\xi_d,\tau),$
  for $c_1,c_2$ well chosen. This reads
\begin{equation*}
\begin{cases}
f(\rho^+_1)= c_1b_{1,\varphi}(x,\xi',\xi_d=\rho^+_1,\tau)+c_2b_{2,\varphi}(x,\xi',\xi_d=\rho^+_1,\tau),\\
f(\rho^+_2)= c_1b_{1,\varphi}(x,\xi',\xi_d=\rho^+_2,\tau)+c_2b_{2,\varphi}(x,\xi',\xi_d=\rho^+_2,\tau).
\end{cases}
\end{equation*}
Being able to solve this system in $c_1$ and $c_2$ for any $f$ is equivalent to having
\begin{equation}\label{kan1}
  \det 
\begin{pmatrix}
b_{1,\varphi}(x,\xi',\xi_d=\rho^+_1,\tau) & b_{2,\varphi}(x,\xi',\xi_d=\rho^+_1,\tau) \\[4pt]
b_{1,\varphi}(x,\xi',\xi_d=\rho^+_2,\tau) & b_{2,\varphi}(x,\xi',\xi_d=\rho^+_2,\tau)
\end{pmatrix}\neq 0.
\end{equation}
 \item[•] Case 4 : A double root lying in the upper complex closed
  half-plane. Denote by $\rho^+$ this root; one has
  $\pconj^+(x,\xi',\xi_d,\tau)=(\xi_d-\rho^+)^2.$ The
  \LS condition holds at  at $(x,\xi',\tau,\sigma)$ if and only if for any choice of $f$, the
  complex number $\rho^+$ is a double root of the polynomial function
  $\xi_d\mapsto f(\xi_d)-c_1
  b_{1,\varphi}(x,\xi',\xi,\tau)-c_2b_{2,\varphi}(x,\xi',\xi,\tau)$
  for some $c_1,c_2\in\C.$ Thus having the
  \LS condition is equivalent of having the following
  determinant condition,
  \begin{equation}\label{kkan1}
    \det 
\begin{pmatrix}
b_{1,\varphi}(x,\xi',\xi_d=\rho^+,\tau) & b_{2,\varphi}(x,\xi',\xi_d=\rho^+,\tau) \\[4pt]
\partial_{\xi_d}b_{1,\varphi}(x,\xi',\xi_d=\rho^+,\tau) & \partial_{\xi_d}b_{2,\varphi}(x,\xi',\xi_d=\rho^+,\tau)
\end{pmatrix}\neq 0.
\end{equation}
\end{enumerate}

Observe that case~4 can only occur if $\sigma =0$ (then one has
$(\xi',\tau) \neq (0,0)$). If $\sigma > 0$ then only cases 1, 2, and 3
are possible. This is precisely stated in Lemma~\ref{lemma: no real
  double root}. This will be an important point in what follows.

We now state the following important proposition.
\begin{proposition}
  \label{prop: LS after conjugation}
  Let $x^0\in\partial\Omega$. Assume that the \LS condition holds for
  $(P_0, B_1, B_2)$ at $x^0$ and thus in a compact neighborhood $V^0$
  of $x^0$ (by Remark~\ref{RR0}). Assume also that
  $\partial_d\varphi\geq C_1> 0$ in $V^0$. There exist $\mu_0>0$ and
  $\mu_1>0$ such that if
  $(x,\xi',\tau,\sigma)\in V^0\times\R^{d-1}\times[0,+\infty) \times[0,+\infty) $ with
  $(\xi',\tau,\sigma)\neq (0, 0,0)$,
  \begin{align*}
    |d_{x'}\varphi(x)|\leq \mu_0\partial_d\varphi(x)
    \ \ \text{and} \ \
    \sigma \leq \mu_1\tau \partial_d\varphi(x),
  \end{align*}
    then the \LS condition holds for $(\Psig , B_1, B_2, \varphi)$ at
  $(x,\xi',\tau,\sigma).$
\end{proposition}

The proof of Proposition~\ref{prop: LS after conjugation} is given below. We first need to
analyze the configuration of the roots of the symbol
$\pconj$ starting with each factor $\qconj^j$, $j=1,2$.

\subsection{Root configuration for each factor}
\label{sec: Root configuration for each factor}
We consider either factors $\xi_d\mapsto
\qconj^j(x,\xi',\xi_d,\tau)$. We recall that 
\begin{align*}
  \qconj^j(x,\xi,\tau)
  =(\xi_d+i\tau\partial_d\varphi)^2
  +r(x,\xi'+i\tau d_{x'}\varphi) + (-1)^j \sigma^2,
  \ \ j=1,2.
\end{align*}

First, we  consider the case  $r(x,\xi'+i\tau
d_{x'}\varphi) + (-1)^j \sigma^2 \in \R^-$, that is, equal to $-
\beta^2$ with $\beta\in \R$ . 
Then, the roots of $\xi_d\mapsto
\qconj^j(x,\xi',\xi_d,\tau)$ are given by
\begin{align*}
  -i\tau\partial_d\varphi + \beta
  \quad \text{and} \quad
  -i\tau\partial_d\varphi - \beta.
\end{align*}
Both lie in the lower complex open half-plane.

Second, we consider the case $r(x,\xi'+i\tau
d_{x'}\varphi) + (-1)^j \sigma^2 \in\C\setminus
\R^-$. There exists a unique $\alpha_j \in \C$ such
that $\Re \alpha_j >0$ and 
\begin{align}\label{Ee1}\notag
  \alpha_j^2 &= r(x,\xi'+i\tau d_{x'}\varphi) + (-1)^j \sigma^2\\
             &= r(x,\xi') - \tau^2 r(x, d_{x'}\varphi) + (-1)^j \sigma^2
               +  i 2 \tau \tilde{r}(x,\xi',d_{x'}\varphi)^2,
\end{align}
where $\tilde{r}(x,.,.)$ denotes the symmetric bilinear form associated with the quadratic form $r(x,.)$.
Then, the two roots of $\xi_d\mapsto
\qconj^j(x,\xi',\xi_d,\tau)$ are given by
\begin{align}
  \label{eq: form of the roots of q}
  \pi_{j,1}= -i\tau\partial_d\varphi - i \alpha_j
  \quad \text{and} \quad
  \pi_{j,2}=-i\tau\partial_d\varphi + i \alpha_j.
\end{align}
One has $\Im \pi_{j,1} <0$ since $\partial_d\varphi \geq C_1>0$. 
With $\Im \pi_{j,2} = -\tau\partial_d\varphi + \Re \alpha_j$ one sees that
the sign of $\Im \pi_{j,2}$ may change.
The following lemma gives an algebraic characterization of the sign of  $\Im \pi_{j,2}$.
\begin{lemma}
  \label{lemma: caracterisation Im pi2 <0}
  Assume that $\partial_d \varphi>0$. Having $\Im \pi_{j,2}<0$ is
  equivalent to having
  \begin{equation*}
    (\partial_d\varphi)^2r(x,\xi')+\tilde{r}(x,\xi',d_{x'}\varphi)^2
    <\tau^2(\partial_d\varphi)^2|d_x\varphi|_x^2+(-1)^{j+1}\sigma^2(\partial_d\varphi)^2.
\end{equation*}
\end{lemma}
\begin{proof}
  From~\ref{eq: form of the roots of q} one has $\Im \pi_{j,2} <0$ if and only if
$\Re \alpha_j < \tau \partial_d \varphi = | \tau \partial_d\varphi|$, that is, if
and only if 
\begin{align*}
  4 (\tau \partial_d\varphi)^2 \Re \alpha_j^2-4 (\tau \partial_d\varphi)^4+(\Im
  \alpha_j^2)^2 < 0,
\end{align*}
by Lemma~\ref{lem1} below. With \eqref{Ee1} this gives the result. 
\end{proof}
\begin{lemma}
  \label{lem1}
Let $z\in\C$ such that $m=z^2.$ For $x_0\in\R$ such
that $x_0\neq 0$, we have 
\begin{align*}
  |\Re z|\lesseqqgtr |x_0|
  \qquad \Longleftrightarrow
  \qquad 4x_0^2 \Re m-4x_0^4+(\Im m)^2
  \lesseqqgtr0.
  \end{align*}
\end{lemma}
\begin{proof}
Let $z=x+iy\in\C.$ On the one hand we have
$z^2=x^2-y^2+2ixy=m$ and $\Re m=x^2-y^2,$  $\Im m=2xy.$ On the other
hand we have 
\begin{align*}
  4x_0^2\Re m-4x_0^4+(\Im m)^2
  &=4x_0^2(x^2-y^2)-4x_0^4+4x^2y^2\\
  &=4(x_0^2+y^2)(x^2-x_0^2),
\end{align*}
thus with the same sign as $(x^2-x_0^2)$.
Since $|\Re z|\lesseqqgtr |x_0| \ \Leftrightarrow \ 
x^2-x_0^2 \lesseqqgtr 0$ the conclusion follows.
\end{proof}

With the following two lemmata we now connect the sign of $\Im \pi_{j,2}$ with the low frequency
regime $|\xi'| \lesssim \tau$.
\begin{lemma}
  \label{lemma: Im root < 0 if tau large}
  Assume there exists $K_0 >0$ such that $|d_{x'} \varphi| \leq K_0 |\partial_d
  \varphi|$. Then, there exists $C_{K_0}>0$ such that $\Im \pi_{j,2} <0$
    if $C_{K_0} |\xi'|  + \sigma \leq \tau \partial_d \varphi$, $j=0,1$.  
\end{lemma}
\begin{proof}
  With Lemma~\ref{lemma: caracterisation Im pi2 <0}  having $\Im
  \pi_{j,2} <0$ reads
\begin{equation}
  \label{eq1: proof  Im root < 0 if tau large}
(\partial_d\varphi)^2r(x,\xi')+\tilde{r}(x,\xi',d_{x'}\varphi)^2
<\tau^2(\partial_d\varphi)^2|d_x\varphi|_x^2+(-1)^{j+1}\sigma^2(\partial_d\varphi)^2.
\end{equation}
On the one hand, since $|d_{x'} \varphi| \leq K_0 |\partial_d
  \varphi|$ one has 
\begin{equation*}
  (\partial_d\varphi)^2r(x,\xi')+\tilde{r}(x,\xi',d_{x'}\varphi)^2
  \leq K (\partial_d\varphi)^2 |\xi'|^2,
\end{equation*}
for some $K>0$ that depends on $K_0$, using that $|\xi'|_x \eqsim
|\xi'|$.
On the other hand one has
\begin{align*}
  \tau^2(\partial_d\varphi)^2|d_x\varphi|_x^2+(-1)^{j+1}\sigma^2(\partial_d\varphi)^2
  \geq  \tau^2 (\partial_d\varphi)^4 - \sigma^2(\partial_d\varphi)^2.
\end{align*}
Thus \eqref{eq1: proof  Im root < 0 if tau large} holds if one has
\begin{align*}
  \tau^2 (\partial_d\varphi)^4
  - \sigma^2(\partial_d\varphi)^2
  \geq K (\partial_d\varphi)^2 |\xi'|^2,
\end{align*}
that is, $\tau^2 (\partial_d\varphi)^2 \geq K |\xi'|^2
  + \sigma^2$.
\end{proof}
\begin{lemma}
  \label{lemma: low frequency caracterisation}
  Let $W$ be a bounded open set of $\R^d$ and $x^0 \in W$. Assume that
  $\partial_d \varphi>0$ in $\ovl{W}$ and let $\k_0>0$. Then, there exists $C>0$ such
  that
  \begin{align*}
  |\xi'| \leq C \tau \quad
    \text{if} \ \Im \pi_{j,2} (x,\xi', \tau,\sigma) <0 
    \ \text{and}\ \k_0 \sigma \leq \tau, \quad x\in W.
  \end{align*}  
\end{lemma}
\begin{proof}
  With Lemma~\ref{lemma: caracterisation Im pi2 <0}  having $\Im
  \pi_{j,2} <0$ reads
\begin{equation*}
(\partial_d\varphi)^2r(x,\xi')+\tilde{r}(x,\xi',d_{x'}\varphi)^2
<\tau^2(\partial_d\varphi)^2|d_x\varphi|_x^2+(-1)^{j+1}\sigma^2(\partial_d\varphi)^2.
\end{equation*} 
In particular,
this implies 
\begin{equation*}
  r(x,\xi') <\tau^2 |d_x\varphi|_x^2+(-1)^{j+1}\sigma^2
  \leq (\sup_W  |d_x\varphi|_x^2 + 1/\k_0^2) \tau^2.
\end{equation*}
The result follows since $|\xi'| \asymp r(x,\xi')$.
\end{proof}

As mentioned in Section~\ref{sec: discussion LS root positions}, we
have the following result.
\begin{lemma}
  \label{lemma: no real double root}
  Assume that $\sigma >0$. Then, 
  $\pi_{1,2}\neq \pi_{2,2}$.  Moreover, the roots  $\pi_{1,2} $ and $\pi_{2,2}$ cannot be both real.
\end{lemma}
\begin{proof}
  With the forms of the roots given in \eqref{eq: form of the roots of
    q} if $\pi_{1,2}=\pi_{2,2}$ then $\alpha_1 = \alpha_2$, thus
  $\alpha_1^2 = \alpha_2^2$ implying $\sigma^2=0$. 

  Assume now that $\pi_{1,2} \in \R$ and $\pi_{2,2}\in \R$, that is,
  $\Im \pi_{1,2} = \Im \pi_{2,2}=0$. This reads $\Re \alpha_j = \tau
  \partial_d \varphi$, giving 
  $|\Re \alpha_j| = |\partial_d \varphi|$, for $j=1$ and $2$. With Lemma~\ref{lem1} one
  has 
  \begin{align*}
  4 (\tau \partial_d\varphi)^2 \Re \alpha_j^2-4 (\tau \partial_d\varphi)^4+(\Im
  \alpha_j^2)^2 =0,\qquad j=1,2.
\end{align*}
  Observing that $\Im \alpha_1^2 = \Im \alpha_2^2$ one thus obtains
  $\Re \alpha_1^2 = \Re \alpha_2^2$,
and the conclusion follows as for the first part. 
\end{proof}

\subsection{Proof of Proposition \ref{prop: LS after conjugation}}\label{rf1}

Here, according to the statement of Proposition \ref{prop: LS after
  conjugation} we consider
\begin{align*}
  |d_{x'}\varphi|\leq \mu_0 \partial_d \varphi
  \ \ \text{and} \ \ 
 \sigma \leq \mu_1 \tau  \partial_d \varphi.
\end{align*}
First, we choose $0 < \mu_0 \leq 1$ and $0< \mu_1 \leq  1/2$. Below both may be chosen much
smaller. According to Lemma~\ref{lemma: Im root < 0 if tau large},
with $K_0=1$ therein, 
for some $C_2 = 2 C_{K_0}>0$ if one has $C_2|\xi'| \leq \tau \partial_d \varphi$
then all four roots of $\xi_d \mapsto \pconj(x,\xi',\xi_d
,\tau)$ lie in the lower complex open half-plane. If so, 
we  face Case~1 as in the discussion of Section~\ref{sec: discussion LS root
  positions} and the \LS condition holds.  
To carry on with the proof of  Proposition~\ref{prop: LS after conjugation} we
now only have to consider having 
\begin{align}
  \label{eq: condition tau small}
  \tau \partial_d \varphi \leq C_2  |\xi'|.
\end{align}
Our proof of Proposition \ref{prop: LS after conjugation} relies on
the following lemma.
\begin{lemma}
  \label{lemma: small perturbation}
  There exists $C_3>0$ such that, for $j=1$ or $2$, for $0<\mu_0\leq
  1$,  $0< \mu_1\leq 1/2$, and for all
  $(x, \xi', \tau, \sigma) \in
  \ovl{V^0}\times\R^{d-1}\times[0,+\infty) \times[0,+\infty) ,$
  one has 
\begin{equation*}
  |d_{x'}\varphi|\leq \mu_0\partial_d\varphi,
  \ \ \sigma\leq  \mu_1\tau \partial_d \varphi
\ \   \text{and} \  \
\Im \pi_{j,2}\geq 0  
\ \ \Longrightarrow \ \ 
\big| \alpha_j-|\xi'|_x\big|+\tau|d_{x'}\varphi|
\leq  |\xi'|_x C_3 (\mu_0+\mu_1^2).
\end{equation*}
\end{lemma}
\begin{proof}
  With \eqref{eq: condition tau small} one has
  \begin{align}
    \label{eq: est 1 small perturbation}
    \tau|d_{x'}\varphi| \leq \mu_0 \tau \partial_d\varphi
    \lesssim \mu_0  |\xi'|_x.
  \end{align}
  With the first-order Taylor formula one has
  \begin{align*}
    \alpha_j^2 &= r(x,\xi'+i\tau d_{x'}\varphi)+(-1)^j\sigma^2\\ 
               &= r(x,\xi')
                 +\int_{0}^1 d_{\xi'}r(x,\xi'+i \tau s\, d_{x'}\varphi)
                 (i \tau d_{x'}\varphi)ds
                 +(-1)^j\sigma^2.
  \end{align*}
  With \eqref{eq: est 1 small perturbation} and homogeneity one
  has
  \begin{align*}
    \big|
    d_{\xi'}r(x,\xi'+i \tau s\, d_{x'}\varphi)  (i \tau d_{x'}\varphi)
    \big|
    \lesssim \mu_0  |\xi'|_x^2.
  \end{align*}
  One also has $\sigma \leq \mu_1 \tau \partial_d\varphi
    \lesssim \mu_1  |\xi'|_x$.
  Since $r(x,\xi') = |\xi'|_x^2$, this yields
  $\alpha_j^2  = |\xi'|_x^2  \big(1 + O (\mu_0 + \mu_1^2)\big)$ and hence 
  $\alpha_j   = |\xi'|_x \big(1 + O (\mu_0 + \mu_1^2)\big)$. This and
  \eqref{eq: est 1 small perturbation} gives the result. 
\end{proof}

Before proceeding, we make the following computation. For $j=1,2$
and $\ell=1,2$ we write 
\begin{align}
  \label{lop6}
  b_{\ell,\varphi} (x,\xi',\xi_d =\pi_{j,2},\tau)
  &=b_\ell(x,\xi'+i\tau d_{x'}\varphi, \pi_{j,2}+i\tau\partial_d\varphi) 
    = b_\ell(x,\xi'+i\tau d_{x'}\varphi, i\alpha_j)\notag\\
  &=b_\ell(x,\xi'+i\tau d_{x'}\varphi, i|\xi'|_x+i(\alpha_j-|\xi'|_x)).
\end{align}
We use Lemma~\ref{lemma: perturbation LS} and the value of $\eps>0$
  given therein.
  We choose $0< \mu_0\leq 1$ and $0 < \mu_1 \leq 1/2$ such that
  \begin{align}
    \label{eq: uniform condition mu0 mu1 small}
    C_3 (\mu_0+\mu_1^2) \leq \eps,
  \end{align}
with $C_3>0$ as given by Lemma~\ref{lemma: small perturbation}.

We now consider the root configurations that remain to consider
according to the discussion in Section~\ref{sec: discussion LS root
  positions}.

\noindent 
{\bfseries Case 2.}\\
In this case, one root of $\pconj$ lies in the upper complex closed
half-plane. We denote this root by $\rho^+$. According to the
discussion in Section~\ref{sec: discussion LS root positions} it
suffices to prove that
\begin{equation}
  \label{eq: LS case 2}
    \det
    \begin{pmatrix}
      b_{1,\varphi} &b_{2,\varphi}\\[2pt]
      \partial_{\xi_d}b_{1,\varphi}& \partial_{\xi_d}b_{2,\varphi} 
    \end{pmatrix}(x,\xi',\xi_d=\rho^+,\tau) \neq 0.
  \end{equation}
  In fact, one has $\rho^+ = \pi_{j,2}$ with $j=1$ or $2$.  We use
  the first part of Lemma~\ref{lemma: perturbation LS} with
  $\zeta' = i \tau d_{x'} \varphi$ and
  $\delta = i (\alpha_j -  |\xi|_x)$. With \eqref{lop6} and \eqref{eq:
    uniform condition mu0 mu1 small} with Lemma~\ref{lemma: small
    perturbation} and the first part of Lemma~\ref{lemma: perturbation
    LS} one obtains \eqref{eq: LS case 2}.

  \noindent 
  {\bfseries Case 3.}\\
  In this case $\Im \pi_{1,2}>0$ and $\Im \pi_{2,2} >0$.  According to the
  discussion in Section~\ref{sec: discussion LS root positions} it
  suffices to prove that
  \begin{equation}\label{eq: LS case 3}
  \det 
\begin{pmatrix}
  b_{1,\varphi}(x,\xi',\xi_d=\pi_{1,2},\tau) & b_{2,\varphi}(x,\xi',\xi_d=\pi_{1,2},\tau) \\[4pt]
  b_{1,\varphi}(x,\xi',\xi_d=\pi_{2,2},\tau) & b_{2,\varphi}(x,\xi',\xi_d=\pi_{2,2},\tau)
\end{pmatrix}\neq 0.
\end{equation}
We use the second part of 
  Lemma~\ref{lemma: perturbation LS} with
  $\zeta' = i \tau d_{x'} \varphi$, 
  $\delta = i (\alpha_1 -  |\xi|_x)$, and $\tilde{\delta} = i (\alpha_2 -  |\xi|_x)$. With \eqref{lop6} and \eqref{eq:
    uniform condition mu0 mu1 small} with Lemma~\ref{lemma: small
    perturbation} and the second part of Lemma~\ref{lemma: perturbation
    LS} one obtains \eqref{eq: LS case 3}.
    
  \noindent 
  {\bfseries Case 4.}\\
  In this case (that only occurs if $\sigma=0$) the \LS condition 
  holds also if one has \eqref{eq: LS case 2}. The proof is thus
  the same as for Case~2. This concludes the proof of
  Proposition~\ref{prop: LS after conjugation}.\hfill \qedsymbol \endproof
 
\subsection{Local stability of the algebraic conditions associated
  with the Lopatinski\u{\i}-\v{S}apiro condition}
\label{sec: Stability LS-conj}
In Section~\ref{sec: LS condition} we saw that the \LS condition for
$(\Psig, B_1, B_2)$ in Definition~\ref{def: LS} exhibits some stability
property.  This was used in the statement of Proposition~\ref{prop: LS
  after conjugation} that states how the \LS condition for
$(\Psig, B_1, B_2)$ can imply the \LS condition of
Definition~\ref{def: LS after conjugation} for
$(\Psig, B_1, B_2, \varphi)$, that is, the version of this condition
for the conjugated operators.

A natural question would then be: does the \LS condition for the
conjugated operators enjoy the same stability property? The answer is
yes. Yet, this is not needed in what follows. In fact, below one
exploits the algebraic conditions listed in Section~\ref{sec: discussion LS root positions} once the
\LS condition is know to hold at a point $\y^{0\prime}  = (x^0,
\xi^{0\prime}, \tau^0, \sigma^0)$ in tangential
phase space. One thus rather needs to know that these algebraic
conditions are stable. Here also the answer is positive and is the
subject of the present section.

\medskip

As in Definition~\ref{def: LS after conjugation} for $\y' =(x,\xi',\tau,\sigma)$ one denotes by
$\rho_j (\y')$
the roots of $\pconj(x,\xi',\xi_d, \tau)$ viewed as a polynomial in
$\xi_d$.

 Let $\y^{0\prime}  = (x^0, \xi^{0\prime}, \tau^0, \sigma^0)\in \partial
 \Omega \times \R^{d-1}\times [0,+\infty) \times [0,+\infty)$.
 One sets
 \begin{align*}
   J^+= \big\{ j \in \{1,2,3,4\}; \ \Im \rho_j (\y^{0\prime}) \geq 0\}, 
   \qquad 
    J^-= \big\{ j \in \{1,2,3,4\}; \ \Im \rho_j (\y^{0\prime}) < 0\}
 \end{align*}
 and, for $\y' = (x,\xi', \tau,\sigma)$,
 \begin{align*}
   \k^+_{\y^{0\prime}} (\y') 
   = \prod_{j \in J^+} \big(\xi_d -\rho_j(\y')\big), 
   \qquad 
   \k^-_{\y^{0\prime}} (\y') 
   = \prod_{j \in J^-} \big(\xi_d -\rho_j(\y')\big). 
 \end{align*}
 Naturally, one has $\k^+_{\y^{0\prime}} (\y^{0\prime},\xi_d) = \pconj^+(x^0,
 \xi^{0\prime}, \xi_d, \tau^0)$ and $\k^-_{\y^{0\prime}} (\y^{0\prime},\xi_d) = \pconj^-(x^0,
 \xi^{0\prime}, \xi_d, \tau^0)$. 
 Moreover, there exists a conic \nhd $\U_0$ of
 $\y^{0\prime}$ where both $\k^+_{\y^{0\prime}}(\y')$ and
 $\k^-_{\y^{0\prime}}(\y')$ are smooth with respect to $\y'$.
One has 
\begin{align*}
  \pconj = \pconj^+ \pconj^- 
  = \k^+_{\y^{0\prime}} \k^-_{\y^{0\prime}} .
 \end{align*}
 Note however that for $\y' = (x,\xi',\tau,\sigma) \in \U_0$ it may very well happen that
 $\pconj^+(x,\xi', \xi_d, \tau) \neq \k^+_{\y^{0\prime}} (\y',\xi_d)$
 and  $\pconj^-(x,\xi', \xi_d, \tau) \neq \k^-_{\y^{0\prime}} (\y',\xi_d)$.

 The following proposition can be found in \cite[proposition 1.8]{ML}.
 \begin{proposition}
   \label{prop: stability LS algegraic conditions}
   Let the \LS condition hold at $\y^{0\prime}  = (x^0, \xi^{0\prime}, \tau^0, \sigma^0)\in \partial
 \Omega \times \R^{d-1}\times [0,+\infty) \times [0,+\infty)$ for
 $(\Psig, B_1, B_2, \varphi)$. Then,
 \begin{enumerate}
 \item The polynomial
 $\xi_d \mapsto \pconj^+(x^0,
 \xi^{0\prime}, \xi_d, \tau^0)$ is of degree less than or equal to
 two.
 \item There exists a conic \nhd $\U$ of $\y^{0\prime}$ such that
   $\{b^1_{\varphi}(\y',\xi_d),b^2_{\varphi}(\y',\xi_d)\}$ is complete
   modulo $\k^+_{\y^{0\prime}} (\y',\xi_d)$ at every point $\y' = (x,\xi',
   \tau, \sigma) \in \U$,
   namely for
  any polynomial function $f(\xi_d)$ with complex coefficients there
  exist $c_1, c_2\in\C$ and a polynomial function
  $\ell(\xi_d)$ with complex coefficients such that, for all
  $\xi_d\in\C$
  \begin{equation}
    \label{eq: extended LS condition}
    f(\xi_d)=c_1 b_{1,\varphi}(x,\xi',\xi_d,\tau)
    +c_2b_{2,\varphi}(x,\xi',\xi_d,\tau)
    +\ell(\xi_d) \k^+_{\y^{0\prime}}(\y',\xi_d).
\end{equation}
\end{enumerate}
\end{proposition}
We emphasize again that the second property in Proposition~\ref{prop:
  stability LS algegraic conditions} looks very much like the
statement of \LS condition for $(\Psig, B_1, B_2, \varphi)$ at $\y'$
in Definition~\ref{def: LS after conjugation}. Yet, it differs by
having $\pconj^+(x,\xi',\xi_d,\tau)$ that only depends on the root configuration at $\y'$ replaced
 by $ \k^+_{\y^{0\prime}}(\y',\xi_d)$ whose structure is based on the
 root configuration at $\y^{0\prime}$. 

\medskip
  Let $m^+$ be the common degree of $\pconj^+(\y^{0\prime},\xi_d)$ and
$\k^+_{\y^{0\prime}} (\y',\xi_d)$ and 
$m^-$ be the common degree of $\pconj^-(\y^{0\prime},\xi_d)$ and
$\k^-_{\y^{0\prime}} (\y',\xi_d)$ for $\y'\in \U$. One has $m^+ + m^-
= 4$ and thus $m^- \geq 2$ for $\y'\in \U$ since $m^+ \leq 2$. 

Invoking the Euclidean division of polynomials, one sees that it is
sufficient to consider polynomials $f$ of degree less than or equal to
$m^+-1\leq 1$  in \eqref{eq: extended LS condition}.
Since the degree of $b_{j,\varphi} (\y',\xi_d)$ can be as high as
$3> m^+-1$ it however makes sense to consider $f$ of degree less than or equal to
$m=3$. Then, the second property in Proposition~\ref{prop:
  stability LS algegraic conditions} is equivalent to having 
\begin{align*}
  \{ b_{1,\varphi}(x,\xi',\xi_d,\tau),
  b_{2,\varphi}(x,\xi',\xi_d,\tau)\} \cup \bigcup_{0 \leq \ell \leq
  3- m^+} \{
  \k^+_{\y^{0\prime}}(\y',\xi_d) \xi_d^\ell\} 
\end{align*}
be a complete in the set of polynomials of degree less than or equal to
$m=3$. Note that this family is made of $m' = 6- m^+ = 2 + m^-$
polynomials. 

\medskip
We now express an inequality that follows from Proposition~\ref{prop:
  stability LS algegraic conditions} that will be key in the boundary
estimation given in Proposition~\ref{prop: local boundary
  estimate} below.

\subsection{Symbol positivity at the boundary}
\label{sec: Symbol postivity at the boundary}
The symbols $b_{j,\varphi}$, $j=1,2$, are polynomial in $\xi_d$ of
degree $k_j \leq 3$ and we
may thus write them in the form
\begin{align*}
b_{j,\varphi}(\y',\xi_d)=\sum\limits_{\ell=0}^{k_j}b_{j,\varphi}^\ell (\y')\xi_d^\ell,
\end{align*}
with $b_{j,\varphi}^\ell$ homogeneous of degree $k_j-\ell$.

The polynomial $\xi_d \mapsto \k^+_{\y^{0\prime}} (\y',\xi_d)$ is of
degree $m^+ \leq 2$ for $\y'\in \U$ with $\U$ given by Proposition~\ref{prop: stability LS
  algegraic conditions}.
Similarly, we write 
\begin{align*}
\k^+_{\y^{0\prime}} (\y',\xi_d)=\sum\limits_{\ell=0}^{m^+}\k^{+,\ell}_{\y^{0\prime}} (\y')\xi_d^\ell,
\end{align*}
with $\k^{+,\ell}_{\y^{0\prime}}$  homogeneous of degree $m^+-\ell$.
We introduce 
\begin{align*}
  e_{j, \y^{0\prime}} (\y', \xi_d) = 
\begin{cases}
  b_{j,\varphi}(\y',\xi_d) & \text{if} \ j=1,2,\\
  \k^+_{\y^{0\prime}} (\y',\xi_d) \xi_d^{j-3} & \text{if} \ j=3,\dots, m'.
\end{cases}
\end{align*}
As explained above, all these polynomials are of degree less than or
equal to three. 
If we now write 
\begin{align*}
  e_{j, \y^{0\prime}} (\y', \xi_d) 
  = \sum_{\ell=0}^{3} e_{j,
  \y^{0\prime}}^\ell (\y') \xi_d^\ell,
\end{align*}
for $j=1,2$ one has $e_{j, \y^{0\prime}}^\ell (\y') =
b_{j,\varphi}^\ell (\y')$, with $\ell =0, \dots, k_j$ and  $e_{j,
  \y^{0\prime}}^\ell (\y') =0$ for $\ell > k_j$, and 
\begin{align*}
  \text{for} \ j=3,\dots,m',  \quad 
   e_{j, \y^{0\prime}}^\ell (\y') = 
  \begin{cases}
   0 & \text{if} \ \ell <  j-3,\\
   \k^{+,\ell+3-j}_{\y^{0\prime}} (\y') & \text{if} \ \ell=j-3,\dots,
   m^+ + j- 3\leq m^+ + m' -3 = 3,\\
   0 & \text{if} \ \ell > m^+ + j- 3.
  \end{cases}
\end{align*}
In particular $e_{j, \y^{0\prime}}^\ell (\y')$ is homogeneous of
degree $m^+ +j - \ell -3$. 
We thus have the following result. 
\begin{lemma}
  \label{lemma: algebraic condition LS}
  Set the $m' \times (m+1) $ matrix $M (\y') = (M_{j, \ell} (\y')
  )_{{1\leq j \leq m'} \atop {0\leq \ell \leq m}}$ with $M_{j, \ell} (\y')  = e_{j, \y^{0\prime}}^\ell (\y')$. Then, the second point
  in Proposition~\ref{prop: stability LS algegraic conditions} states
  that $M (\y') $ is of rank $m+1=4$ for $\y ' \in \U$. 
\end{lemma}
Recall that  $m' = m^- + 2 \geq 4$.

\medskip
We now set 
\begin{align}
  \label{eq: boundary symbol}
  \un{e_{j, \y^{0\prime}}} (\y', \z) = \sum_{\ell=0}^{3} e_{j,
  \y^{0\prime}}^\ell (\y') z_\ell  = \sum_{\ell=0}^{3}  M_{j, \ell}
  (\y') z_\ell , \qquad \z = (z_0, \dots, z_3).
\end{align}
in agreement with the notation introduced in \eqref{eq: symbol z} in Section~\ref{sec: Boundary  differential quadratic forms}.
One has the following positivity result. 
\begin{lemma}
  \label{lemma: postivity boundary form}
  Let the \LS condition hold at $\y^{0\prime}  = (x^0, \xi^{0\prime}, \tau^0, \sigma^0)\in \partial
 \Omega \times \R^{d-1}\times [0,+\infty) \times [0,+\infty)$ for
 $(\Psig, B_1, B_2, \varphi)$ and let $\U$ be as given by
 Proposition~\ref{prop: stability LS algegraic conditions}.
 Then, if $\y' \in \U$ there exists $C>0$ such that 
 \begin{align*}
  \sum_{j=1}^{m'} \big|\un{e_{j, \y^{0\prime}}} (\y', \z)\big|^2 \geq
   C |\z|_{\C^{4}}^2, 
   \qquad \z = (z_0, \dots, z_3) \in \C^{4}.
\end{align*}
\end{lemma}
\begin{proof}
In $\C^4$ define  the bilinear form 
$\Sigma_{\mathscr B} (\z, \z') 
  = \sum_{j=1}^{m'} \un{e_{j,\y^{0\prime}}} (\y', \z)  
  \ovl{\un{e_{j,\y^{0\prime}}} (\y', \z')}$.
With \eqref{eq: boundary symbol} one has
\begin{align*}
  \Sigma_{\mathscr B} (\z, \z') 
  = \biginp{M (\y') \z}{M (\y') \z'}_{\C^{m'}}
  = \biginp{\transp \ovl{M (\y')}  M (\y') \z}{\z'}_{\C^{4}}.
\end{align*}
As $\rank \transp \ovl{M (\y')}  M (\y') = \rank  M (\y') =4$ by
Lemma~\ref{lemma: algebraic condition LS} one obtains
the result. 
\end{proof}

 \section{Estimate for the boundary norm under
    Lopatinski\u{\i}-\v{S}apiro condition}
  \label{sec: boundary norm under LS}
Near $x^0 \in \partial \Omega$ we consider two boundary operators $B_1$
and $B_2$. As in Section~\ref{sec: LS condition for conjugated
  bilaplace} the associated conjugated operators are denoted by 
$B_{j,\varphi}$, $j=1,2$ with respective principal symbols $b_{j,\varphi}(x,\xi,\tau)$.

The main result of this section is the following proposition for the
fourth-order conjugated operator $\Pconj$. It is key in the final
result of the present article. It states that all traces are
controlled by norms of $B_{1,\varphi} v\br$ and $B_{2,\varphi} v\br$ if the
\LS condition holds for $(P,B_1, B_2,\varphi)$.
\begin{proposition}
  \label{prop: local boundary estimate}
  Let $\k_0>0$. Let $x^0 \in \partial \Omega$, with
  $\Omega$ locally given by $\{ x_d>0\}$.
  Assume that  $(\Psig, B_1, P_2,\varphi)$ satisfies the
  \LS condition of Definition~\ref{def: LS after conjugation} at  $\y'=(x^0,\xi',\tau,\sigma)$
  for all $(\xi',\tau,\sigma) \in \R^{d-1} \times [0,+\infty) \times
  [0,+\infty)$ such that $\tau \geq \k_0\sigma$.

  Then, there exist $W^0$ a \nhd of $x^0$, $C>0$, $\tau_0>0$ such that
  \begin{align*}
    \normsc{ \trace(v)}{3,1/2} \leq C 
    \big(
    \Norm{\Pconj v}{+}
    +\sum\limits_{j=1}^{2}
    \normsc{B_{j, \varphi}v\br}{7/2-k_{j}}
    + \Normsc{v}{4,-1}
    \big),
  \end{align*}
  for $\sigma\geq 0$, $\tau\geq \max(\tau_0, \k_0\sigma)$  and $v\in \Cbarc(W^0_+)$.
\end{proposition}
The notation of the function space $\Cbarc(W^0_+)$ is introduced in \eqref{eq: notation Cbarc}.

For the proof of Proposition~\ref{prop: local boundary estimate} we
start with a microlocal version of the result.

\subsection{A microlocal estimate}
\begin{proposition}
  \label{prop: microlocal boundary estimate}
  Let $\k_1 > \k_0>0$. Let $x^0 \in \partial \Omega$, with
  $\Omega$ locally given by $\{ x_d>0\}$ and let $W$ be a bounded open
  \nhd of $x^0$ in $\R^d$. Let $(\xi^{0\prime},\tau^0,\sigma^0)\in \R^{d-1}\times [0,+\infty)
  \times [0,+\infty)$ nonvanishing with $\tau^0 \geq \k_1 \sigma^0$
  and such that $(\Psig, B_1, P_2,\varphi)$ satisfies the
  \LS condition of Definition~\ref{def: LS after conjugation} at $\y^{0\prime}
  = (x^0, \xi^{0\prime},\tau^0,\sigma^0)$.
 
  Then, there exists $\U$ a
  conic neighborhood of $\y^{0\prime}$ in
  $W\times \R^{d-1}\times [0,+\infty) \times [0,+\infty)$ where
  $\tau \geq \k_0 \sigma$  such that
  if  $\chi \in \Ssct^0$, homogeneous of degree $0$ in $(\xi',\tau,\sigma)$ with
  $\supp (\chi)\subset \U$, there exist $C>0$ and $\tau_0>0$ such that
  \begin{align*}
    \normsc{\trace(\Opt(\chi) v)}{3,1/2} \leq C 
    \Big( 
    \sum\limits_{j=1}^{2}
    \normsc{B_{j, \varphi}v\br}{7/2-k_{j}}
    + \Norm{\Pconj v}{+}
    + \Normsc{v}{4,-1}
    + \normsc{\trace(v)}{3,-1/2} 
    \Big),
  \end{align*}
  for $\sigma\geq 0$, $\tau\geq \max(\tau_0, \k_0\sigma)$  and $v\in \Cbarc(W_+)$.
\end{proposition}
\begin{proof}
We choose a conic \nhd $\U_0$ of $\y^{0\prime}$ according to
Proposition~\ref{prop: stability LS algegraic conditions} and such that $\U_0 \subset W \times \R^{d-1}
\times [0,+\infty) \times [0,+\infty)$ . Assume
moreover that $\tau \geq \k_0 \sigma$ in $\U_0$.

In Section~\ref{sec: Symbol postivity at the boundary} we introduced
the symbols $e_{j, \y^{0\prime}} (\y', \xi_d)$, $j=1, \dots, m' = m^-
+2 = 6 - m^+$. 
Set $\mathbb S_{\ovl{\U_0}} = \{ \y' = (x,\xi', \tau,\sigma) \in \ovl{\U_0}; \ |(\xi', \tau,\sigma)|=1\}$. 

Consequence of the \LS condition holding at
$\y^{0\prime}$ for all $\y' \in \mathbb S_{\ovl{\U_0}}$, by
Lemma~\ref{lemma: postivity boundary form} there exists $C>0$ such
that  
\begin{align*}
  \sum_{j=1}^{m'} \big|\un{e_{j, \y^{0\prime}}} (\y', \z)\big|^2 \geq
   C |\z|_{\C^{4}}^2, 
   \qquad \z = (z_0, \dots, z_3) \in \C^{4}.
\end{align*}
Since $\mathbb S_{\ovl{\U_0}}$ is compact (recall that $W$ is bounded), there exists $C_0>0$ such that 
\begin{align*}
  \sum_{j=1}^{m'} \big|\un{e_{j, \y^{0\prime}}} (\y', \z)\big|^2 \geq
   C_0 |\z|_{\C^{4}}^2, 
   \qquad \z = (z_0, \dots, z_3) \in \C^{4}, \ \y' \in \mathbb S_{\ovl{\U_0}}.
\end{align*}
Introducing the map $N_{t}\y'=(x, t \xi', t \tau, t \sigma)$, for
$\y'=(x,  \xi',  \tau, \sigma)$ with $t = |(\xi', \tau,\sigma)|^{-1}$ 
one has 
\begin{align}
  \label{eq: symbol positivity 1}
  \sum_{j=1}^{m'} \big|\un{e_{j, \y^{0\prime}}} (N_t \y', \z)\big|^2 \geq
   C_0 |\z|_{\C^{4}}^2, 
   \qquad \z = (z_0, \dots, z_3) \in \C^{4}, \ \y' \in \U_0, 
\end{align}
since $N_t \y' \in \mathbb S_{\ovl{\U_0}}$.
Now, for $j=1,2$  one has
\begin{align*}
  \un{e_{j, \y^{0\prime}}} (\y', \z)= \sum_{\ell=0}^{k_j} e_{j,
  \y^{0\prime}}^\ell (\y') z_\ell,
\end{align*}
with $ e_{j, \y^{0\prime}}^\ell (\y')$ homogeneous of degree
$k_j-\ell$, 
and for $3\leq j \leq m'$ one has 
\begin{align*}
  \un{e_{j, \y^{0\prime}}} (\y', \z)= \sum_{\ell=0}^{3} e_{j,
  \y^{0\prime}}^\ell (\y') z_\ell,
\end{align*}
with $ e_{j, \y^{0\prime}}^\ell (\y')$ homogeneous of degree
$m^++j-\ell-3$.
We define $\z' \in \C^4$ by $z'_\ell = t^{\ell-7/2} z_\ell$, $\ell=0,
\dots, 3$. 
One has 
\begin{align*}
  \un{e_{j, \y^{0\prime}}} (N_t \y', \z')= t^{k_j-7/2} \un{e_{j,
  \y^{0\prime}}} (\y', \z), \qquad j=1,2,
\end{align*}
and 
\begin{align*}
  \un{e_{j, \y^{0\prime}}} (N_t \y', \z')= t^{m^++j-13/2} \un{e_{j,
  \y^{0\prime}}} (\y', \z), \qquad j=3,\dots,m'.
\end{align*}
Thus from \eqref{eq: symbol positivity 1} we deduce
\begin{align}
  \label{eq: symbol positivity 2}
  \sum_{j=1}^{2} \lsct^{2(7/2-k_j)} \big|\un{e_{j, \y^{0\prime}}} (\y', \z)\big|^2 
  + \sum_{j=3}^{m'} \lsct^{2(13/2-m^+-j)} \big|\un{e_{j, \y^{0\prime}}} (\y', \z)\big|^2 
  \geq
   C_0 \sum_{\ell=0}^{3} \lsct^{2(7/2-\ell)}| z_\ell|^2, 
\end{align}
for $\z = (z_0, \dots, z_3) \in \C^{4}$, and $\y' \in \U_0$,
since $t \asymp \lsct^{-1}$ as $\tau \gtrsim \sigma$ in $\U_0$. 

We now choose $\U$ a conic open neighborhood of $\y^{0\prime}$, such that
$\ovl{\U}\subset\U_0$.  Let $\chi\in \Ssct^0$ be as in the statement and let  $\tchi\in
\Ssct^0$ be homogeneous of degree $0$, with
$\supp(\tchi)\subset\U_0$ and $\tchi\equiv 1$ in a neighborhood of
$\ovl{\U}$, and thus in a \nhd of $\supp(\chi)$.

For $j=3, \dots, m'$ one has $e_{j, \y^{0\prime}}(\y',\xi_d) =
\k^+_{\y^{0\prime}} (\y',\xi_d) \xi_d^{j-3} \in \Ssc^{m^++j-3,0}$. 
Set $E_{j} = \Op ( \tchi e_{j, \y^{0\prime}})$. The introduction
of $\tchi$ is made such that $\tchi e_{j, \y^{0\prime}}$ is defined on
the whole tangential phase-space. Observe that 
\begin{align*}
  \mathscr B (w) &= \sum\limits_{j=1}^{2} \normsc{B_{j, \varphi}w\br}{7/2-k_{j}}^2
  +\sum\limits_{j=3}^{m'} \normsc{E_{j}w\br}{13/2-m^+-j}^2\\
  &=\sum\limits_{j=1}^{2} 
    \norm{\Lsct^{7/2-k_{j}} B_{j, \varphi}w\br}{\partial}^2
    + \sum\limits_{j=3}^{m'} 
    \norm{\Lsct^{13/2-m^+-j} E_{j}w\br}{\partial}^2
\end{align*}
is a boundary quadratic form of type $(3,1/2)$ as in
Definition~\ref{def: boundary quadratic form}.
From Proposition~\ref{prop: boundary form -Gaarding tangentiel} and
\eqref{eq: symbol positivity 2} we have 
\begin{align}
  \label{eq: boundary inequality - est traces}
  \normsc{\trace(u)}{3,1/2}^2 \lesssim  \sum\limits_{j=1}^{2} \normsc{B_{j, \varphi}u\br}{7/2-k_{j}}^2
  +\sum\limits_{j=3}^{m'} \normsc{E_{j}u\br}{13/2-m^+-j}^2
  +  \normsc{\trace(v)}{3,-N}^2.
\end{align}
for $u = \Opt(\chi) v$ and $\tau \geq \k_0 \sigma$ chosen \suff
large.

\bigskip
In $\U_0$ one can write 
\begin{align*}
  \pconj = \pconj^+ \pconj^- 
  = \k^+_{\y^{0\prime}} \k^-_{\y^{0\prime}},
 \end{align*}
with $\k^+_{\y^{0\prime}}$ of degree $m^+$ and $\k^-_{\y^{0\prime}}$
of degree $m^-$.
In fact we set 
\begin{align*}
   \tk^+_{\y^{0\prime}} (\y') 
   = \prod_{j \in J^+} \big(\xi_d -\tchi \rho_j(\y')\big), 
   \qquad 
   \tk^-_{\y^{0\prime}} (\y') 
   = \prod_{j \in J^-} \big(\xi_d -\tchi \rho_j(\y')\big),
 \end{align*}
with the notation of Section~\ref{sec: Stability LS-conj},
thus making the two symbols defined on
the whole tangential phase-space.
In $\U$, one has also 
\begin{align*}
  \pconj 
  = \tk^+_{\y^{0\prime}} \tk^-_{\y^{0\prime}}.
 \end{align*}
The factor $\tk^-_{\y^{0\prime}}$ is associated with roots with
negative imaginary part. With Lemma~\ref{el1} given in
Appendix~\ref{SB} one has the following microlocal elliptic estimate
\begin{align*}
  \Normsc{\Opt(\chi)w}{m^-}
  + \normsc{\trace(\Opt(\chi)w)}{m^--1, 1/2}
  \lesssim 
  \Norm{\Opt(\tk^-_{\y^{0\prime}}) \Opt(\chi)w}{+}
  + \Normsc{w}{m^-,-N},
\end{align*}
for $w \in \Sbarp$ and $\tau \geq \k_0 \tau$ chosen \suff large. 
We apply this inequality to $w=\Opt(\tk^+_{\y^{0\prime}} )v$. 
Since 
\begin{align*}
  \Opt(\tk^-_{\y^{0\prime}}) \Opt(\chi)\Opt(\tk^+_{\y^{0\prime}} ) =
  \Opt(\chi) \Pconj \mod \Psisc^{4,-1},
\end{align*}
one obtains
\begin{align*}
  \normsc{\trace(\Opt(\chi) \Opt(\tk^+_{\y^{0\prime}})v)}{m^--1, 1/2}
  \lesssim 
  \Norm{\Pconj v}{+}
  + \Normsc{v}{4,-1}.
\end{align*}
With $[\Opt(\chi), \Opt(\tk^+_{\y^{0\prime}})] \in \Psisc^{m^+,-1}$ one
then has
\begin{align*}
  \normsc{\trace(\Opt(\tk^+_{\y^{0\prime}}) u)}{m^--1, 1/2}
  \lesssim 
  \Norm{\Pconj v}{+}
  + \Normsc{v}{4,-1}
  + \normsc{\trace(v)}{3,-1/2},
\end{align*}
with $u= \Opt(\chi) v$ as above, using that $m^+ + m^-=4$. 
Note that 
\begin{align*}
  \normsc{\trace(\Opt(\tk^+_{\y^{0\prime}}) u)}{m^--1, 1/2}
  & \asymp \sum_{j=0}^{m^--1} 
  \normsc{D_d^j \Opt(\tk^+_{\y^{0\prime}}) u\br}{m^--j- 1/2}\\
  &\gtrsim \sum_{j=3}^{m'} 
  \normsc{E_j u\br}{5/2+ m^--j} - \normsc{\trace(v)\br}{3, - 1/2} ,
\end{align*}
using that $\xi_d^j \tk^+_{\y^{0\prime}} = \tchi e_{j+3,
  \y^{0\prime}}$ in a conic \nhd of $\supp(\chi)$ and using that 
$m^- = m'-2$.
We thus obtain 
\begin{align*}
   \sum_{j=3}^{m'}  \normsc{E_{j}u\br}{13/2-m^+-j}
  \lesssim 
  \Norm{\Pconj v}{+}
  + \Normsc{v}{4,-1}
  + \normsc{\trace(v)}{3,-1/2},
\end{align*}
since $13/2-m^+ = 5/2+ m^-$.
With \eqref{eq: boundary inequality - est traces} then one finds
\begin{align*}
  \normsc{\trace(u)}{3,1/2} \lesssim  \sum \limits_{j=1}^{2} \normsc{B_{j, \varphi}u\br}{7/2-k_{j}}^2
  + \Norm{\Pconj v}{+}
  + \Normsc{v}{4,-1}
  + \normsc{\trace(v)}{3,-1/2}.
\end{align*}
In addition, observing that 
\begin{align*}
  \normsc{B_{j, \varphi}u\br}{7/2-k_{j}}\lesssim
  \normsc{B_{j, \varphi}v\br}{7/2-k_{j}} + \normsc{\trace(v)}{3,-1/2},
\end{align*}
the result of Proposition~\ref{prop: microlocal boundary estimate}
follows.
\end{proof}

\subsection{Proof of Proposition~\ref{prop: local boundary estimate}}
\label{sec: proof prop: local boundary estimate}
  As mentioned above the proof relies on a patching procedure of
  microlocal estimates given by Proposition~\ref{prop: microlocal
boundary estimate}.

   Let $0 < \k_0'  < \k_0$.
  We set 
  \begin{align*}
    \Gamma^{d-1}_{+, \k_0} = \{ (\xi',\tau,\sigma)\in\R^{d-1}\times [0,+\infty) \times
  [0,+\infty);\ \tau \geq \k_0 \sigma, \},
  \end{align*}
  and 
  \begin{align*}
    \mathbb{S}^{d-1}_{+, \k_0} = \{ (\xi',\tau,\sigma)\in
    \Gamma^{d-1}_{+, \k_0} ;\ |(\xi',\tau,\sigma)|=1\}.
  \end{align*}

 Consider $(\xi^{0\prime},\tau^0,\sigma^0) \in \mathbb{S}^{d-1}_{+, \k_0}$.
  Since the \LS condition holds at
  $\y^{0\prime} = (x^0, \xi^{0\prime},\tau^0,\sigma^0)$, we can invoke 
  Proposition~\ref{prop: microlocal boundary estimate}:
  \begin{enumerate}
  \item There exists a conic open neighborhood
  $\U_{\y^{0\prime}}$ of $\y^{0\prime}$ in
  $W \times \R^{d-1}\times [0,+\infty) \times [0,+\infty)$ where
  $\tau \geq \k_0' \sigma$;
  \item For any $\chi_{\y^{0\prime}} \in \Ssct^0$ homogeneous of
  degree $0$ supported in $\U_{\y^{0\prime}}$ the estimate of
  Proposition~\ref{prop: microlocal boundary estimate} applies to $\Opt(\chi_{\y^{0\prime}} ) v$ for $\tau
  \geq \max (\tau_{\y^{0\prime}} , \k_0\sigma)$. 
  \end{enumerate}
  Without any loss of generality we may choose $\U_{\y^{0\prime}}$ of
  the form $\U_{\y^{0\prime}} = \scrO_{\y^{0\prime}} \times
  \Gamma_{\y^{0\prime}}$, with $\scrO_{\y^{0\prime}} \subset W$ an open \nhd of
  $x^0$ and $\Gamma_{\y^{0\prime}}$ a conic open \nhd of
  $(\xi^{0\prime},\tau^0,\sigma^0)$ in $\R^{d-1}\times [0,+\infty) \times
  [0,+\infty)$ where $\tau \geq \k_0' \sigma$. 

  Since $\{ x^0\} \times \mathbb{S}^{d-1}_{+, \k_0}$ is compact we can
  extract a finite covering of it by open sets of the form of
  $\U_{\y^{0\prime}}$. We denote by
  $\tilde{\U}_i$, $i \in I$ with $|I| < \infty$, such a finite covering. This is
  also a finite covering of 
  $\{x^0\} \times \Gamma^{d-1}_{+, \k_0}$.

  Each $\tilde{\U}_i$ has the form $\tilde{\U}_i =  \scrO_i \times
  \Gamma_i$, with $\scrO_i$ an open \nhd of
  $x^0$ and $\Gamma_i$ a conic open set in $\R^{d-1}\times [0,+\infty) \times
  [0,+\infty)$ where $\tau \geq \k_0' \sigma$. 

  We set $\scrO = \cap_{i \in I} \scrO_i$ and 
  $\U_i = \scrO \times \Gamma_i$, $i \in I$.

  Let $W^0$ be an open \nhd of $x^0$ such that $W^0 \Subset
  \scrO$. The open sets $\U_i$
  give also an open covering of $\ovl{W^0} \times \mathbb{S}^{d-1}_{+,
    \k_0}$ and  $\ovl{W^0} \times \Gamma^{d-1}_{+, \k_0}$.
  With this second covering we associate a partition of unity
  $\chi_i$, $i \in I$, of $\ovl{W^0}  \times \mathbb{S}^{d-1}_{+,
    \k_0}$,
where each $\chi_i$ is chosen smooth and homogeneous
  of degree one for $|(\xi', \tau, \sigma)| \geq 1$, that is:
  \begin{align*}
    \sum\limits_{i \in I}\chi_i(\y')=1 
    \qquad \text{for}\ \y' = (x, \xi', \tau, \sigma) \ \text{in a \nhd
    of}\
    \ovl{W^0} \times \Gamma^{d-1}_{+, \k_0},
    \ \text{and} \  |(\xi', \tau, \sigma)| \geq 1.
  \end{align*}

  Let $u \in \Cbarc(W^0_+)) $.
 Since each $\chi_i$ is in $\Ssct^0$ and 
  supported in $\U_i$,  Proposition~\ref{prop: microlocal boundary
    estimate} applies:
  \begin{align}
     \label{eq: local boundary estimate-1}
    \normsc{\trace(\Opt(\chi_i) v)}{3,1/2} \leq C_i
    \Big( 
    \sum\limits_{j=1}^{2}
    \normsc{B_{j, \varphi}v\br}{7/2-k_{j}}
    + \Norm{\Pconj v}{+}
    + \Normsc{v}{4,-1}
    + \normsc{\trace(v)}{3,-1/2} 
    \Big),
    \end{align}
 for some $C_i>0$, for $\sigma\geq 0$, $\tau\geq \max(\tau_i,
 \k_0\sigma)$ for some $\tau_i>0$. 

 We set $\tchi=1-\sum\limits_{i \in I}\chi_i$. One has $\tchi \in
 \Ssct^{-\infty}$ microlocally in a \nhd of $\ovl{W^0} \times
 \Gamma^{d-1}_{+, \k_0}$.
 Thus, considering the definition of $\Gamma^{d-1}_{+, \k_0}$, if one
 imposes $\tau \geq \k_0 \sigma$, as we do,  then $\tchi \in
 \Ssct^{-\infty}$ locally in a \nhd of $\ovl{W^0}$.
  
 For any $N\in \N$ using that $\supp(v) \subset W^0$ one has
\begin{align*}
  \normsc{\trace (v)}{3,1/2}
  &\leq \sum\limits_{i \in I}  \normsc{\trace (\Opt (\chi_i)v)}{3,1/2}
  + \normsc{\trace (\Opt (\tchi)v)}{3,1/2}\\
  &\lesssim 
    \sum\limits_{i \in I}  \normsc{\trace (\Opt (\chi_i)v)}{3,1/2}
  + \normsc{\trace (v)}{3,-N}\\
  &\lesssim 
    \sum\limits_{i \in I}  \normsc{\trace (\Opt (\chi_i)v)}{3,1/2}
  + \Normsc{v}{4,-N}.
\end{align*}
Summing estimates \eqref{eq: local boundary estimate-1} together for $i
\in I$ we thus obtain
\begin{align*}
    \normsc{\trace(v)}{3,1/2} \lesssim
    \sum\limits_{j=1}^{2}
    \normsc{B_{j, \varphi}v\br}{7/2-k_{j}}
    + \Norm{\Pconj v}{+}
    + \Normsc{v}{4,-1}
    + \normsc{\trace(v)}{3,-1/2} ,
\end{align*}
for $\tau \geq \max( \max_i \tau_i,
 \k_0\sigma)$. 
Therefore, by choosing $\tau \geq \k_0 \sigma$ \suff large one obtains
the result of Proposition~\ref{prop: local boundary estimate}.
\hfill \qedsymbol \endproof

\section{Microlocal estimate for each second-order factors composing
  $\Pconj$}
\label{sec: Microlocal estimate for each second-order factors}
We recall that
$\Psig =\Delta^2-\sigma^4=(-\Delta-\sigma^2)(-\Delta+\sigma^2)$
with $\sigma\geq 0$. Set $\Qsig^j=-\Delta + (-1)^j\sigma^2$; then
$\Psig =\Qsig^1\Qsig^2.$
We also set $Q = - \Delta$, that is, $Q = Q^1_{0} = Q^2_{0}$.

The principal symbols of $\Qsig^j$ and $Q$ are given by
\begin{align}
  \label{eq: symbols laplace operator}
  \qsig^j (x,\xi)= \xi_d^2 + r(x,\xi') + (-1)^j \sigma^2
  \ \ \text{and} \ \ q (x,\xi) = \xi_d^2 + r(x,\xi'),
\end{align}
respectively.
The conjugated operator
$\Pconj=e^{\tau\varphi}\Psig e^{-\tau}$ reads
\begin{align*}
  \Pconj=\Qconj^1\Qconj^2,
  \ \ \text{with}\ \
  \Qconj^j=e^{\tau\varphi}\Qsig^je^{-\tau\varphi}.
\end{align*}
We set
\begin{align*}
  \Qs^j=\frac{\Qconj^j+(\Qconj^j)^*}{2}
  \ \ \text{and} \ \
  \Qa=\frac{\Qconj^j-(\Qconj^j)^*}{2i},
\end{align*}
both formally selfadjoint and such that 
  $\Qconj^j=\Qs^j+i\Qa$. Note that $\Qa$ is
  independent of $\sigma$. Their respective
  principal symbols
  are
  \begin{align*}
    &\qs^j(x,\xi,\tau,\sigma)=\xi_d^2-(\tau\partial_d\varphi)^2+r(x,\xi')
      + (-1)^j\sigma^2-\tau^2r(x,d_{x'}\varphi),\\
    &\qa(x,\xi,\tau)=2\tau\xi_d\partial_d\varphi+2\tau\tilde{r}(x,\xi',
      d_{x'}\varphi).
  \end{align*}
  Note that $\Qs^j$ and $\Qa$ take the forms
  \begin{align}
    \label{eq: form of Qs Qa}
    \Qs^j = D_d^2 + T_s^j, \qquad
    \Qa = \tau (\partial_d \varphi D_d
    + D_d \partial_d \varphi) + T_a,
  \end{align}
  where $T_s^j, T_a \in \Dsct^2$ are such that $(T_s^j)^* = T_s^j$ and
  $T_a^* = T_a$. 
  Naturally, the principal symbol of $\Qconj^j$ is 
  \begin{align*}
    \qconj^j(x,\xi,\tau)=\qs^j(x,\xi,\tau,\sigma)+i\qa(x,\xi,\tau).
  \end{align*}
The principal symbol of $\Qconj^j=e^{\tau\varphi}\Qsig^je^{-\tau\varphi}$ is 
\begin{align*}
  \qconj^j(x,\xi,\tau,\sigma)
  =(\xi_d+i\tau\partial_d\varphi)^2+r(x,\xi')
  + (-1)^j\sigma^2-\tau^2r(x,d_{x'}\varphi)+2i\tau\tilde{r}(x,\xi',d_{x'}\varphi).
  \end{align*}
As in Section~\ref{sec: Root configuration for each factor} we let  $\alpha_j \in \C$ be such that
\begin{align*}
  \alpha_j(x,\xi',\tau,\sigma)^2
  &= r(x, \xi' + i \tau d_{x'}\varphi) + (-1)^j \sigma^2\\
  &=r(x,\xi')
  -\tau^2 r(x,d_{x'}\varphi)
  +2i\tau\tilde{r}(x,\xi',d_{x'}\varphi)
  + (-1)^j \sigma^2,
\end{align*}
and $\Re \alpha_j \geq 0$. Note that uniqueness in the choice of
$\alpha_j$ holds except if $r(x, \xi' + i \tau d_{x'}\varphi)
+ (-1)^j \sigma^2 \in \R^-$; this lack of uniqueness in that case is however not an issue
in what follows. 
One has 
\begin{align*}
  \qconj^j(x,\xi',\xi_d,\tau)
  &= (\xi_d+i\tau\partial_d\varphi)^2+\alpha_j (x,\xi',\tau,\sigma)^2\\
  &=
    \big(\xi_d+i\tau\partial_d\varphi+i\alpha_j(x,\xi',\tau,\sigma)\big)
    \big(\xi_d+i\tau\partial_d\varphi-i\alpha_j(x,\xi',\tau,\sigma)\big).
\end{align*}
We recall from \eqref{eq: form of the roots of q} that we write $\qconj^j(x,\xi',\xi_d,\tau)=(\xi_d-\pi_{j,1})(\xi_d-\pi_{j,2})$ with
\begin{align*}
  \pi_{j,1}=-i\tau\partial_d\varphi- i\alpha_j(x,\xi',\tau,\sigma)
  \ \ \text{and} \ \
  \pi_{j,2} = -i\tau\partial_d\varphi + i\alpha_j(x,\xi',\tau,\sigma).
\end{align*}
The roots $\pi_{j,k}$, $k=1,2$ are functions of $x,\xi',\tau$ and $\sigma.$

\bigskip We denote by $B$ a boundary operator of order $k$ that takes
the form $$B(x,D)=B^{k}(x,D') + B^{k-1}(x,D')D_d,$$ with $B^{k}(x,D')$
and $B^{k-1}(x,D')$ tangential differential operators of order $k$ and
$k-1$ respectively. The boundary operator $B(x,D)$ has
$b(x,\xi)=b^k(x,\xi') + b^{k-1}(x,\xi')\xi_d$ for principal
symbol. The conjugate boundary operator
$B_{\varphi}=e^{\tau\varphi}Be^{-\tau\varphi}$ is then given by
\begin{align*}
B_{\varphi}(x,D,\tau) &=B^k_{\varphi}(x, D',\tau)+B^{k-1}_{\varphi}(x,D',\tau)(D_d+i\tau\partial_d\varphi)\\
&= 
\hat{B}^k_{\varphi}(x,D',\tau)+
  B^{k-1}_{\varphi}(x,D',\tau)D_d,
\end{align*}
with
$\hat{B}^k_{\varphi}(x,D',\tau)
  =B^k_{\varphi}(x,D',\tau)+i \tau
  B^{k-1}_{\varphi}(x,D',\tau)\partial_d\varphi$.
The principal symbol of $B_{\varphi}(x,D,\tau)$ is 
\begin{align*}
  b_\varphi(x,\xi,\tau) =  \hat{b}^k_{\varphi}(x,\xi',\tau) + b^{k-1}_{\varphi}(x,\xi',\tau) \xi_d,
\end{align*}
 where $b^{k-1}_{\varphi}(x,\xi',\tau)$  is homogeneous of
 degree $k-1$ in $\lsct$ and 
 $\hat{b}^k_{\varphi}(x,\xi',\tau)=b^k_{\varphi}(x,\xi',\tau)+ i \tau
 b^{k-1}_{\varphi}(x,\xi',\tau)\partial_d\varphi$ is homogeneous of
 degree $k$ in $\lsct$.

\subsection{Sub-ellipticity}
 Set
   \begin{align*}
     \qs(x,\xi,\tau)
     =\xi_d^2+r(x,\xi')-(\tau\partial_d\varphi)^2
     -r(x,\tau d_{x'}\varphi)
     =|\xi|^2_x-|\tau d\varphi|_x^2,
   \end{align*}
   where $|\xi|^2_x=\xi_d^2+r(x,\xi')$. One has $\qs^j= \qs + (-1)^j
   \sigma^2$.
   Observe that $\{\qs^j,\qa\}=\{\qs, \qa\}$.
\begin{definition}[Sub-ellipticity]\label{de1}
  Let $W$ be a bounded open subset of $\R^d$ and
  $\varphi\in \Cinf(\ovl{W})$ such that
  $|d_x\varphi|>0$. Let $j =1$ or $2$.  We say that the couple
  $(\Qsig^j,\varphi)$ satisfies the sub-ellipticity condition
  in $\ovl{W}$ if there exist $C>0$ and $\tau_0>0$ such that for
  $\sigma>0$
\begin{align*}
  \forall~(x,\xi)\in\ovl{W}\times\R^d,~\tau\geq
 \tau_0\sigma,~\qconj^j(x,\xi,\tau)=0\Rightarrow
 \{\qs^j,\qa\}(x,\xi,\tau) = \{\qs,\qa\}(x,\xi,\tau)\geq C>0.
 \end{align*}
\end{definition}
\begin{remark}
  Note that with homogeneity the sub-ellipticity property also reads
  \begin{align*}
    \forall~(x,\xi)\in\ovl{W}\times\R^d,~\tau\geq
    \tau_0\sigma,~\qconj^j(x,\xi,\tau)=0\Rightarrow
    \{\qs^j,\qa\}(x,\xi,\tau)\geq C \lsc^3.
 \end{align*}
\end{remark}
\begin{proposition}
  \label{prop: sub-ellipticity recipe}
    Let $W$ be a bounded open subset of $\R^d$ and
    $\psi\in\Cinf(\R^d)$ such that $\psi\geq 0$ and
    $|d_x\psi|\geq C>0$ on $\ovl{W}.$ Let $\tau_0>0$.  Then, there exists
    $\csp_0\geq 1$ such that $(\Qsig^j,\varphi)$
    satisfies the sub-ellipticity condition
    on $\ovl{W}$ for $\tau\geq\tau_0 \sigma$ for
    $\varphi=e^{\csp\psi}$, with $\csp\geq \csp_0$, for both
    $j=1$ and $2$.
  \end{proposition}
 \begin{proof}
   We note that $|d_x \varphi(x)|\neq0$.

   The proof is slightly different whether one considers the symbol
   $\qconj^1$ or the symbol $\qconj^2.$

   \medskip
   \noindent
\textbf{Case 1:   proof for $\bld{\qconj^1}$}. Assume that
$\qconj^1=0$. Thus $|\xi|_x^2-|\tau
d\varphi|_x^2-\sigma^2=0$ implying $|\xi|\sim \sigma+\csp\tau\varphi$. On the one hand by Lemma 3.55 in \cite{JGL-vol1}, one has 
 \begin{align}\label{mes1}
   \{\qs, \qa\}(x,\xi,\tau)
   &=\tau(\csp^2\varphi)(\csp\varphi)^2\left((H_{q}\psi(x,\beta))^2+4\tau^2 q(x, d\psi(x))^2\right) \notag\\
 &\quad+(\csp\varphi)^3\frac{1}{2i}\{\ovl{q_{\psi}}, q_{\psi}\}(x,\beta,\tau),
 \end{align}
 with $\beta = \xi /(\csp\varphi)$, and where $H_q$ denotes the
 Hamiltonian vector field associated with the symbol $q$ as defined in
 \eqref{eq: symbols laplace operator}. Here, $q_\psi$ denotes the
 principal symbol of $e^{\tau \psi} Q e^{-\tau \psi}$, that is,
 \begin{align*}
   q_\psi (x,\xi, \tau)= q(x,\xi + i \tau d_x \psi(x))
   =   (\xi_d + i \tau \partial_d \psi(x))^2 + r(x,\xi' +
   i \tau d_{x'} \psi(x)).
 \end{align*}
On the other hand, one has $(H_{q}\psi(x,\beta))^2+4\tau^2 q(x,
d\psi(x))^2\gtrsim \tau^2$ and since
$\frac{1}{2i}\{\ovl{q_{\psi}}, q_{\psi}\}(x,\beta,\tau)$ is
homogeneous of degree $3$ in $(\beta,\tau),$ we obtain 
\begin{align*}
  \{\qs, \qa\}
  \geq C \csp(\csp\tau\varphi)^3
  -C'  (\csp\tau\varphi+|\beta|\csp\varphi)^3
  =C\csp\tilde{\tau}^3
  -C''(\tilde{\tau}+|\xi|)^3,\ \ \text{with} \ \tilde{\tau}=\csp\tau\varphi.
   \end{align*}
Yet one has $|\xi|\sim\sigma+\tilde{\tau}$ implying  
\begin{align*}
  \{\qs, \qa\}
  \geq C\csp\tilde{\tau}^3
  -C'''(\tilde{\tau}+|\xi|)^3
  \geq C \csp\tilde{\tau}^3-C^{(4)}(\tilde{\tau}+\sigma)^3.
   \end{align*}
Since $\psi\geq 0$ and $\csp\geq 1$ one has $\varphi\geq1$ implying 
$\tau_0 \sigma \leq\tau\lesssim\tilde{\tau}$ and thus
\begin{align*}
  \{\qs, \qa\}(x,\xi,\tau)\geq \tilde{\tau}^3(C\csp-C^{(5)}).
\end{align*}
It follows that for $\csp$ chosen \suff large one finds $\{\qs, \qa\}(x,\xi,\tau)\geq C>0.$

\textbf{Case 2:   proof for $\bld{\qconj^2}$}.  Assume
that $\qconj^2=0$. Then  $|\xi|_x^2+\sigma^2=|\tau
d\varphi|_x^2$ implying $|\xi|+\sigma\sim\tau|d\varphi| \sim \tilde{\tau}.$
The same computation as in Case~1 gives
\begin{align*}
\{\qs, \qa\}(x,\xi,\tau) \geq C \csp\tilde{\tau}^3-C'(\tilde{\tau}+|\xi|)^3.
\end{align*}
Here $|\xi|+\tilde{\tau}\lesssim \tilde{\tau}$ yielding
\begin{align*}
  \{\qs, \qa\}(x,\xi,\tau) \geq (C \csp-C'')\tilde{\tau}^3.
  \end{align*}
The remaining part of the proof is the same.
\end{proof}

 \begin{lemma}\label{urg1}
   Let $j=1$ or $2$. Let $(\Qsig^j,\varphi)$ have the sub-ellipticity property
   of Definition \ref{de1} in $\ovl{W}$. For $\mu>0$ one sets
   $t(\y)=\mu
   ((\qs^j)^2+\qa^2)(\y)+\tau\{\qs^j,\qa\}(\y)$
   with
   $\y=(x,\xi,\tau,\sigma)\in\ovl{W}\times\R^d\times[0,\infty)\times[0,\infty)$.
   Let $\tau_0 >0$. Then, for $\mu$ chosen \suff large and
   $\tau\geq \tau_0\sigma$ one has $t(\y)\geq C\lsc^4$
   for some $C>0$.
\end{lemma}
The proof of Lemma \ref{urg1} uses the following lemma.
\begin{lemma}\label{urg2}
Consider two continuous functions, $f$ and $g$, defined in a compact set $\mathscr{L},$ and assume that $f\geq 0$ and moreover
$$f(y)=0\Rightarrow g(y)>0~~~\text{for all}~~~y\in\mathscr{L}.$$ Setting $h_{\mu}=\mu f+g,$ we have $h_{\mu}\geq C>0$ for $\mu>0$ chosen \suff large.
\end{lemma}
\begin{proof}[Proof of Lemma~\ref{urg1}]
  Consider the compact set
  \begin{align*}
    \mathscr{L}=
    \{(x,\xi,\tau,\sigma); \ x\in\ovl{W},\
    |\xi|^2+\tau^2+\sigma^2=1, \ \tau\geq \tau_0\sigma\}.
  \end{align*}
  Applying the result of Lemma \ref{urg2} to $t(\y)$ on
  $\mathscr{L}$ with $f=(\qs^j)^2+\qa^2$ and $g=\tau\{\qs^j,
  \qa\}$ we find for $t(\y)\geq C$ on $\mathscr{L}$ for some
  $C>0$ for $\mu$
  chosen \suff large. Since $t(\y)$ is homogeneous of degree $4$
  in the variables $(\xi,\tau,\sigma)$ it follows that
  $t(\y)\geq C(\sigma^2+\tau^2+|\xi|^2)^4\gtrsim \lsc^4.$
\end{proof}

 \subsection{Lopatinski\u{\i}-\v{S}apiro condition for the
   second-order factors}

 Above, in Section~\ref{sec: LS condition for conjugated bilaplace}, the \LS condition is addressed for the
 fourth-order operator $\Pconj$. Here, we consider the
 two second-order factors $\Qconj^j$.

  With the roots $\pi_{j,1}$ and $\pi_{j,2}$ defined in \eqref{eq: form of
    the roots of q} one sets
  \begin{align*}
    \qconj^{j,+} (x,\xi', \xi_d ,\tau)
    =\prod\limits_{{\Im\pi_{j,k} \geq 0}\atop {k=1,2}}
    \big(\xi_d-\pi_{j,k} (x,\xi',\tau,\sigma) \big).
  \end{align*}
 \begin{definition}\label{de2}
   Let $j=1,2$.  Let $x \in \partial\Omega$, with $\Omega$ locally given by
  $\{ x_d>0\}$. 
   Let $(\xi',\tau,\sigma)\in \R^{d-1}\times [0,+\infty) \times
   [0,+\infty)$ with
   $(\xi',\tau,\sigma)\neq0$. One says that the \LS condition holds for
   $(\Qsig^j,B, \varphi)$ at $\y' = (x,\xi',\tau,\sigma)$ if for any
   polynomial function $f(\xi_d)$ with complex coefficients there
   exist $c\in\C$ and a polynomial function $\ell(\xi_d)$ with
   complex coefficients such that, for all $\xi_d\in\C$
   \begin{equation}\label{sts1}
     f(\xi_d)=c b_{\varphi}(x,\xi',\xi_d,\tau)
     +\ell(\xi_d) \qconj^{j,+}(x,\xi',\xi_d,\tau).
\end{equation}
\end{definition}
 \begin{remark}
   \label{sts2}
   With the Euclidean division of polynomials, we see that it suffices
   to consider the polynomial function $f(\xi_d)$ to be of degree less
   than that of $\qconj^{j,+} (x,\xi',\xi_d,\tau)$ in
   \eqref{sts1}. Thus, in any case, the degree of $f(\xi_d)$ can be chosen
   less than or equal to one.
 \end{remark}
 \begin{lemma}
   \label{sl1}
   Let $j=1$ or $2$. Let $x \in \partial\Omega$ and $(\xi',\tau, \sigma)\in
   \R^{d-1}\times [0,+\infty) \times
   [0,+\infty)$ with $(\xi',\tau,\sigma)\neq 0$. The \LS condition
   holds for $(\Qsig^j,B,\varphi)$ at $(x,\xi',\tau,\sigma)$
   if and only if
   \begin{enumerate}
   \item either $\qconj^{j,+} (x,\xi',\xi_d,\tau)=1$;
 \item or
   $\qconj^{j,+} (x,\xi',\xi_d,\tau)=\xi_d-\pi$
   and $b_{\varphi}(x,\xi',\pi,\tau)\neq 0.$
 \end{enumerate}
 \end{lemma}
 \begin{proof}
   If
   $\qconj^{j,+} (x,\xi',\xi_d,\tau)=(\xi_d-\pi_{j,1})(\xi_d-\pi_{j,2})$,
   that is, both roots $\pi_{j,1}$ and $\pi_{j,2}$ are in the upper
   complex half-plane, then condition \eqref{sts1} cannot hold, since
   by Remark \ref{sts2} it means that the vector space of polynomials
   of degree less than or equal to one would be  generated by the single polynomial $b_{\varphi}(x,\xi',\xi_d,\tau)$.\\
   Suppose that $\qconj^{j,+} (x,\xi',\xi_d,\tau)=\xi_d-\pi$ that is
   one the root $\pi_{j,1}$ and $\pi_{j,2}$ has a nonnegative imaginary
   part and the other root has a negative imaginary part. Then, the \LS condition holds at $(x,\xi',\sigma,\tau)$ if for any $f(\xi_d)$, the polynomial function $\xi_d\mapsto f(\xi_d)-c b_{\varphi}(x,\xi',\xi_d,\tau)$ admits $\pi$ as a root for some $c\in\C$. A necessary and sufficient condition is then $b_{\varphi}(x,\xi',\xi_d=\pi,\tau)\neq 0$.\\
   Finally if $\qconj^{j,+} (x,\xi',\xi_d,\tau)=1$, that
   is, both roots $\pi_{j,1}$ and $\pi_{j,2}$ lie in the lower complex
   half-plane, then condition \eqref{sts1} trivially holds. 
\end{proof}

\subsection{Microlocal estimates for a second-order factor}

Here, for $j=1$ or $2$, we establish estimates for the operator
$\Qsig^j$ in a microlocal \nhd of point at the boundary where
$(\Qsig^j,B,\varphi)$ satisfies the \LS condition (after conjugaison)
of Definition~\ref{de2}.

The quality of the estimation depends on the
position of the roots. We shall assume that $\partial_d
\varphi >0$. Thus, from the form of the roots $\pi_{j,1}$
and $\pi_{j,2}$ given in \eqref{eq: form of the roots of q},
the root $\pi_{j,1}$ always lies in the  lower
complex half-plane. The sign of $\Im \pi_{j,2}$ may however vary. 
Three cases can thus occur:
\begin{enumerate}
\item The root $\pi_{j,2}$ at the considered point lies in the upper complex half-plane.
\item The root $\pi_{j,2}$ at the considered point is real.
 \item The root $\pi_{j,2}$ at the considered point lies in the lower complex half-plane.
 \end{enumerate}

\begin{proposition}\label{sl3}
  Let $j=1$ or $2$ and $\k_1 > \k_0>0$. Let $x^0 \in \partial \Omega$, with
  $\Omega$ locally given by $\{ x_d>0\}$ and let $W$ be a bounded open
  \nhd of $x^0$ in $\R^d$. Let $\varphi$ be such that
  $\partial_d \varphi \geq C>0$ in $W$ and such that
  $(\Qsig^j, \varphi)$ satisfies the sub-ellipticity condition in
  $\ovl{W}$.  Let
  $\y^{0\prime} = (x^0, \xi^{0\prime},\tau^0,\sigma^0)$ with
  $(\xi^{0\prime},\tau^0,\sigma^0)\in \R^{d-1}\times [0,+\infty)
  \times [0,+\infty)$ nonvanishing with $\tau^0 \geq \k_1 \sigma^0$
  and such that $(\Qsig^j, B,\varphi)$ satisfies the
  \LS condition of Definition~\ref{de2} at $\y^{0\prime}$.

  \begin{enumerate}
    \item Assume that $\Im \pi_{j,2} (\y^{0\prime})>0$.
  Then, there exists $\U$ a 
  conic neighborhood of $\y^{0\prime}$ in
  $W\times \R^{d-1}\times [0,+\infty) \times [0,+\infty)$ where
  $\tau \geq \k_0 \sigma$  such that if  $\chi \in \Ssct^0$,
  homogeneous of degree $0$ in $(\xi',\tau,\sigma)$ with $\supp (\chi)\subset \U$, there exist $C>0$ and $\tau_0>0$ such that
  \begin{align}\label{eq: microlocal estimate Qconj-1}
    \Normsc{\Opt (\chi) v}{2}
    +\normsc{\trace (\Opt (\chi)v)}{1,1/2}\leq
    C \big(
    \Norm{\Qconj^jv}{+}
    + \normsc{B_{\varphi}v\br}{3/2-k}
    +\Normsc{v}{2,-1} \big),
\end{align}
for $\sigma\geq 0$, $\tau\geq \max(\tau_0, \k_0\sigma)$  and $v\in
\Cbarc(W_+)$.

\item Assume that $\Im \pi_{j,2} (\y^{0\prime})=0$.
  Then, there exists $\U$ a 
  conic neighborhood of $\y^{0\prime}$ in
  $W\times \R^{d-1}\times [0,+\infty) \times [0,+\infty)$ where
  $\tau \geq \k_0 \sigma$  such that if  $\chi \in \Ssct^0$,
  homogeneous of degree $0$ in $(\xi',\tau,\sigma)$ with $\supp (\chi)\subset \U$, there exist $C>0$ and $\tau_0>0$ such that
  \begin{align}\label{eq: microlocal estimate Qconj-2}
    \tau^{-1/2 }\Normsc{\Opt (\chi) v}{2}
    +\normsc{\trace (\Opt (\chi)v)}{1,1/2}\leq
    C \big(
    \Norm{\Qconj^jv}{+}
    + \normsc{B_{\varphi}v\br}{3/2-k}
    +\Normsc{v}{2,-1} \big),
\end{align}
for $\sigma\geq 0$, $\tau\geq \max(\tau_0, \k_0\sigma)$  and $v\in
\Cbarc(W_+)$.

\item Assume that $\Im \pi_{j,2} (\y^{0\prime})<0$.
  Then, there exists $\U$ a 
  conic neighborhood of $\y^{0\prime}$ in
  $W\times \R^{d-1}\times [0,+\infty) \times [0,+\infty)$ where
  $\tau \geq \k_0 \sigma$  such that if  $\chi \in \Ssct^0$,
  homogeneous of degree $0$ in $(\xi',\tau,\sigma)$ with $\supp (\chi)\subset \U$, there exist $C>0$ and $\tau_0>0$ such that
  \begin{align}\label{eq: microlocal estimate Qconj-3}
    \Normsc{\Opt (\chi) v}{2}
    +\normsc{\trace (\Opt (\chi)v)}{1,1/2}\leq
    C \big(
    \Norm{\Qconj^jv}{+}
    +\Normsc{v}{2,-1} \big),
\end{align}
for $\sigma\geq 0$, $\tau\geq \max(\tau_0, \k_0\sigma)$  and $v\in
\Cbarc(W_+)$.
\end{enumerate}
\end{proposition}
The notation of the function space $\Cbarc(W_+)$ is introduced in \eqref{eq: notation Cbarc}.

\subsubsection{Case (i): one root lying in the upper complex half-plane.}
One has
 $\Im\pi_{j,2}(\y^{0\prime}) > 0$ and $\Im \pi_{j,1}
 (\y^{0\prime})<0$.

 Since the \LS condition holds for $(\Qsig^j, B,\varphi)$ at
 $\y^{0\prime}$, by Lemma \ref{sl1} one has
 \begin{align*}
   b_{\varphi}(x^0,\xi^{0\prime},\xi_d=\pi_{j,2}(\y^{0\prime}),\tau^0)
   =b \big(x^0, \xi^{0\prime}+i\tau^0 d_{x'}\varphi(x^0),
   i\alpha_j(\y^{0\prime})\big)
   \neq 0.
 \end{align*}
 As the roots
 $\pi_{j,1}$ and $\pi_{j,2}$ are locally smooth with respect to
 $\y'=(x,\xi',\tau,\sigma)$ and homogeneous of degree one in
 $(\xi',\tau,\sigma),$ there exists $\U$ a conic neighborhood
 of $\y^{0\prime}$ in
 $W\times\R^{d-1}\times[0,+\infty) \times[0,+\infty) $ and $C_1>0,$ $C_2>0$ such
 that
 $\mathbb{S}_{\ovl{\U}}=\{\y'\in
 \ovl{\U}; |\xi'|^2+\tau^2+\sigma^2=1\}$ is
 compact
 and
 \begin{align*}
   \tau \geq \k_0 \sigma, \quad
   \Im\pi_{j,2}(\y')\geq C_ 2 \lsct,
   \quad \text{and}\
   \Im\pi_{j,1}(\y')\leq -C_1 \lsct,
 \end{align*}
 and
\begin{equation}\label{oub1}
  b_{\varphi}(x,\xi',\xi_d=\pi_{j,2}(\y'),\tau)\neq 0,
\end{equation}
if $\y'= (x,\xi',\tau,\sigma) \in\ovl{\U}$.

We let $\chi\in \Ssct^0$ and
$\tchi\in \Ssct^0$ be homogeneous of degree zero
in the variable $(\xi',\tau,\sigma)$ and be such that
$\supp(\tchi)\subset\U$ and
$\tchi\equiv 1$ on  a \nhd of $\supp(\chi)$. From the smoothness and
the homogeneity of the roots, one  has
$\tchi\pi_{j,k}\in \Ssct^1$, $k=1,2$. We
set
\begin{align*}
  L_2=D_d-\Opt(\tchi\pi_{j,2})
  \quad \text{and} \
  L_1=D_d-\Opt(\tchi\pi_{j,1}).
\end{align*}

\medskip
The proof of Estimate~\eqref{eq: microlocal estimate Qconj-2} is based on three lemmata that we now
list. Their proofs are given at the end of this section.

\medskip
The following lemma provides an estimate for $L_2$ and  boundary traces.
\begin{lemma}\label{lr1}
There exist $C>0$ and $\tau_0>0$ such that for any $N\in\N,$ there exists $C_N>0$ such that
\begin{equation*}
  \bignormsc{\trace \big(\Opt(\chi)w\big)}{1,1/2}
  \leq C\big(
  \bignormsc{B_{\varphi}\Opt(\chi)w\br}{3/2-k}
  +\bignormsc{L_2 \Opt(\chi)w\br}{1/2}
  \big)
  +C_N \normsc{\trace(w)}{1,-N},
\end{equation*}
for $\tau\geq \max(\tau_0, \k_0\sigma)$ and $w\in \Sbarp.$
\end{lemma}
The proof of Lemma~\ref{lr1} relies on the \LS condition.

The
following lemma gives an estimate for $L_1$.
\begin{lemma}\label{sl2}
  Let $\chi\in \Ssct^0$, homogeneous of degree 0, be such that
  $\supp(\chi)\subset\U$ and $s \in \R$.  There exist $C>0$, $\tau_0>0$ and
  $N\in\N$ such that
  \begin{align*}
    \Normsc{\Opt(\chi)w}{1,s}
    +\normsc{\Opt(\chi)w\br}{s+1/2}
    \leq C \big(\Normsc{L_1 \Opt(\chi)w}{0,s}
    +\Normsc{w}{0,-N} \big),
  \end{align*}
  for  $w\in\Sbarp$ and $\tau\geq \max(\tau_0, \k_0\sigma)$.
\end{lemma}
The proof of Lemma \ref{sl2} is based on a multiplier method and
relies on the fact that the root $\pi_{j,1}$ that appears in the
principal symbol of $L_1$ lies in the lower complex  half-plane.

The
following lemma gives an estimate for $L_2$.
\begin{lemma}\label{lemma: estimate L2}
  Let $\chi\in \Ssct^0$, homogeneous of degree 0, be such that
  $\supp(\chi)\subset\U$ and $s \in \R$.  There exist $C>0$, $\tau_0>0$ and
  $N\in\N$ such that
  \begin{align*}
    \Normsc{\Opt(\chi)w}{1,s}
    \leq C \big(\Normsc{L_2 \Opt(\chi)w}{0,s}
    + \normsc{\Opt(\chi)w\br}{s+1/2}
    +\Normsc{w}{0,-N} \big),
  \end{align*}
  for  $w\in\Sbarp$ and $\tau\geq \max(\tau_0, \k_0\sigma)$.
\end{lemma}
Note that this estimate is weaker than that of Lemma~\ref{sl2}

\bigskip
Observing that
\begin{align*}
  L_1 \Opt(\chi)L_2 &=\Opt(\chi)L_1 L_2  \mod \Psisc^{1,0}\\
  &= \Opt(\chi)\Qconj^j \mod \Psisc^{1,0},
\end{align*}
and applying Lemma~\ref{sl2} to $w=L_2 v$ with $s=0$, one obtains
\begin{align*}
  \Normsc{\Opt(\chi) L_2 v}{1}
 + \normsc{\Opt(\chi) L_2 v\br}{1/2}
    &\lesssim \Norm{L_1 \Opt(\chi) L_2 v}{+}
    +\Normsc{v}{1,-N}\\
    &\lesssim  \Norm{\Qconj^jv}{+}
                +\Normsc{v}{1},
\end{align*}
for $\tau \geq \k_0 \sigma$ chosen \suff large.
We set $u=\Opt(\chi)v,$ and using the trace inequality
\begin{align*}
  \normsc{w\br}{s}\lesssim \Normsc{w}{s+1/2},
  \qquad w\in
  \Sbarp \ \text{and} \ s>0,
\end{align*}
we have
\begin{align*}
   \Normsc{L_2 u}{1}+ 
  \normsc{L_2 u\br}{1/2}
  &\lesssim \Normsc{\Opt(\chi) L_2 v}{1}
    + \normsc{\Opt(\chi)L_2 v\br}{1/2}
    + \Normsc{v}{1} +\normsc{v\br}{1/2}\\
  &\lesssim  \Normsc{\Opt(\chi) L_2 v}{1} + \normsc{\Opt(\chi)L_2 v\br}{1/2} + \Normsc{v}{1}.
\end{align*}
Therefore, we obtain
\begin{align*}
   \Normsc{L_2 u}{1}+ \normsc{L_2 u\br}{1/2} 
  \lesssim  \Norm{\Qconj^jv}{+}
                +\Normsc{v}{1}.
\end{align*}
With Lemma~\ref{lr1}, one has the estimate
\begin{equation*}
  \bignormsc{\trace (u)}{1,1/2} + \Normsc{L_2 u}{1}
  \lesssim 
  \normsc{B_{\varphi}u\br}{3/2-k}
  + \Norm{\Qconj^jv}{+}
  +\Normsc{v}{2,-1},
\end{equation*}
for $\tau \geq \k_0 \sigma$ chosen \suff large
using the following trace inequality
\begin{align*}
  \normsc{\trace(w)}{m,s}
  \lesssim \Normsc{w}{m+1,s-1/2},
  \qquad w\in \Sbarp \ \text{and} \ m\in\N, \
  s\in\R.
\end{align*}
With Lemma~\ref{lemma: estimate L2} for $s=1$ one obtains
\begin{equation*}
   \Normsc{u}{1,1} + \bignormsc{\trace (u)}{1,1/2} + \Normsc{L_2 u}{1}
  \lesssim 
  \normsc{B_{\varphi}u\br}{3/2-k}
  + \Norm{\Qconj^jv}{+}
  +\Normsc{v}{2,-1},
\end{equation*}
for $\tau \geq \k_0 \sigma$ chosen \suff large.
Finally, we write 
\begin{equation*}
  \Normsc{D_du}{1} \leq
  \Normsc{L_2 u}{1} +  \Normsc{\Opt(\tchi \pi_{j,2}) u}{1}
  \lesssim
  \Normsc{L_2 u}{1} + \Normsc{u}{1,1}, 
  \end{equation*}
  yielding
  \begin{equation*}
   \Normsc{u}{2} + \bignormsc{\trace (u)}{1,1/2} 
  \lesssim 
  \normsc{B_{\varphi}u\br}{3/2-k}
  + \Norm{\Qconj^jv}{+}
  +\Normsc{v}{2,-1}.
\end{equation*}
As $u=\Opt(\chi)v$, with a commutator argument we obtain
\begin{align*}
  \normsc{B_{\varphi}u\br}{3/2-k}
  & \lesssim
    \normsc{B_{\varphi}v\br}{3/2-k}
    + \normsc{\trace(v)}{1,-1/2}\\
    & \lesssim
    \normsc{B_{\varphi}v\br}{3/2-k}
      + \Normsc{v}{2,-1}.
\end{align*}
yielding~\eqref{eq: microlocal estimate Qconj-1} and
thus concluding the proof of
Proposition~\ref{sl3} in  Case~(i).\hfill \qedsymbol \endproof

\medskip
We now provide the proofs the three key lemmata used above.
\begin{proof}[Proof of Lemma~\ref{lr1}]
Set 
\begin{align*}
  \mathcal T(w) &= \normsc{B_\varphi w\br}{3/2-k}^2 
         + \normsc{L_2 w\br}{1/2}^2
=\norm{\Lsct^{3/2-k} B_\varphi w\br}{\partial}^2
                  + \norm{\Lsct^{1/2} L_2 w\br}{\partial}^2.
\end{align*}
This is a boundary differential quadratic form of type $(1,1/2)$ in
the sense of Definition~\ref{def: boundary quadratic form}.
The associated bilinear symbol is given by 
\begin{align*}
  \un{\mathcal T}(\y',\z,\z') 
  &= \lsct^{3-2k} \big (\hat{b}^k_{\varphi}(x,\xi',\tau) z_0 +
  b^{k-1}_{\varphi}(x,\xi',\tau) z_1\big) 
  \big(\ovl{\hat{b}^k_{\varphi}}(x,\xi',\tau) \bar{z}_0'+
  \ovl{ b^{k-1}_{\varphi}}(x,\xi',\tau) \bar{z}_1'\big)\\
  &\quad + \lsct \big(z_1 - \tchi \pi_{j,2}(\y') z_0 \big)
    \big(\bar{z}_1' - \tchi \ovl{\pi_{j,2}(\y')}\bar{z}_0'\big),
\end{align*}
with $\z = (z_0, z_1) \in \C^2$ and $\z' = (z_0', z_1') \in \C^2$,
yielding
\begin{align*}
  \un{\mathcal T}(\y',\z,\z) 
  &= \lsct^{3-2k} \big |\hat{b}^k_{\varphi}(x,\xi',\tau) z_0 +
  b^{k-1}_{\varphi}(x,\xi',\tau) z_1\big|^2 
    + \lsct \big|z_1 - \tchi \pi_{j,2}(\y') z_0 \big|^2.
\end{align*}
One has $\un{\mathcal T}(\y',\z,\z)\geq 0$. 
For $\z\neq (0,0)$ if $\un{\mathcal T}(\y',\z,\z)=0$ then 
\begin{align*}
  \begin{cases}
    z_1 = \tchi \pi_{j,2}(\y') z_0,\\
    \hat{b}^k_{\varphi}(x,\xi',\tau) z_0 +
  b^{k-1}_{\varphi}(x,\xi',\tau) z_1=0,
  \end{cases}
\end{align*}
implying that $z_0 \neq 0$ and 
\begin{align*}
   b_\varphi \big( x,\xi', \xi_d= \tchi \pi_{j,2}(\y'), \tau \big) 
  = \hat{b}^k_{\varphi}(x,\xi',\tau) +
    b^{k-1}_{\varphi}(x,\xi',\tau)  \tchi \pi_{j,2}(\y') =0.
\end{align*}
Let $\U_1 \subset \U$ be a conic open set such that $\supp(\chi)
\subset \U_1$ and $\tchi=1$ in a conic \nhd of 
$\ovl{\U_1}$. Then, for  $\y' \in \ovl{\U_1}$  one has 
\begin{align*}
  b_\varphi \big( x,\xi', \xi_d= \tchi \pi_{j,2}(\y'), \tau \big) 
  = b_\varphi \big( x,\xi', \xi_d= \pi_{j,2}(\y'), \tau \big) \neq 0,
\end{align*}
by \eqref{oub1}. From the homogeneity of
$b^{k-1}_{\varphi}(x,\xi',\tau)$ and
$\hat{b}^k_{\varphi}(x,\xi',\tau)$ in $\y'$, 
it follows that there exists some $C>0$ such that 
\begin{align*}
  \un{\mathcal T}(\y',\z,\z) \geq C \big(\lsct^3 |z_0|^2 + \lsct  |z_1|^2 \big),
\end{align*}
if $\y' \in \U_1$.
The result of Lemma~\ref{lr1} thus follows from Proposition~\ref{prop:
  boundary form -Gaarding tangentiel}, having in mind what is exposed in
Section~\ref{sec: calculus with spectral parameter} since we have $\tau \geq \k_0 \sigma$ here.
\end{proof}
\begin{proof}[Proof of Lemma~\ref{sl2}]
We let $u=\Opt(\chi)w.$ Performing an integration by parts, one has
\begin{align*}
  2\Re\biginp{L_1 u}{i\Lsct^{2s+1} u}_+
  & = 2\Re\biginp{ \big(D_d-\Opt(\tchi\pi_{j,1})\big)u} {i\Lsct^{2s+1} u}_+\\
  &=  \Re \biginp{i
    \big(\Lsct^{2s+1}\Opt(\tchi\pi_{j,1})-\Opt(\tchi\pi_{j,1} )^*\Lsct^1\big)u
    }{u}_+\\
 &\quad + \Re\inp{\Lsct^{2s+1}u\br}{u\br}_{\partial}.
\end{align*}
Note that 
$\Re\inp{\Lsct^{2s+1}u\br}{u\br}_{\partial}
  = \normsc{u\br}{s+1/2}^2$.

\medskip
 Next, the operator
$i \big(
\Lsct^{2s+1}\Opt(\tchi\pi_{j,1})-\Opt(\tchi\pi_{j,1})^*\Lsct^{2s+1}\big) $
has the following {\em real} principal symbol
\begin{align*}
  \vartheta(\y')=-2\Im\pi_{j,1}(\y')\lsct^{2s+1}.
\end{align*}
and since $\Im\pi_{j,1}(\y') \leq - C_1 \lsct <0$ in $\U$ one obtains
$\vartheta(\y')\gtrsim\lsct^{2s+2}$ in $\U$. Since $\U$ is \nhd of
$\supp(\chi)$, the microlocal G{\aa}rding inequality of Theorem~2.49 in
 \cite{JGL-vol1} (the proof adapts to the case with parameter $\sigma$ as
 explained in Section~\ref{sec: calculus with spectral parameter} since $\sigma
\lesssim \tau$) yields
 \begin{align*}
   2\Re\biginp{L_1 u}{i\Lsct^{2s+1} u}_+
   \geq \normsc{u\br}{s+1/2}^2
   +C\Norm{\Lsct^{s+1}u}{+}^2
   -C_N\Normsc{w}{0,-N}^2,
 \end{align*}
 for $\tau\geq \k_0\sigma$ chosen \suff
 large. 
 With the Young inequality one obtains
 \begin{align*}
   |\biginp{L_1 u}{i\Lsct^{2s+1} u}_+|
   \lesssim \frac{1}{\varepsilon} \Norm{\Lsct^s L_1 u}{+}^2
   +\varepsilon
   \Norm{\Lsct^{s+1} u}{+}^2,
   \end{align*}
which yields for $\varepsilon$ chosen \suff small,
 \begin{equation}\label{mch1}
   \normsc{u\br}{s+1/2}
   +\Normsc{u}{0,s+1}
   \lesssim \Normsc{ L_1 u}{0,s} +\Normsc{w}{0,-N}.
 \end{equation}
Finally, we write 
\begin{equation}\label{mch2}
  \Normsc{D_du}{0,s} \leq
  \Normsc{L_1 u}{0,s} +  \Normsc{\Opt(\tchi \pi_{j,1}) u}{0,s}
  \lesssim
  \Normsc{L_1 u}{0,s} + \Normsc{u}{0,s+1}.
  \end{equation}
Putting together \eqref{mch1} and \eqref{mch2}, the result of
Lemma~\ref{sl2} follows.
\end{proof}

\begin{proof}[Proof of Lemma~\ref{lemma: estimate L2}]
We let $u=\Opt(\chi)w.$ Performing an integration by parts, one has
\begin{align*}
  2\Re\biginp{L_2 u}{-i\Lsct^{2s+1} u}_+
  & = 2\Re\biginp{ \big(D_d-\Opt(\tchi\pi_{j,2})\big)u} {-i\Lsct^{2s+1} u}_+\\
  &=  \Re \biginp{i
    \big(\Opt(\tchi\pi_{j,2})^*\Lsct^1 - \Lsct^{2s+1}\Opt(\tchi\pi_{j,2})\big)
    u}{u}_+\\
 &\quad - \Re\inp{\Lsct^{2s+1}u\br}{u\br}_{\partial}.
\end{align*}
Note that 
$\Re\inp{\Lsct^{2s+1}u\br}{u\br}_{\partial}
  = \normsc{u\br}{s+1/2}^2$.

\medskip
 Next, the operator
$i \big(\Opt(\tchi\pi_{j,2})^*\Lsct^{2s+1}
-\Lsct^{2s+1}\Opt(\tchi\pi_{j,2})\big) $
has the following {\em real} principal symbol
\begin{align*}
  \vartheta(\y')=2\Im\pi_{j,2}(\y')\lsct^{2s+1}.
\end{align*}
and since $\Im\pi_{j,2}(\y') \geq C_2 \lsct >0$ in $\U$ one obtains
$\vartheta(\y')\gtrsim\lsct^{2s+2}$ in $\U$. Since $\U$ is \nhd of
$\supp(\chi)$, the microlocal G{\aa}rding inequality of Theorem~2.49 in
 \cite{JGL-vol1} (the proof adapts to the case with parameter $\sigma$ as
 explained in Section~\ref{sec: calculus with spectral parameter} since $\sigma
\lesssim \tau$) yields
 \begin{align*}
   2\Re\biginp{L_2 u}{i\Lsct^{2s+1} u}_+
   \geq - \normsc{u\br}{s+1/2}^2
   +C\Norm{\Lsct^{s+1}u}{+}^2
   -C_N\Normsc{w}{0,-N}^2,
 \end{align*}
 for $\tau\geq \k_0\sigma$ chosen \suff
 large. 
 The end of the proof is then similar to that of Lemma~\ref{sl2}.
\end{proof}

 \subsubsection{Case (ii): one real root.}
 \label{sec: Case 1: one root  up}
 One has
 $\Im\pi_{j,2}(\y^{0\prime}) = 0$ and $\Im \pi_{j,1} (\y^{0\prime})<0$.
 
 Since the \LS condition holds for $(\Qsig^j, B,\varphi)$ at
 $\y^{0\prime}$, by Lemma \ref{sl1} one has
 \begin{align*}
   b_{\varphi}(x^0,\xi^{0\prime},\xi_d=\pi_{j,2}(\y^{0\prime}),\tau^0)
   =b \big(x^0, \xi^{0\prime}+i\tau^0 d_{x'}\varphi(x^0),
   i\alpha_j(\y^{0\prime})\big)
   \neq 0.
 \end{align*}
 As the roots
 $\pi_{j,1}$ and $\pi_{j,2}$ are locally smooth with respect to
 $\y'=(x,\xi',\tau,\sigma)$ and homogeneous of degree one in
 $(\xi',\tau,\sigma),$ there exists $\U$ a conic neighborhood
 of $\y^{0\prime}$ in
 $W\times\R^{d-1}\times[0,+\infty) \times[0,+\infty) $ and $C_1>0,$ $C_2>0$ such
 that
 $\mathbb{S}_{\ovl{\U}}=\{\y'\in
 \ovl{\U}; |\xi'|^2+\tau^2+\sigma^2=1\}$ is
 compact
 and
 \begin{align*}
   \tau \geq \k_0 \sigma, \quad
   \pi_{j,1}(\y')\neq\pi_{j,2}(\y'),
   \quad \Im\pi_{j,2}(\y')\geq - C_ 2 \lsct,
   \quad \text{and}\
   \Im\pi_{j,1}(\y')\leq -C_1 \lsct,
 \end{align*}
 and
\begin{equation}\label{oub2}
  b_{\varphi}(x,\xi',\xi_d=\pi_{j,2}(\y'),\tau)\neq 0,
\end{equation}
if $\y'= (x,\xi',\tau,\sigma) \in\ovl{\U}$.

\medskip We let $\chi\in \Ssct^0$ and
$\tchi\in \Ssct^0$ be homogeneous of degree zero
in the variable $(\xi',\tau,\sigma)$ and be such that
$\supp(\tchi)\subset\U$ and
$\tchi\equiv 1$ on $\supp(\chi)$. From the smoothness and
the homogeneity of the roots, one  has
$\tchi\pi_{j,k}\in \Ssct^1$, $k=1,2$. We
set
\begin{align*}
  L_2=D_d-\Opt(\tchi\pi_{j,2})
  \quad \text{and} \
  L_1=D_d-\Opt(\tchi\pi_{j,1}).
\end{align*}

\medskip
Lemma~\ref{lr1} and Lemma~\ref{sl2} also apply in Case~(ii) and we shall use them.
In addition to these two lemmata we
shall need the following lemma.
\begin{lemma}
  \label{lemma: carre Carleman}
  There exist $C>0$, $\tau_0>0$ such that
  \begin{align*}
    \tau^{-1/2} \Normsc{w}{2} \leq C \big(
    \Norm{\Qconj^j w}{+} + \normsc{\trace(w)}{1,1/2}
    \big),
  \end{align*}
  for  $\tau\geq \max(\tau_0, \k_0\sigma)$ and $w\in \Cbarc(W_+)$.
\end{lemma}
Proving Lemma~\ref{lemma: carre Carleman} is fairly classical,
based on writing $\Qconj^j=\Qs^j+i\Qa$  and on an expansion of
$\Norm{\Qconj^j w}{+}^2$ and some integration by parts. We provide
the details in the proof below  as the occurence of the parameter
$\sigma$ is not that  classical.  Lemma~\ref{lemma: carre Carleman}
expresses the loss of a half-derivative if one root, here $\pi_{j,2}$,
is real.

\bigskip
Observing that
\begin{align*}
  L_1 \Opt(\chi)L_2 &=\Opt(\chi)L_1 L_2  \mod \Psisc^{1,0}\\
  &= \Opt(\chi)\Qconj^j \mod \Psisc^{1,0},
\end{align*}
and applying Lemma~\ref{sl2} to $w=L_2 v$, one obtains
\begin{align*}
 \normsc{\Opt(\chi) L_2 v\br}{1/2}
    &\lesssim \Norm{L_1 \Opt(\chi) L_2 v}{+}
    +\Normsc{v}{1,-N}\\
    &\lesssim  \Norm{\Qconj^jv}{+}
                +\Normsc{v}{1},
\end{align*}
for $\tau \geq \k_0 \sigma$ chosen \suff large.
We set $u=\Opt(\chi)v,$ and using the trace inequality
\begin{align*}
  \normsc{w\br}{s}\lesssim \Normsc{w}{s+1/2},
  \qquad w\in
  \Sbarp \ \text{and} \ s>0,
\end{align*}
we have
\begin{align*}
  \normsc{L_2 u\br}{1/2}
  &\lesssim \normsc{\Opt(\chi)L_2 v\br}{1/2}+\normsc{v\br}{1/2}\\
  &\lesssim  \normsc{\Opt(\chi)L_2 v\br}{1/2} + \Normsc{v}{1}.
\end{align*}
Therefore, we obtain
\begin{align*}
  \normsc{L_2 u\br}{1/2}
  \lesssim  \Norm{\Qconj^jv}{+}
                +\Normsc{v}{1}.
\end{align*}

On the one hand, together with Lemma~\ref{lr1}, one has the estimate
\begin{equation}\label{lr3}
  \bignormsc{\trace (u)}{1,1/2}
  \lesssim 
  \normsc{B_{\varphi}u\br}{3/2-k}
  + \Norm{\Qconj^jv}{+}
  +\Normsc{v}{2,-1},
\end{equation}
for $\tau \geq \k_0 \sigma$ chosen \suff large
using the following trace inequality
\begin{align*}
  \normsc{\trace(w)}{m,s}
  \lesssim \Normsc{w}{m+1,s-1/2},
  \qquad w\in \Sbarp \ \text{and} \ m\in\N, \
  s\in\R.
\end{align*}
On the other hand, with Lemma~\ref{lemma: carre Carleman} one has
\begin{align*}
  \tau^{-1/2} \Normsc{u}{2}
  \lesssim 
  \Norm{\Qconj^j u}{+} + \normsc{\trace(u)}{1,1/2},
\end{align*}
again for $\tau \geq \k_0 \sigma$ chosen \suff large
and since $[\Qconj^j, \Opt(\chi)]\in \Psisc^{1,0}$  one finds
\begin{align}
  \label{lr4}
  \tau^{-1/2} \Normsc{u}{2}
  \lesssim 
  \Norm{\Qconj^j v}{+} + \Normsc{v}{1} + \normsc{\trace(u)}{1,1/2},
\end{align}
Now, with $\varepsilon>0$ chosen \suff small if one computes $\eqref{lr3} + \varepsilon\times\eqref{lr4}$ one obtains 
\begin{align*}
  \tau^{-1/2} \Normsc{u}{2}
  + \normsc{\trace(u)}{1,1/2}
  &\lesssim
    \Norm{\Qconj^j v}{+}
    + \normsc{B_{\varphi}u\br}{3/2-k}
    +\Normsc{v}{2,-1}.
\end{align*}
As $u=\Opt(\chi)v$, with a commutator argument we obtain
\begin{align*}
  \normsc{B_{\varphi}u\br}{3/2-k}
  & \lesssim
    \normsc{B_{\varphi}v\br}{3/2-k}
    + \normsc{\trace(v)}{1,-1/2}\\
    & \lesssim
    \normsc{B_{\varphi}v\br}{3/2-k}
      + \Normsc{v}{2,-1}.
\end{align*}
yielding~\eqref{eq: microlocal estimate Qconj-2} and thus concluding the proof of
Proposition~\ref{sl3} in  Case~(ii).\hfill \qedsymbol \endproof

\medskip
We now provide a proof of Lemma~\ref{lemma: carre Carleman}.
\begin{proof}[Proof of Lemma~\ref{lemma: carre Carleman}]
We recall that $\Qconj^j=\Qs^j+i\Qa$,  yielding 
\begin{align}
  \label{eq: carre Carleman}
  \Norm{\Qconj^j w}{+}^2
  =\Norm{\Qs^j w}{+}^2 +\Norm{\Qa w}{+}^2
    +2\Re \inp{\Qs^jw}{i\Qa w}_+.
\end{align}
With the integration by parts formula
$\inp{f}{D_d g}_+=\inp{D_d f}{g}_+- i \inp{f\br}{g\br}_{\partial}$, 
and the forms of $\Qs^j$ and $\Qa$ given in \eqref{eq: form of Qs Qa} one has
\begin{align*}
  \inp{f}{\Qs^j g}_+ = \inp{\Qs^j f}{ g}_+
  - i \inp{f\br}{D_d g\br}_{\partial} - i \inp{D_d f\br}{g\br}_{\partial},
\end{align*}
and
\begin{align*}
  \inp{f}{\Qa g}_+ = \inp{\Qa f}{ g}_+
  - 2 \tau i \inp{ \partial_d \varphi f\br}{ g\br}_{\partial},
\end{align*}
yiedling 
\begin{align*}
  &\inp{\Qa w}{\Qs^jw}_+
  =\inp{\Qs^j\Qa w}{w}_+
  -i \inp{\Qa w\br}{D_d w\br}_{\partial}
  -i \inp{D_d\Qa w\br}{w\br}_{\partial}\\
  &\inp{\Qs^jw}{\Qa w}_+
    =\inp{\Qa\Qs^jw}{w}_+
  -2i\tau\inp{\partial_d\varphi \Qs^jw \br}{w\br}_{\partial}.
\end{align*}
This gives
\begin{align}
  \label{eq: double product after IPP}
  2\Re \inp{\Qs^jw}{i\Qa w}_+
  = i \inp{[\Qs^j,\Qa]w}{w}_+
  +\tau\mathcal{A}(w)
\end{align}
with
\begin{align}
  \label{eq: boundary quadratic form}
  \mathcal{A}(w)
  =\tau^{-1} \inp{\Qa w}{D_d w}_{\partial}
  +\tau^{-1} \biginp{(D_d\Qa  -2 \tau \partial_d\varphi \Qs^j)w}{w}_{\partial}.
\end{align}
We have the following lemma adapted from Lemma~3.25 in \cite{JGL-vol1}.
\begin{lemma}
  \label{lemma: quadratic form at the boundary}
 The operators $\Qa \in \tau \mathscr D^1$ and $D_d\Qa-2\tau \partial_d\varphi
 \Qs^j \in
  \Dsc^3$ can be cast in the following forms
  \begin{align*}
    \Qa  = 2 \tau \partial_d\varphi D_d + 2 \tilde{r} (x,D',\tau d_{x'}\varphi)
    \mod \tau \mathscr D^0,
  \end{align*}
  and 
  \begin{align*}
    D_d \Qa-2\tau \partial_d\varphi \Qs^j 
    &= - 2 \tau \partial_d\varphi
      \big(R(x,D') + (-1)^j\sigma^2
      -(\tau\partial_d\varphi)^2-r(x,\tau d_{x'}\varphi)   \big)\\
    & \quad
      + 2 \tilde{r} (x,D',\tau d_{x'} \varphi) D_d
     \mod \tau \Psisc^{1,0}.
  \end{align*}
\end{lemma}
With this lemma we find
\begin{align}
  \label{mb0}
  \mathcal{A}(w) &=2 \inp{\partial_d\varphi D_d  w\br}{D_dw\br}_\partial 
  + 2 \inp{\tilde{r} (x,D', d_{x'}\varphi) w\br}{D_dw\br}_\partial  \notag\\
 & \quad + 2 \biginp{ \tilde{r} (x,D', d_{x'} \varphi) D_d
 w\br}{w\br}_\partial \notag\\
 &\quad -  2 \biginp{  \partial_d\varphi
   \big(R(x,D') +(-1)^j \sigma^2-(\tau\partial_d\varphi)^2-r(x,\tau d_{x'}\varphi)\big)
 w\br}{w\br}_\partial \notag\\
 & \quad
 +  \inp{\Op(c_0) w\br}{ D_dw\br}_\partial 
 +  \biginp{  \big(\Op(\tilde{c}_0)  D_d+  \Op(c_1)\big)  w\br}{w\br}_\partial,
\end{align}
with $\Op(c_0), \Op(\tilde{c}_0) \in \mathscr D^0$ and $\Op(c_1) \in \Dsct^1$.
Observe that one has
\begin{equation}\label{nn1}
 |\mathcal{A}(w)|\lesssim \normsc{\trace(w)}{1,0}^2.
\end{equation}
From~\eqref{eq: carre Carleman} and \eqref{eq: double product after IPP} one writes 
\begin{align}
  \label{eq: carre Carleman-bis}
  \Norm{\Qconj^j w}{+}^2 + \tau \normsc{\trace(w)}{1,0}^2
  \gtrsim \Norm{\Qs^j w}{+}^2 +\Norm{\Qa w}{+}^2
    +\Re \inp{i [\Qs^j,\Qa]w}{w}_+.
\end{align}
We now use the following lemma whose proof is given below. 
\begin{lemma}\label{fg2}
  There exists $C, C'>0$,  $\mu>0$ and $\tau_0>0$ such that
  \begin{align*}
    \mu \big( \Norm{\Qs^j w}{+}^2 +\Norm{\Qa w}{+}^2\big)
    + \tau \Re \inp{i [\Qs^j,\Qa]w}{w}_+\geq
    C \Normsc{w}{2}^2
    -C'\normsc{\trace(w)}{1,1/2}^2
  \end{align*}
for $\tau\geq \max(\tau_0, \k_0\sigma)$  and $w\in \Cbarc(W_+)$
\end{lemma}
Let $\mu>0$ be as in Lemma~\ref{fg2} and let $\tau >0$ be such that
$\mu \tau^{-1}\leq 1$.
From~\eqref{eq: carre Carleman-bis} one then writes
\begin{align*}
  \Norm{\Qconj^j w}{+}^2 + \tau \normsc{\trace(w)}{1,0}^2
  &\gtrsim \tau^{-1} \Big( \mu \big(\Norm{\Qs^j w}{+}^2 +\Norm{\Qa w}{+}^2\big)
    +i \tau \inp{[\Qs^j,\Qa]w}{w}_+\Big),
\end{align*}
which with Lemma~\ref{fg2} yields
the result of Lemma~\ref{lemma: carre Carleman}
using that $\tau \normsc{\trace(w)}{1,0} \lesssim \normsc{\trace(w)}{1,1/2}$.
\end{proof}
\begin{proof}[ Proof of Lemma \ref{fg2}]
  One has $[\Qs^j,\Qa]\in \tau\Dsc^2$. Writing
  \begin{align*}
    \tau\Re\left(i[\Qs^j,\Qa]w,w\right)_+
    =\Re\left(i\tau^{-1}[\Qs^j,\Qa]w, \tau^2 w\right)_+,
  \end{align*}
  it can be seen as a interior differential quadratic form of type $(2,0)$ as
  in Definition~\ref{def: interior quadratic form}. Therefore
  \begin{align*}
    T(w)=\mu \big( \Norm{\Qs^j w}{+}^2 +\Norm{\Qa
    w}{+}^2\big)+\tau\Re\left(i[\Qs^j,\Qa]w,w\right)_+
  \end{align*}
  is also an
  interior differential quadratic form of this type with principal symbol given by 
  \begin{align*}
    t(\y)=\mu |\qconj^j(\y)|^2+\tau
    \{q^j_s,q_{a}\}(\y),\qquad \y=(x,\xi,\tau,\sigma).
  \end{align*}
  Let $\tau_0>0$. By Lemma \ref{urg1}, the sub-ellipticity property of $(\Qsig^j,\varphi)$ implies
\begin{align*}
  t(\y)\gtrsim \lsc^4,\qquad \y\in\ovl{W}\times\R^d\times[0,+\infty) \times
  [0,+\infty) , \ \tau\geq \tau_0\sigma,
   \end{align*}
  for $\mu>0$ chosen \suff large. The G\aa rding inequality of
   proposition~\ref{prop: Gaarding quadratic forms local} yields
  \begin{align*}
    T(w)\geq C\|w\|^2_{2,\tau}-C'|\trace(w)|^2_{1,1/2,\tau},
  \end{align*}
  for some $C, C'>0$ and for $\tau\geq \k_0\sigma$ chosen \suff large).
\end{proof}

\subsubsection{Case (iii): both roots lying in the lower complex half-plane.}
\label{sec: Case 2: two roots down}

The result in the present case is a simple consequence of the general
result given in Lemma~\ref{el1} whose proof can be found in
\cite{ML}. In the second order case however, the proof does not require
the same level of technicality.

One has $\Im \pi_{j,1}(\y^{0\prime}) <0$ and  $\Im \pi_{j,2}(\y^{0\prime}) <0$. 
As the roots $\pi_{j,1}$ and $\pi_{j,2}$ depend continuously on the
variable $\y' = (x,\xi',\tau,\sigma)$, there exists $\U$ a conic open
neighborhood of $\y^{0\prime}$ in
$W\times\R^{d-1}\times[0,+\infty) \times[0,+\infty) $ and $C_0>0$ such that
\begin{align*}
   \tau \geq \k_0 \sigma, \quad \Im\pi_{j,1}(\y')\leq - C_ 0 \lsct,
   \quad \text{and}\
   \Im\pi_{j,2}(\y')\leq -C_0 \lsct,
 \end{align*}
if $\y'= (x,\xi',\tau,\sigma)\in\ovl{\U}$.

Let $\chi\in \Ssct^0$ be as in the statement of Proposition~\ref{sl3} and
set $u=\Opt(\chi)v$.

We recall that $\Qconj^j=\Qs^j+i\Qa$,  yielding 
\begin{align*}
  \Norm{\Qconj^j u}{+}^2
  =\Norm{\Qs^j u}{+}^2 +\Norm{\Qa u}{+}^2
    +2\Re \inp{\Qs^ju}{i\Qa u}_+.
\end{align*}
We set $L(u)=\Norm{\Qs^ju}{+}^2+\Norm{\Qa u}{+}^2$.
This is an  interior differential quadratic form in the sense of Definition~\ref{def: interior quadratic form}. Its principal symbol is given by
\begin{align*}
  \ell(\y)=(q^j_s)(\y)^2+q_a(\y)^2, \qquad \y= (x,\xi,\tau,\sigma).
\end{align*}
For $\varepsilon\in (0,1)$ we write 
\begin{equation}\label{mb1}
  \Norm{\Qconj^ju}{+}^2\geq \varepsilon L(u)
  + 2\Re \inp{\Qs^ju}{i\Qa u}_+.
\end{equation}
For concision we  write $\y = (\y',\xi_d)$ with $\y' = (x,\xi',\tau,\sigma)$. 
The set
\begin{align*}
  \mathscr{L}
  =\{\y = (\y',\xi_d); \y' \in \ovl{\U},\
  \xi_d\in\R, \ \text{and}\ |\xi|^2+\tau^2+\sigma^2=1\}
\end{align*}
is compact recalling that $W$ is bounded.
On $\mathscr{L}$ one has $|\qconj^j(\y)|\geq C>0$. By
homogeneity one has 
\begin{equation}\label{mb2}
  |\qconj^j(\y)|\gtrsim \lsc^2, \qquad
  \y' \in \U, \ \xi_d \in \R, \ \ \text{if} \  \tau \geq \tau_0 \sigma,
\end{equation}
for some $\tau_0>0$. 
Therefore 
\begin{equation}\label{mb3}
  \ell(\y)  \gtrsim \lsc^4,
  \qquad \y' \in \U, \ \xi_d \in \R, \ \ \text{if} \  \tau \geq \tau_0 \sigma.
\end{equation}
By the G\aa rding inequality  of Proposition~\ref{prop: Gaarding quadratic forms} one obtains
\begin{equation}
  \label{mb4}
  \Re L(u)\geq
  C\Normsc{u}{2}^2-C' \norm{\trace(u)}{1,1/2}^2-C_N\Normsc{v}{2,-N}^2,
\end{equation}
for $\tau\geq \k_0\sigma$ chosen \suff
large.

From the proof of Lemma~\ref{lemma: carre Carleman} one has 
\begin{align}
  \label{eq: lemma: carre Carleman}
  2\Re \inp{\Qs^ju}{i\Qa u}_+
  = i \inp{[\Qs^j,\Qa]u}{u}_+
  +\tau\mathcal{A}(u)
\end{align}
with the boundary quadratic form $\mathcal{A}$ given in \eqref{eq:
  boundary quadratic form}--\eqref{mb0}.

On the one hand, one has  $[\Qs^j,\Qa]\in\tau\Dsc^2$ and therefore 
\begin{equation}\label{mb5}
  |\Re \inp{[\Qs^j,\Qa]u}{u}_+|
  \lesssim \tau \Normsc{u}{2,-1}^2
  \lesssim  \tau^{-1}\Normsc{u}{2}^2.
\end{equation}
On the other hand, we have the following lemma that provides a
microlocal positivity property for the  boundary quadratic form $\mathcal{A}$. A proof is given
below.
\begin{lemma}\label{sl4}
There exist $C, C_N$ and $\tau_{0}>0$ such that
\begin{align*}
  \tau \Re\mathcal{A}(u)
  \geq C \normsc{\trace (u)}{1,1/2}^2
  -C_N \normsc{\trace(v)}{1,-N}^2,
  \quad \text{for}\  u = \Opt(\chi)v,
  \end{align*}
 for $\tau\geq \max(\tau_0, \k_0\sigma)$.
\end{lemma}
With \eqref{eq: lemma: carre Carleman}--\eqref{mb5},  and 
Lemma \ref{sl4} one obtains
 \begin{align}\label{mb6}\notag
   2\Re \inp{\Qs^ju}{i\Qa u}_+
   &\geq C \normsc{\trace (u)}{1,1/2}^2
     - C' \tau^{-1} \Normsc{u}{2}^2
     - C_N \normsc{\trace(v)}{1,-N}^2\\
   &\geq C \normsc{\trace (u)}{1,1/2}^2
     - C' \tau^{-1} \Normsc{u}{2}^2
     - C_N' \Normsc{v}{2,-N}^2,
 \end{align}
 with a  trace inequality, for $\tau \geq \k_0 \sigma$ chosen \suff large.

 With \eqref{mb1}, \eqref{mb4}, and \eqref{mb6} one obtains
 \begin{align*}
   \Norm{\Qconj^ju}{+}^2 &\geq
   \varepsilon   C\Normsc{u}{2}^2
   -C'\varepsilon \norm{\trace(u)}{1,1/2}^2
   -C_N\varepsilon\Normsc{v}{2,-N}^2\\
   & \quad + C \normsc{\trace (u)}{1,1/2}^2
     - C' \tau^{-1} \Normsc{u}{2}^2
     - C_N' \Normsc{v}{2,-N}^2.
 \end{align*}
With $\varepsilon$ chosen \suff small and $\tau \geq \k_0 \sigma$
\suff large one obtains for any $N\in\N$
\begin{equation*}
  \Normsc{u}{2} + \normsc{\trace (u)}{1,1/2}
  \lesssim \Norm{\Qconj^ju}{+} +\Normsc{v}{2,-N}.
\end{equation*}
With a commutator argument, as $u=\Opt(\chi)v$ one finds
$ \Norm{\Qconj^ju}{+} \lesssim
  \Norm{\Qconj^j v}{+} +\Normsc{v}{2,-1}
$, yielding estimate~\eqref{eq: microlocal estimate Qconj-3}
and thus concluding the proof of Proposition~\ref{sl3} in Case~(iii).\hfill
\qedsymbol \endproof

 \begin{proof}[ Proof of Lemma \ref{sl4}]
With \eqref{mb0} one sees that it suffices to consider the following boundary
quadratic form 
\begin{align*}
\tilde{\mathcal{A}}(w) &=2 \inp{\partial_d\varphi D_d  w\br}{D_dw\br}_\partial 
  + 2 \inp{\tilde{r} (x,D', d_{x'}\varphi) w\br}{D_dw\br}_\partial  \notag\\
 & \quad + 2 \biginp{ \tilde{r} (x,D', d_{x'} \varphi) D_d
 w\br}{w\br}_\partial \notag\\
 &\quad -  2 \biginp{  \partial_d\varphi
   \big(R(x,D') +(-1)^j \sigma^2-(\tau\partial_d\varphi)^2-r(x,\tau d_{x'}\varphi)\big)
 w\br}{w\br}_\partial,
\end{align*}
in place of $\mathcal A$. It is of type $(1,0)$ in the sense of Definition~\ref{def: boundary quadratic form}.
Its principal symbol is given by $a_0(\y',\xi_d, \xi_d') = (1, \xi_d) A(\y')\, \transp(1, \xi_d)$ with
\begin{align*}
   A (\y') = \begin{pmatrix}
     - (\partial_d\varphi)\big(
       r(x,\xi') + (-1)^j\sigma^2-(\tau\partial_d\varphi)^2-r(x,\tau
       d_{x'}\varphi)\big)\br
     & \tilde{r}(x,\xi',d_{x'}\varphi)\br
     \\
     \tilde{r}(x,\xi',d_{x'}\varphi) \br
     & \partial_d\varphi\br
     \end{pmatrix},
 \end{align*}
 with $\y' = (x,\xi',\tau,\sigma)$.  The associated bilinear symbol
 introduced in \eqref{eq: bilinear
   symbol-boundary quadratic form}
is given by 
\begin{align*}
  \un{\mathcal A}(\y',\z,\z') = \z  A (\y') \, \transp \bar{\z}',\qquad \z=(z_0,
  z_1) \in \C^2, \ \z' =(z_0',z_1') \in \C^2.
\end{align*}
 One computes 
\begin{align*}
  \det A(\y')  = -
  \Big( 
  (\partial_d\varphi)^2 \big(  r(x,\xi') + (-1)^j\sigma^2-(\tau\partial_d\varphi)^2-r(x,\tau
  d_{x'}\varphi)\big)
  + \tilde{r}(x,\xi',d_{x'}\varphi) ^2
  \Big) \br.
\end{align*}
With  Lemma~\ref{lemma: caracterisation Im pi2 <0} one sees that $\Im \pi_{j,2}<0$
is equivalent to having $\det A (\y') > 0$.
We thus have
 \begin{align*}
   \det A (\y') \geq C>0,
   \quad \text{for} \ \y' = (x,\xi',\tau,\sigma)\in
   \mathbb{S}_{\ovl{\U}},
 \end{align*}
 with $\mathbb{S}_{\ovl{\U}}=\{\y'\in
 \ovl{\U};\ \xi_d\in\R,\ |\xi|^2+\tau^2+\sigma^2=1\}$ since
 $\mathbb{S}_{\ovl{\U}}$ is compact. 
  Since $\partial_d\varphi\br\geq C' >0$ then one finds that
 \begin{equation*}
    \Re \un{\mathcal A}(\y',\z,\z)  \geq C (|z_0|^2 + |z_1|^2),
   \quad \y' = (x,\xi',\tau,\sigma)\in\ovl{\U},
   \ \ |(\xi',\tau,\sigma)|= 1.
 \end{equation*}
 By homogeneity one obtains
\begin{equation*}
    \Re \un{\mathcal A}(\y',\z,\z)  \geq C ( \lsct^2 |z_0|^2 + |z_1|^2),
   \quad \y' = (x,\xi',\tau,\sigma)\in\ovl{\U},
   \ \ |(\xi',\tau,\sigma)|\geq 1.
 \end{equation*}
With Proposition~\ref{prop: boundary form -Gaarding tangentiel}, having in mind what is exposed in
Section~\ref{sec: calculus with spectral parameter} since we have
$\tau \geq \k_0 \sigma$ here, one obtains
\begin{align*}
  \Re\mathcal{\tilde{A}}(u)
  \geq C \normsc{\trace (u)}{1,0}^2
  -C_N \normsc{\trace(v)}{1,-N}^2,
  \quad \text{for}\  u = \Opt(\chi)v,
  \end{align*}
 for $\tau\geq \k_0\sigma$ chosen \suff large. 

 Here, we have $\Im \pi_{j,2} <0$ and thus $|\xi'| \lesssim \tau$ by
 Lemma~\ref{lemma: low frequency caracterisation}. Thus one has 
\begin{align*}
  \tau \normsc{\trace (u)}{1,0}^2 \gtrsim \normsc{\trace (u)}{1,1/2}^2
- \normsc{\trace(v)}{1,-N}^2,
\end{align*}
by the microlocal G{\aa}rding inequality, for instance invoking
Proposition~\ref{prop: boundary form -Gaarding tangentiel} for a boundary quadratic form of type $(1,1/2)$. This concludes the proof. 
 \end{proof}

 \section{Local Carleman estimate for the fourth-order operator}
 \label{sec: Local Carleman estimate for the fouth-order operator}
\subsection{A first estimate}

\begin{proposition}
\label{prop: first microlocal estimate P}
  Let $\k_0' > \k_1'> \k_1 > \k_0>0$. Let $x^0 \in \partial \Omega$, with
  $\Omega$ locally given by $\{ x_d>0\}$ and let $W$ be a bounded open
  \nhd of $x^0$ in $\R^d$. Let $\varphi$ be such that
  $\partial_d \varphi \geq C>0$ in $W$ and such that
  $(\Qsig^j, \varphi)$ satisfies the sub-ellipticity condition in
  $\ovl{W}$ for both $j=1$ and $2$.  Let
  $\y^{0\prime} = (x^0, \xi^{0\prime},\tau^0,\sigma^0)$ with
  $(\xi^{0\prime},\tau^0,\sigma^0)\in \R^{d-1}\times [0,+\infty)
  \times [0,+\infty)$ nonvanishing with $\k_1 \sigma^0 \leq \tau^0
  \leq \k_1' \sigma^0$.

  Then, there exists $\U$ a 
  conic neighborhood of $\y^{0\prime}$ in
  $W\times \R^{d-1}\times [0,+\infty) \times [0,+\infty)$ where
  $\k_0 \sigma \leq \tau \leq \k_0' \sigma $  such that if  $\chi \in \Ssct^0$,
  homogeneous of degree $0$ in $(\xi',\tau,\sigma)$ with $\supp (\chi)\subset \U$, there exist $C>0$ and $\tau_0>0$ such that
  \begin{align}\label{eq: microlocal estimate Pconj}
    \tau^{-1/2} \Normsc{\Opt (\chi) v}{4}
    \leq
    C \big(
    \Norm{\Pconj v}{+}
    + \normsc{\trace (v)}{3,1/2}
    +\Normsc{v}{4,-1} \big),
\end{align}
for $\tau\geq \tau_0$,
$\k_0 \sigma \leq \tau \leq \k_0' \sigma $, and $v\in \Cbarc(W_+)$.
\end{proposition}

An important aspect is that here we have $\sigma \gtrsim \tau$; this
explains that only one root of $\pconj$ can lie on the real axis and
thus only one half derivative is lost in this estimate. The proof of
Proposition~\ref{prop: first microlocal estimate P} is based on the
microlocal results of Proposition~\ref{sl3}.

\begin{proof}
  We shall concatenate the estimates of Proposition~\ref{sl3} for
  $\Qconj^1$ and $\Qconj^2$ with the boundary operator $B$ simply given by
  the Dirichlet trace operator, $B u\br = u\br$. 
  
  One has $b (x,\xi) =1$ and $b_\varphi(x,\xi',\xi_d,\tau) = 1$. 
Since $\partial_d \varphi >0$ then $\Im \pi_{j,1} <0$. Thus, either
  $\qconj^{j,+} (x,\xi',\xi_d,\tau)=1$ or
  $\qconj^{j,+} (x,\xi',\xi_d,\tau)=\xi_d-\pi_{j,2}$.  With
  Lemma~\ref{sl1} one sees that the \LS holds for
  $(\Qconj^1, B,\varphi)$ and $(\Qconj^2, B,\varphi)$ at
  $\y^{0\prime}$.

  Proposition~\ref{sl3} thus applies. Let $\U_j$ be the conic \nhd of
  $\y^{0\prime}$ obtained invoking this proposition for $\Qconj^j$,
  for $j=1$ or $2$. 
  In $\U_j$ one has $\tau \geq \k_0 \sigma$. We set
  \begin{align*}
    \U = \U_1 \cap \U_2 \cap \{ \tau \leq \k_0' \sigma\},
  \end{align*}
  and we consider $\chi \in \Ssct^0$, homogeneous of degree $0$ in
  $(\xi',\tau,\sigma)$ with $\supp (\chi)\subset \U$.

  Since in $\U$ one has $\sigma >0$ then $\pi_{1,2}$ and $\pi_{2,2}$
  cannot be both real by Lemma~\ref{lemma: no real double root}.
  Proposition~\ref{sl3} thus implies that we necessarily have the
  following two estimates
  \begin{align}\label{eq: microlocal estimate Q1}
    \tau^{-\ell_1}\Normsc{\Opt (\chi) w}{2}
    \lesssim
    \Norm{\Qconj^1w}{+}
    +\normsc{w\br}{3/2}
    +\Normsc{w}{2,-1},
  \end{align}
  and 
  \begin{align}\label{eq: microlocal estimate Q2}
    \tau^{-\ell_2}\Normsc{\Opt (\chi) w}{2}
    \lesssim
    \Norm{\Qconj^2w}{+}
    +\normsc{w}{3/2}
    +\Normsc{w}{2,-1},
  \end{align}
  with either $(\ell_1,\ell_2) = (1/2,0)$ or  $(\ell_1,\ell_2) = (0,
  1/2)$, for $w\in
\Cbarc(W_+)$ and $\tau \geq \k_0\sigma$ chosen \suff large.

\medskip
Let us assume that $(\ell_1,\ell_2) = (1/2,0)$. The other case can be
treated similarly. Writing $\Pconj =\Qconj^2 \Qconj^1$, with \eqref{eq: microlocal estimate Q2} one has
\begin{align*}
    \Normsc{\Opt (\chi) \Qconj^1 v}{2}
    &\lesssim
    \Norm{ \Pconj  v}{+}
    +\normsc{\Qconj^1 v\br}{3/2}
      +\Normsc{v}{4,-1}\\
    &\lesssim
      \Norm{\Pconj  v}{+}
    +\normsc{\trace (v)}{2,3/2}
      +\Normsc{v}{4,-1}.
  \end{align*}
  Since $[\Opt (\chi), \Qconj^1 ] \in  \Psisc^{1,0}$ one finds
  \begin{align}
    \label{eq: concatenate 1}
    \Normsc{\Qconj^1 \Opt (\chi)  v}{2}
    \lesssim
    \Norm{\Pconj  v }{+}
    +\normsc{\trace (v)}{3,1/2}
    +\Normsc{v}{4,-1}.
  \end{align}
  For $k=0,1$ or $2$, one writes
  \begin{multline*}
    \Norm{\Qconj^1 \Opt (\chi)  D_d^k \Lsct^{2-k} v}{+}
    + \normsc{\trace (\Opt (\chi) D_d^k \Lsct^{2-k} v)}{1,1/2}
    + \Normsc{\Opt (\chi) D_d^k \Lsct^{2-k} v}{2,-1}\\
    \lesssim 
    \Normsc{\Qconj^1 \Opt (\chi)  v}{2}
    + \normsc{\trace (v)}{3,1/2}
    + \Normsc{v}{4,-1},
  \end{multline*}
  since $[\Qconj^1 \Opt (\chi),  D_d^k\Lsct^{2-k} ] \in \Psisc^{4,-1}$.

  Let $\tchi\in \Ssct^0$ be homogeneous of degree zero
in the variable $(\xi',\tau,\sigma)$ and be such that
$\supp(\tchi)\subset\U$  and
$\tchi\equiv 1$ on a \nhd of $\supp(\chi)$.
   With \eqref{eq: microlocal estimate Q1}, from \eqref{eq: concatenate
    1} one thus obtains
  \begin{align*}
    \tau^{-1/2}\Normsc{\Opt(\tchi)\Opt (\chi)  D_d^k\Lsct^{2-k}  v}{2}
    \lesssim
    \Norm{\Pconj  v }{+}
    +\normsc{\trace (v)}{3,1/2}
    +\Normsc{v}{4,-1}.
  \end{align*}
  Since $\Opt(\tchi)\Opt (\chi) D_d^k\Lsct^{2-k}  =\Lsct^{2-k}  D_d^k \Opt (\chi)  \mod
  \Psisc^{2,-1}$ one deduces 
  \begin{align*}
    \tau^{-1/2}\Normsc{D_d^k\Opt (\chi)  v}{2,2-k}
    \lesssim
    \Norm{\Pconj  v}{+}
    +\normsc{\trace (v)}{3,1/2}
    +\Normsc{v}{4,-1}.
  \end{align*}
Using that $k=0,1$ or $2$, the result follows. 
\end{proof}

Consequence of this  microlocal result is the following local result
by means of a patching procedure as for the proof of
Proposition~\ref{prop: local boundary estimate} in Section~\ref{sec: proof prop: local boundary estimate}.
\begin{proposition}
\label{prop: first local estimate P}
  Let $\k_0' > \k_0>0$. Let $x^0 \in \partial \Omega$, with
  $\Omega$ locally given by $\{ x_d>0\}$ and let $W$ be a bounded open
  \nhd of $x^0$ in $\R^d$. Let $\varphi$ be such that
  $\partial_d \varphi \geq C>0$ in $W$ and such that
  $(\Qsig^j, \varphi)$ satisfies the sub-ellipticity condition in
  $\ovl{W}$ for both $j=1$ and $2$.

  Then, there exists $W^0$ a \nhd of $x^0$, $C>0$, $\tau_0>0$ such that
  \begin{align}\label{eq: local estimate Pconj-1}
    \tau^{-1/2} \Normsc{ v}{4}
    \leq
    C \big(
    \Norm{\Pconj v}{+}
    + \normsc{\trace (v)}{3,1/2}\big),
\end{align}
for $\tau\geq \tau_0$,
$\k_0 \sigma \leq \tau \leq \k_0' \sigma $, and $v\in \Cbarc(W_+)$.
\end{proposition}

\subsection{Final estimate}
\label{sec: Carleman final estimate}

Combining the local results of Section~\ref{sec: boundary norm under LS} for
the estimation of the boundary norm under  the \LS condition and the
previous local result without any prescribed boundary condition we
obtain the Carleman estimate of Theorem~\ref{theorem: local Carleman estimate-intro}. For a precise statement we write the following theorem. 
\begin{theorem}[local Carleman estimate for $\Psig$]
  \label{theorem: main Carleman estimate}
  Let $\k_0' > \k_0>0$. Let $x^0 \in \partial \Omega$, with
  $\Omega$ locally given by $\{ x_d>0\}$ and let $W$ be a bounded open
  \nhd of $x^0$ in $\R^d$. Let $\varphi$ be such that
  $\partial_d \varphi \geq C>0$ in $W$ and such that
  $(\Qsig^j, \varphi)$ satisfies the sub-ellipticity condition in
  $\ovl{W}$ for both $j=1$ and $2$.

  Assume that  $(\Psig, B_1, B_2,\varphi)$ satisfies the
  \LS condition of Definition~\ref{def: LS after conjugation} at  $\y'=(x^0,\xi',\tau,\sigma)$
  for all $(\xi',\tau,\sigma) \in \R^{d-1} \times [0,+\infty) \times
  [0,+\infty)$ such that $\tau \geq \k_0\sigma$.
  
  Then, there exists $W^0$ a \nhd of $x^0$, $C>0$, $\tau_0>0$ such that
  \begin{align}\label{eq: local estimate Pconj-final}
    \tau^{-1/2} \Normsc{e^{\tau \varphi} u}{4}
    + \normsc{ \trace( e^{\tau \varphi} u)}{3,1/2}
    \leq
    C \big(
    \Norm{e^{\tau \varphi} \Psig u}{+}
    + \sum\limits_{j=1}^{2}
    \normsc{e^{\tau \varphi} B_{j}v\br}{7/2-k_{j}}
    \big),
\end{align}
for $\tau\geq \tau_0$,
$\k_0 \sigma \leq \tau \leq \k_0' \sigma $, and $u\in \Cbarc(W^0_+)$.
\end{theorem}
The notation of the function space $\Cbarc(W^0_+)$ is introduced in
\eqref{eq: notation Cbarc}. 

For the application of this theorem, one has to design a weight
function that yields the two important properties: sub-ellipticity and
the \LS condition. Sub-ellipticity
is obtained by means of Proposition~\ref{prop: sub-ellipticity recipe}; the
\LS condition by means of Proposition~\ref{prop: LS after conjugation}.

\begin{proof}[Proof of Theorem~\ref{theorem: main Carleman estimate}]
  
  The assumption of the theorem allows one to invoke both
  Propositions~\ref{prop: local boundary estimate} and \ref{prop:
    first local estimate P} yielding the existence of a \nhd $W^0$ of $x^0$
where,  by Proposition~\ref{prop: local boundary estimate},  one has
  \begin{align}
    \label{eq: local estimate Pconj*1}
    \normsc{ \trace(v)}{3,1/2} \lesssim
    \Norm{\Pconj v}{+}
    +\sum\limits_{j=1}^{2}
    \normsc{B_{j, \varphi}v\br}{7/2-k_{j}}
    + \Normsc{v}{4,-1},
  \end{align}
  for $\sigma\geq 0$, $\tau\geq \max(\tau_1, \k_0\sigma)$ for some
  $\tau_1>0$ and $v\in \Cbarc(W_+^0)$. With the Proposition~\ref{prop:
    first local estimate P} one also has 
  \begin{align}
    \label{eq: local estimate Pconj*2}
    \tau^{-1/2} \Normsc{ v}{4}
    \lesssim
    \Norm{\Pconj v}{+}
    + \normsc{\trace (v)}{3,1/2},
  \end{align}
  for $\tau\geq \tau_1'$ and $\k_0 \sigma \leq \tau \leq \k_0' \sigma$ for some
  $\tau_1'>0$.

  \medskip
  Consider $\sigma>0$ and $\tau \geq \max (\tau_1, \tau_1')$ such that
  $\k_0 \sigma \leq \tau \leq \k_0' \sigma$. Combined together 
  \eqref{eq: local estimate Pconj*1} and \eqref{eq: local estimate
    Pconj*2} yield
  \begin{align*}   
    \tau^{-1/2} \Normsc{ v}{4}
    + \normsc{ \trace(v)}{3,1/2}
    \lesssim
    \Norm{\Pconj v}{+}
    +\sum\limits_{j=1}^{2}
    \normsc{B_{j, \varphi}v\br}{7/2-k_{j}}
    + \Normsc{v}{4,-1}.
  \end{align*}
  Since $\Normsc{v}{4,-1} \ll \tau^{-1/2} \Normsc{ v}{4}$ for $\tau$
  large one obtains
  \begin{align*}
    \tau^{-1/2} \Normsc{ v}{4}
    + \normsc{ \trace(v)}{3,1/2}
    \lesssim
    \Norm{\Pconj v}{+}
    + \sum\limits_{j=1}^{2}
    \normsc{B_{j, \varphi}v\br}{7/2-k_{j}}.
  \end{align*}
  If we set $v=e^{\tau\varphi}u$ then the conclusion follows. 
\end{proof}

 \section{Global Carleman estimate and observability}
 \label{sec: Global Carleman estimate}
Using the local Carleman estimate of Theorem~\ref{theorem: main Carleman
  estimate} we prove a global version of this estimate.
This allows us to obtain an observability inequality with observation in some open subset
 $\scrO$ of $\Omega$. 
 In turn in Section~\ref{sec:resolvent, stab} we use this
 latter inequality to obtain a resolvent estimate for the plate
 semigroup generator that allows one to deduce a stabilization result
 for the damped plate equation.

\subsection{A global Carleman estimate}
Assume that the \LS condition of Definition~\ref{def: LS} holds for $(P_0, B_1, B_2)$ on $\partial
\Omega$. 

Let $\scrO_0, \scrO_1, \scrO$ be  open sets such that $\scrO_0 \Subset \scrO_1\Subset \scrO \Subset\Omega$. With
Proposition~3.31 and Remark~3.32 in \cite{JGL-vol1} there exists $\psi
\in \Cinf(\ovl{\Omega})$ such that
\begin{enumerate}
\item $\psi = 0$ and $\partial_\nu \psi <- C_0<0$ on $\partial \Omega$;
  \item $\psi >0$ in $\Omega$;
  \item $d \psi \neq 0$ in $\Omega \setminus \scrO_0$.
\end{enumerate}
Then, by Proposition~\ref{prop: sub-ellipticity recipe}, for $\csp$
chosen \suff large, one finds that $\varphi = \exp(\csp \psi)$ is such
that a
\begin{enumerate}
\item $\varphi = 1$ and $\partial_\nu \varphi <- C_0<0$ on $\partial \Omega$;
  \item $\varphi >1$ in $\Omega$;
  \item $(\Qsig^j,\varphi)$ satisfies the sub-ellipticity condition in
    $\Omega \setminus \scrO_0$, for $j=1,2$, for $\tau \geq \tau_0
    \sigma$ for $\tau_0$ chosen \suff large.
\end{enumerate}
Then, with Proposition~\ref{prop: LS after conjugation}, for $\k_0>0$
 chosen \suff large one finds that the \LS condition holds for
 $(\Psig, B_1, B_2, \varphi)$ at any 
 $(x,\xi',\tau,\sigma)$ for any $x \in \partial \Omega$, $\xi' \in
 T_x^* \partial \Omega \simeq \R^{d-1}$, $\tau>0$, and $\sigma>0$ such that $\tau \geq
 \k_0 \sigma$, for $\k_0$ chosen \suff large, using that $\partial
 \Omega$ is compact.

 Thus for any $x \in \partial \Omega$ the local estimate of
 Theorem~\ref{theorem: main Carleman estimate} applies.
 A similar result applies in the \nhd of any point of $\Omega
 \setminus \scrO_0$.

 With the weight function $\varphi$
 constructed above, following the patching procedure described in the proof of Theorem~3.34 in
 \cite{JGL-vol1}, one obtains the following global estimate
\begin{align}\label{eq: global estimate Pconj}
    \tau^{-1/2} \Normsc{e^{\tau \varphi} u}{4}
    + \normsc{ \trace( e^{\tau \varphi} u)}{3,1/2}
    \lesssim
    \Norm{e^{\tau \varphi} \Psig u}{L^2(\Omega)}
    + \sum\limits_{j=1}^{2} \normsc{e^{\tau \varphi} B_{j}u\bd}{7/2-k_{j}}
    + \tau^{-1/2} \Normsc{e^{\tau \varphi} \chi_0 u}{4},
\end{align}
for $\tau\geq \tau_0$,
$\k_0 \sigma \leq \tau \leq \k_0' \sigma $, and $u\in
\Cinf(\ovl{\Omega})$, and where $\chi_0 \in \Cinfc(\scrO)$ such that
$\chi_0 \equiv 1$ in a \nhd of $\ovl{\scrO_1}$.
Here, $\Normsc{.}{s}$ and $\normsc{.}{s}$, the Sobolev norms  with the large parameter $\tau$,
are understood
in $\Omega$ and $\partial\Omega$ respectively.

\begin{remark}
  \label{remark: global estimate third-order perturbation}
  Observe that inequality~\eqref{eq: global estimate Pconj} also holds
  for third-order perturbations of $\Psig$. Below, we shall use it for
  a second-order perturbation $\Psig-i \sigma^2 \alpha = \Delta^2 -
  \sigma^4 -i \sigma^2 \alpha $.
\end{remark}

\subsection{Observability inequality}

By density one finds that inequality~\ref{eq: global estimate Pconj}
holds for $u \in H^4(\Omega)$.

Let $ C_0 > \sup_{\ovl{\Omega}}\varphi-1$. Since $1 \leq \varphi
\leq \sup_{\ovl{\Omega}}\varphi$ one obtains
\begin{align}\label{eq: global estimate Pconj-2}
    \Norm{u}{H^4(\Omega)}
    \lesssim
    e^{C_0 \tau} \big( \Norm{\Psig u}{L^2(\Omega)}
    + \sum\limits_{j=1}^{2} \norm{ B_{j}u\bd}{H^{7/2-k_j}(\partial \Omega)}
  + \Norm{u}{H^4(\scrO_1)}
  \big).
\end{align}
 for $\tau\geq \tau_0$,
$\k_0 \sigma \leq \tau \leq \k_0' \sigma $.

With the ellipticity of $P_0$ one has
\begin{align*}
  \Norm{u}{H^4(\scrO_1)}
  \lesssim \Norm{P_0 u}{L^2(\scrO)}
  + \Norm{u}{L^2(\scrO)},
\end{align*}
since $\scrO_1 \Subset \scrO$. This can be proven by the introduction
of a parametrics for $P_0$.
One thus obtain
\begin{align*}
  \Norm{u}{H^4(\scrO_1)}
  \lesssim \Norm{\Psig u}{L^2(\Omega)}
  + (1+ \sigma^4) \Norm{u}{L^2(\scrO)},
\end{align*}
and thus with \eqref{eq: global estimate Pconj-2} one obtains the
following observability result.
\begin{theorem}[observability inequality]
  \label{th: observation}
  Let $\Psig = \Delta^2 - \sigma^4$ and 
  let $B_1$ and $B_2$ be two boundary operators of order $k_1$ and
  $k_2$ as given in Section~\ref{sec: LS for bi-Laplace}.
  Assume that the \LS condition of Definition~\ref{def: LS} holds. 
  Let $\scrO$ be an open set of $\Omega$. There exists $C>0$ such that
  \begin{align*}
    \Norm{u}{H^4(\Omega)} \leq C
    e^{C|\sigma|} \big(
    \Norm{\Psig u}{L^2(\Omega)}
    + \sum\limits_{j=1}^{2} \norm{B_j u\bd}{H^{7/2-k_j}(\partial\Omega)}
    + \Norm{u}{L^2(\scrO)}
    \big),
  \end{align*}
  for $u \in H^4(\Omega)$.
\end{theorem} 

\begin{remark}
  \label{remark: observability third-order perturbation}
With Remark~\ref{remark: global estimate third-order perturbation} the
result of Theorem~\ref{th: observation} hold for $\Psig= \Delta^2 -\sigma^4$ replaced by $\Psig-i \sigma^2 \alpha = \Delta^2 -\sigma^4 -i \sigma^2 \alpha $.
\end{remark}

\section{Solutions to the damped plate equations}
\label{sec: Solutions to the damped plate equations}
Here, we review some aspects of the solutions of the damped plate
equation whose form we recall from the introduction:
\begin{equation}
  \label{eq: damped-plate equation}
\begin{cases}
\partial^2_ty +P y+\alpha(x)\partial_t y =0 
&(t,x) \in  \R_+ \times\Omega,\\
B_1 y_{|\R_+ \times\partial \Omega}= B_2u_{|\R_+ \times\partial \Omega}=0,\\
y_{|t=0}=y^0, \ \ \partial_t y_{|t=0}=y^1,
\end{cases}
\end{equation}
where $P = \Delta^2$ and $\alpha\geq 0$, positive on some open subset
of $\Omega$. The boundary operators $B_1$
and $B_2$ of orders $k_j$, $j=1,2$, less than or equal to $3$ in the normal direction are chosen so that
\begin{enumerate}
\item[(i)] 
  the \LS condition of Definition~\ref{def: LS} is fulfilled for $(P,
  B_1, B_2)$ on $\partial\Omega$;
\item[(ii)] 
  the operator $P$ is symmetric under homogeneous boundary
  conditions, that is, 
  \begin{align}
    \label{eq: symmetry property}
    \inp{Pu }{v}_{L^2(\Omega)} = \inp{u}{P v}_{L^2(\Omega)}, 
  \end{align}
  for $u, v \in H^4(\Omega)$ such that $B_j u_{|\partial \Omega} = B_j v_{|\partial \Omega} = 
  0$ on $\partial\Omega$, $j=1,2$.  Examples of such conditions are
  given in Section~\ref{sec: Examples  of boundary operators yielding symmetry}. 
\end{enumerate}
With the assumed \LS condition the operator 
\begin{align}
  \label{eq: Fredholm operator}
  L: H^4 (\Omega) &\to L^2(\Omega) \oplus H^{7/2-k_1} (\partial \Omega)
  \oplus H^{7/2-k_2} (\partial \Omega), \notag\\
  u & \mapsto (Pu, B_1 u_{|\partial \Omega} , B_2 u_{|\partial \Omega}),
\end{align}
is Fredholm.

\begin{enumerate}
  \item[(iii)]
We shall further assume that the Fredholm  index of the operator $L$ is zero.
\end{enumerate}

The previous symmetry property gives
$\inp{Pu}{u}_{L^2(\Omega)} \in \R$. We further assume the following
nonnegativity property:
\begin{enumerate}
\item[(iv)] For $u \in H^4(\Omega)$ such that
  $B_j u_{|\partial \Omega} = 0$ on $\partial\Omega$, $j=1,2$ one has
  \begin{align}
    \label{eq: nonnegativity assumption P}
    \inp{Pu}{u}_{L^2(\Omega)}\geq 0.
    \end{align}
  \end{enumerate}
 This last property is very natural to define a nonnegative energy for the
     plate equation given in \eqref{eq: damped-plate equation}.

     We first review some properties of the unbounded operator
     associated with the bi-Laplace operator and the two homogeneous
     boundary conditions based on the assumptions made here. Second,
     the well-posedness of the plate equation is reviewed by means of
     the a semigroup formulation. This semigroup formalism is also
     central in the stabilization result in Sections~\ref{sec: Resolvent estimate}--\ref{sec: Stabilization result}.

\subsection{The unbounded operator associated with the bi-Laplace
  operator}
\label{sec: unbounded bi-Laplace operator}

Associated with $P$ and the boundary operators $B_1$
and $B_2$ is the operator $(\Pell, D(\Pell))$ on $L^2(\Omega)$, with domain
\begin{align*}
  D(\Pell) = \big\{ u \in L^2(\Omega); \ P u \in L^2(\Omega), \ B_1
  u_{|\partial \Omega}
  = B_2 u_{|\partial \Omega}=0\big\},
\end{align*}
and given by $\Pell u = P u \in  L^2(\Omega)$ for $u \in  D(\Pell)$.
The definition of $D(\Pell)$ makes sense since having $P u \in
L^2(\Omega)$ for $u \in L^2(\Omega)$ implies that the traces $\partial_\nu^k
u_{|\partial \Omega}$ are well defined for $k=0, 1,2,3$. 

Since the \LS condition holds on $\partial \Omega$ one has $D(\Pell)
\subset H^4(\Omega)$ (see for instance Theorem~20.1.7 in
\cite{Hoermander:V3}) and thus one can also write $D(\Pell)$ as in
\eqref{eq: domain P intro}. 
From the assumed nonnegativity in \eqref{eq: nonnegativity assumption P} above one finds that
$\Pell + \id$ is injective. Since the operator
\begin{align*}
  L': H^4 (\Omega) &\to L^2(\Omega) \oplus H^{7/2-k_1} (\partial \Omega)
  \oplus H^{7/2-k_2} (\partial \Omega)\\
  u & \mapsto (Pu +u, B_1 u_{|\partial \Omega} , B_2 u_{|\partial \Omega})
\end{align*}
is Fredholm and has the same zero index as $L$ defined in \eqref{eq:
  Fredholm operator}, one finds that $L'$ is surjective. Thus
$\range(\Pell + \id) = L^2(\Omega)$.
One thus concludes that $\Pell$ is maximal monotone. 
From the assumed symmetry property \eqref{eq: symmetry property} and
one finds that   $\Pell$ is selfadjoint, using
that a symmetric maximal monotone operator is selfadjoint (see for
instance Proposition~7.6 in \cite{Brezis:11}).

The resolvent of $\Pell +\id$ being compact on $L^2(\Omega)$,  $\Pell$ has a sequence of
eigenvalues with finite multiplicities. With the assumed nonnegativity
\eqref{eq: nonnegativity assumption P}   they take the form of a sequence 
\begin{align*}
0\leq \mu_0 \leq \mu_1 \leq \cdots \leq \mu_k \leq \cdots 
\end{align*}
that grows to $+\infty$. Associated with this sequence is
$(\phi_j)_{j \in \N}$ a Hilbert basis of $L^2(\Omega)$. 
Any $u \in L^2(\Omega)$ reads $u = \sum_{j \in \N} u_j \phi_j$, with
$u_j = \inp{u}{\phi_j}_{L^2(\Omega)}$.
We define the Sobolev-like scale 
\begin{align}
  \label{eq: Sobolev scale}
  H^k_B (\Omega)= \{ u \in L^2(\Omega); (\mu_j^{k/4} u_j)_j \in \ell^2(\C)\} 
  \quad \text{for} \ k\geq 0.  
\end{align}
One has $D(\Pell) = H^4_B (\Omega)$ and $L^2(\Omega) = H^0_B(\Omega)$. Each $H^k_B (\Omega)$, $k \geq 0$, is equipped with
the inner product and norm
\begin{align*}
  \inp{u}{ v}_{H^k_B (\Omega)} 
  = \sum_{j \in \N} (1 + \mu_j )^{k/2} u_j \ovl{v_j}.
  \qquad \Norm{u}{H^k_B (\Omega)}^2 = \sum_{j \in \N} (1 + \mu_j )^{k/2}  |u_j|^2,
\end{align*}
yielding a Hilbert space structure. 
The space $H^k_B (\Omega)$ is dense in $H^{k'}_B (\Omega)$ if $0\leq
k' \leq k$ and the injection is compact.
Note that one uses $(1 + \mu_j)^{k/2}$ in place of $\mu_j^{k/2}$ since
$\ker(\Pell)$ may not be trivial. Note that if $k=0$ one recovers the
standard $L^2$-inner product and norm.

Using $L^2(\Omega)$ as a pivot space, for $k >0$ we also define the space
$H^{-k}_B(\Omega)$ as the dual space of $H^k_B(\Omega)$. One
finds that any $u \in H^{-k}_B(\Omega)$ takes  the form of the
following limit of $L^2$-functions
\begin{align*}
  u = \lim_{\ell \to \infty} \sum_{j=0}^\ell u_j \phi_j, 
\end{align*}
for some $(u_j)_j \subset  \C$ such that $\big((1 + \mu_j )^{-k/4} u_j\big)_j   \in
\ell^2(\C)$, with the limit occurring in
$\big(H^k_B(\Omega)\big)'$ with the natural dual strong topology. Moreover, one has $u_j =
\dup{u}{\ovl{\phi_j}}_{H^{-k}_B, H^{k}_B}$.
If $u = \sum_{j \in \N} u_j \phi_j  \in H^{-k}_B(\Omega)$ and $v =
\sum_{j \in \N} v_j \phi_j  \in H^{k}_B(\Omega)$ one finds
\begin{align*}
  \dup{u}{\ovl{v}}_{H^{-k}_B, H^{k}_B} = \sum_{j \in \N} u_j \ovl{v_j}.
\end{align*}

One can then extend (or restrict) the action of $\Pell$ on any space
$H^k_B(\Omega)$, $k \in \R$. One has $\Pell: H^k_B(\Omega) \to
H^{k-4}_B(\Omega)$ continuously with 
\begin{align}
  \label{eq: action P}
  \Pell u  = \sum_{j\in \N}\mu_j u_j \phi_j,  
  \quad \text{with convergence in} \ H^{k-4}_B(\Omega) \ \text{for} \ u = \sum_{j\in \N}u_j \phi_j
  \in H^k_B(\Omega).
\end{align} 
In particular, for $u \in H^4_B(\Omega)= D(\Pell)$ and $v \in H^{2}_B(\Omega)$
one has
\begin{align}
  \label{eq: pseudo-norm}
  \inp{\Pell u }{v}_{L^2(\Omega)} = \dup{\Pell u }{\bar{v}}_{H^{-2}_B,
  H^{2}_B}
  = \sum_{j \in \N} \mu_j u_j \ovl{v_j}
\end{align} 
and if $u,v \in H^{2}_B(\Omega)$ one has 
\begin{align}
  \label{eq: norm H2B}
  \inp{u}{v}_{H^{2}_B(\Omega)} = \inp{u }{v}_{L^2(\Omega)} 
  +  \dup{\Pell u }{\bar{v}}_{H^{-2}_B, H^{2}_B}
  = \sum_{j \in \N} (1+\mu_j) u_j \ovl{v_j}.
\end{align}
Note that 
\begin{align}
  \label{eq: norm H2B-bis}
 \dup{\Pell u }{\bar{v}}_{H^{-2}_B, H^{2}_B}
 = \inp{\Pell^{1/2} u}{\Pell^{1/2} v}_{L^2(\Omega)},
\end{align}
with the operator $\Pell^{1/2}$ easily defined by means of the Hilbert
basis $(\phi_j)_{j \in \N}$. In fact, $H^{2}_B(\Omega)$ is the domain
of $\Pell^{1/2}$ viewed as un unbounded operator on $L^2(\Omega)$.

We make the following observations. 
\begin{enumerate}
\item 
If $\ker (\Pell) = \{0\}$ then  
\begin{align*}
  (u,v) \mapsto \dup{\Pell u }{\bar{u}}_{H^{-2}_B, H^{2}_B},
\end{align*}
is also an inner-product on $H^{2}_B(\Omega)$, that yields an
equivalent norm.

\item 
If $0$ is an eigenvalue, that is, 
$\dim \ker(\Pell)= n\geq 1$ then $(\phi_0, \dots, \phi_{n-1})$ is a
orthonormal basis of $\ker(\Pell)$ for the $L^2$-inner product. 
From a classical unique continuation
property, since $\alpha(x) >0$ for $x$ in an open subset of $\Omega$ one sees that
\begin{align}
  \label{eq: second inner product ker P}
  (u,v) \mapsto \inp{\alpha u}{v}_{L^2(\Omega)}
\end{align} is also an inner product on
the finite dimensional space $\ker(\Pell) \subset L^2(\Omega)$. 
We introduce a second basis $(\varphi_0, \dots, \varphi_{n-1})$ of  $\ker(\Pell)$ orthonormal with respect to this second
inner product. 
\end{enumerate}
In what follows, we treat the more difficult case where $\dim
\ker(\Pell)= n\geq 1$. The case $\ker (\Pell)=\{0\}$ is left to the
reader.

\subsection{The plate semigroup generator}
\label{sec: The plate semigroup generator}
Set $\H = H^2_B (\Omega) \oplus L^2(\Omega)$ with natural inner
product and norm
\begin{align}
  \label{eq: norm H}
  &\biginp{(u^0,u^1)}{(v^0,v^1)}_{\H} = \inp{u^0}{v^0}_{H^2_B(\Omega)} +
  \inp{u^1}{v^1}_{L^2(\Omega)}, \\
  &\bigNorm{(u^0,u^1)}{\H}^2 = \Norm{u^0}{H^2_B(\Omega)}^2 + \Norm{u^1}{L^2(\Omega)}^2.
\end{align}

Define the unbounded operator
\begin{align}
  \label{eq: generator damped plate}
  A = \begin{pmatrix} 0 & -1\\ \Pell & \alpha(x)  \end{pmatrix},
\end{align}
on $\H$ 
with domain given by $D(A) = D(\Pell) \oplus H^2_B(\Omega)$.
This domain is dense in $\H$ and $A$ is a closed operator. One has 
\begin{align*}
  \mathcal N = \ker  (A) = \big\{ \transp (u^0, 0); \ u^0 \in \ker
  (\Pell)\big\}.
\end{align*} 

The important result of this section is the following proposition.
\begin{proposition}
  \label{prop: A semigroup generator}
  The operator $(A, D(A))$ generates a bounded semigroup $S(t) = e^{-t
  A}$ on $\H$. 
\end{proposition}

The understanding of this generator property relies on the
introduction of a reduced function space associated with
$\ker (\Pell)$, following for instance the analysis of
\cite{LR:97}. It will be also important in the derivation of a precise
resolvent estimate in Section~\ref{sec: Resolvent estimate}.  If
$\ker(\Pell) = \{0\}$, that is, $\mu_0 >0$, this procedure is not
necessary.  For $v \in \ker(\Pell)$, $v\neq 0$, we introduce the
linear form
\begin{align}
  \label{eq: linear form ker(P)}
  F_{v}: \H &\to \C\\ \notag
(u^0, u^1 )
  &\mapsto
    \inp{\alpha v}{v}_{L^2(\Omega)}^{-1} 
    \big( 
    \inp{\alpha u^0}{v}_{L^2(\Omega)}
    + \inp{u^1}{v}_{L^2(\Omega)}
    \big),
 \end{align}
We set 
\begin{align}
   \label{eq: defintion dot H}
  \dot{\H} = \bigcap_{{v \in \ker(\Pell)} \atop {v \neq 0}} \ker (F_{v})
  = \bigcap_{0\leq j \leq n-1} \ker (F_{\varphi_j}),
\end{align}
with the basis $(\varphi_0, \dots,
\varphi_{n-1})$ of $\ker(\Pell)$ introduced above. 
If $(v,0) \in \ker (A)$, with $0 \neq
 v \in \ker(\Pell)$,
note that  $F_{v} (v,0) = 1$. 
We set $\Theta_j = \transp (\varphi_j,0)$, $j=0, \dots, n-1$ and
\begin{align*}
  \Pi_{\mathcal N} V = \sum_{j=0}^{n-1}  F_{\varphi_j} (V)
  \Theta_j, 
  \qquad \text{for} \ V \in \H,
\end{align*}
and $\Pi_{\dot{\H}} = \id_{\H} - \Pi_{\mathcal N}$. We obtain that 
$\Pi_{\mathcal N}$ and $\Pi_{\dot{\H}}$ are {\em continuous} projectors associated with
the direct sum
\begin{align}
  \label{eq: decomposition H}
  \H = \dot{\H} \oplus \mathcal N 
  \ \ \text{and} \ \ 
  \dot{\H}= \ker (\Pi_{\mathcal N}).
\end{align}
Note that
  $\dot{\H}$ and ${\mathcal N}$ are {\em not orthogonal} in $\H$. 
Yet, it is important to note the following result.
\begin{lemma}
  \label{lemma: stability dot H}
  We have $\range(A) \subset \dot{\H}$.
\end{lemma}
\begin{proof}
  Let $U = \transp(u^0,u^1)= A V$ with $V= \transp(v^0,v^1) \in
  D(A)$. One has $u^0 = -v^1 \in  H^2_B (\Omega) $ and $u^1 = \Pell
  v^0 + \alpha v_1 \in
  L^2(\Omega)$. If $0 \neq \varphi \in \ker(\Pell)$ one writes 
  \begin{align*}
    \inp{\alpha\varphi}{\varphi}_{L^2(\Omega)}
    F_{\varphi} (U) &= 
    \inp{-\alpha v^1}{\varphi}_{L^2(\Omega)}
    + \inp{\Pell v^0 + \alpha v^1}{\varphi}_{L^2(\Omega)}\\
    &=\inp{\Pell v^0}{\varphi}_{L^2(\Omega)}
      = \inp{v^0}{\Pell \varphi}_{L^2(\Omega)}=0.    
  \end{align*}
  using that $v^0, \varphi \in D(\Pell)$, that $(\Pell, D(\Pell))$ is
  selfadjoint, and that $\varphi \in \ker(\Pell)$. The conclusion
  follows from the definition of $\dot{\H}$ in \eqref{eq: defintion dot H}. 
\end{proof}
The space $\dot{\H}$ inherits the natural inner product and  norm of
$\H$ given in \eqref{eq: norm H}.
Yet one finds that the inner product 
\begin{align}
  \label{eq: inner product dot H}
  \inp{(u^0,u^1)}{(v^0,v^1)}_{\dot \H} 
  = \dup{\Pell u^0}{\ovl{v^0}}_{H^{-2}_B,H^2_B} 
  + \inp{u^1}{v^1}_{L^2(\Omega)}, 
\end{align}
and associated norm
\begin{align}
  \label{eq: norm dot H}
  \Norm{(u^0,u^1)}{\dot \H}^2 
  = \dup{\Pell u^0}{\ovl{u^0}}_{H^{-2}_B,H^2_B} + \Norm{u^1}{L^2(\Omega)}^2, 
\end{align}
yields an equivalent norm on $\dot \H$ by a Poincar\'e-like argument.

We introduce the unbounded operator
$\dot{A}$ on $\dot{\H}$ given by the domain
  $D(\dot{A}) = D(A) \cap \dot{\H}$ 
and such that $\dot{A} V = A V$ for $V \in D(\dot{A})$. We then have
$ A= \dot{A}\circ \Pi_{\dot \H}$.
Observe that $D(\dot{A}) = \Pi_{\dot \H} \big( D(A)\big)$ since
$\mathcal N = \ker(A) \subset D(A)$.   Thus, 
one has
\begin{align}
  \label{eq: decomposition D(A)}
  D(A) =D(\dot{A}) \oplus {\mathcal N}.
\end{align}
As for the decomposition of $\H$ given in \eqref{eq: decomposition H}
note that $D(\dot{A})$ and ${\mathcal N}$ are not orthogonal. 

\begin{lemma}
  \label{lemma: a priori resolvent estimate}
  Let $z \in \C$ 
  be such that $\Re z<0$. We have 
  \begin{equation*}
    \Norm{(z \id_{\dot{\H}} - \dot A) U}{\dot{\H}} \geq |\Re z| \,
    \Norm{U}{\dot{\H}}, \quad U \in D(\dot A).
  \end{equation*}
\end{lemma}
The proof of this lemma is quite classical. It is given in
Appendix~\ref{sec: lemma: a priori resolvent estimate}.

With the previous lemma, with the Hille-Yosida theorem one
proves the following result. 
\begin{lemma}
  \label{lemma: reduced generator}
  The operator $(\dot A, D(\dot A))$ generates a semigroup of
  contraction 
  $\dot S(t) = e^{- t \dot A}$ on $\dot{\H}$. 
\end{lemma}
If we set 
\begin{align}
  \label{eq: semigroup}
    S(t) = \dot{S}(t) \circ \Pi_{\dot{\H}}  + \Pi_{\mathcal N},
  \end{align}
we find that $S(t)$ is a semigroup on $\H$ generated by $(A,D(A))$,
thus proving Proposition~\ref{prop: A semigroup generator}. 
If $Y^0 \in D(A)$, the solution of the semigroup equation
  $\frac{d}{d t} Y(t) + A Y(t) =0$ reads
  \begin{align}
    \label{eq: semigroup solution}
    Y(t)= S(t) Y^0= \dot{S}(t) \circ \Pi_{\dot{\H}}  Y^0 +
    \Pi_{\mathcal N} Y^0.
  \end{align}
  We set $\dot{Y}(t) = \Pi_{\dot{\H}}  Y(t) = \dot{S}(t) \circ \Pi_{\dot{\H}}  Y^0$.

\bigskip
The adjoint of $\dot A$ has domain $D(\dot A^*) = D(A)$ and is given
by 
\begin{align*}
  \dot A^* = \begin{pmatrix} 0 & 1\\ -\Pell & \alpha(x)  \end{pmatrix}.
\end{align*}
Similarly to Lemma~\ref{lemma: a priori resolvent estimate} one has
the following result with a similar proof.
\begin{lemma}
  \label{lemma: a priori resolvent estimate bis}
  Let $z \in \C$ 
  be such that $\Re z<0$. We have 
  \begin{equation*}
    \Norm{(z \id_{\dot{\H}} - \dot A^*) U}{\dot{\H}} \geq |\Re z| \,
    \Norm{U}{\dot{\H}}, \quad U \in D(\dot A^*) = D(\dot A).
  \end{equation*}
\end{lemma} 

\subsection{Strong and weak solutions to the damped plate equation}
\label{sec: Strong and weak solutions}
For $y(t)$ a solution to the damped plate equation~\eqref{eq:
  damped-plate equation} one has $Y(t) = \transp(y(t), \partial_t y(t))$
formally solution to $\frac{d}{d t} Y(t) + A Y(t) =0$ and conversely. 

The semigroup $S(t)$ generated by $A$ as given by
Proposition~\ref{prop: A semigroup generator} allows one to go beyond
this formal observation and one obtains the
following well-posedness result for strong solutions of the damped
plate equation.
\begin{proposition}[strong solutions of the damped plate equation]
  \label{prop: strong solutions plate equation}
  For $(y^0, y^1) \in H^4_B(\Omega)\times H^2_B(\Omega)$ there exists a unique 
  \begin{align*}
    y \in \Con^0 \big([0,+\infty); H^4_B(\Omega)\big)
    \cap \Con^1 \big([0,+\infty); H^2_B(\Omega) \big) \cap
  \Con^2\big([0,+\infty); L^2(\Omega)\big) 
  \end{align*}
  such that 
  \begin{align}
    \label{eq: damped plate equation-strong solution}
    \partial_t^2 y + P y + \alpha \partial_t y =0 \quad \text{in} \
    L^\infty([0,+\infty); L^2(\Omega)), \qquad y_{|t=0}  = y^0, \ \partial_t y_{|t=0}  = y^1.
  \end{align}
  Moreover, there exists $C>0$ such that 
  \begin{align}
    \label{eq: continuity strong solution damped plate equation-LS}
    \Norm{y (t)}{H^4_B(\Omega)} +  \Norm{\partial_t y (t)}{H^2_B(\Omega)}
    \leq C  \big( \Norm{y^0}{H^4_B(\Omega)} +
    \Norm{y^1}{H^2_B(\Omega)}\big), \qquad t\geq 0.
  \end{align}
\end{proposition}
With $Y(t)$ as above, for such a solution $y(t)$ one has 
\begin{align*}
  \frac{d}{d t} Y(t) + A Y(t) =0, \qquad Y(0) = Y^0 = \transp(y^0,y^1), 
\end{align*}
that is,
\begin{align*}
  Y(t) = S(t) Y^0\in
  \Con^0 \big([0,+\infty); D(A) \big)
  \cap \Con^1 \big([0,+\infty); H^2_B(\Omega) \oplus L^2(\Omega) \big).
  \end{align*}

A weak solution to the damped plate equation is simply associated with
an initial data $(y^0, y^1) \in H^2_B(\Omega) \times L^2(\Omega)$ and
given by the first coordinate of $Y(t) = S(t) Y^0$. Then one has
\begin{align*}
   Y(t) \in
  \Con^0 \big([0,+\infty); \H \big)
  \cap \Con^1 \big([0,+\infty); L^2(\Omega)  \oplus H^{-2}_B(\Omega) \big).
\end{align*}
or equivalently
\begin{align*}
  y \in \Con^0 \big([0,+\infty); H^2_B(\Omega)\big)
  \cap \Con^1 \big([0,+\infty); L^2(\Omega) \big)
  \cap\Con^2\big([0,+\infty); H^{-2}_B(\Omega)\big) .
  \end{align*}

  \medskip
  For a strong solution, the natural energy is given by
  \begin{align}
    \label{eq: energy}
    \E( y) (t) = \frac12 \big(
    \Norm{\partial_t y (t)}{L^2(\Omega)}^2
    + \inp{\Pell y (t)}{y (t)}_{L^2(\Omega)}
    \big).
  \end{align}
  Observe that if $y^0 \in \ker(\Pell)$ then $y(t) = y^0$ is solution
  to \eqref{eq: damped-plate equation} with $y^1=0$. This is
  consistent with the form of the semigroup $S(t)$ given in \eqref{eq:
    semigroup}. Such a solution is independent of the evolution
  variable $t$, and thus, despite damping, there is no decay. However,
  note that such a solution is `invisible' for the energy defined in
  \eqref{eq: energy}. In fact, for a strong solution to \eqref{eq:
    damped-plate equation}  as given by Proposition~\ref{prop: strong solutions plate equation} one has
  \begin{align}
    \label{eq: connection energy norm H dot}
     \E( y) (t) = \frac12 \Norm{\dot Y(t)}{\dot \H}^2,
  \end{align}
  with $\dot Y(t)$ as defined below
  \eqref{eq: semigroup solution} and $\Norm{.}{\dot \H}$ defined in
  \eqref{eq: norm dot H}. For a strong solution, we write
  \begin{align*}
    \frac{d}{d t} \E( y) (t)
    &= \Re \inp{\partial_t y(t) }{\partial_t^2 y(t)}_{L^2(\Omega)}
    + \frac12 \dup{\Pell \partial_t y(t)}{\ovl{y(t)}}_{H^{-2}_B, H^2_B}
    + \frac12 \inp{\Pell  y(t)}{\partial_t y(t)}_{L^2(\Omega)}\\
    &= \Re \inp{\partial_t y(t) }{( \partial_t^2 + \Pell) y(t)}_{L^2(\Omega)}
    = - \Re \inp{\partial_t y(t) }{\alpha\partial_t  y(t)}_{L^2(\Omega)} \leq 0
  \end{align*}
  since $\alpha \geq 0$. Thus, the energy of a strong solution is nonincreasing.
  To understand the decay of the energy one has to focus on the
  properties of the semigroup $\dot S(t)$ and its generator $(\dot A,
  D(\dot A))$ on $\dot \H$. This is done in Section~\ref{sec: Resolvent estimate}.

  \medskip
  For a weak solution $y(t) \in \Con^0 \big([0,+\infty); H^2_B(\Omega)\big)
  \cap \Con^1 \big([0,+\infty); L^2(\Omega) \big)$ the energy is
  defined by
  \begin{align*}
    \E( y) (t) = \frac12 \big(
    \Norm{\partial_t y (t)}{L^2(\Omega)}^2
    + \dup{\Pell y (t)}{\ovl{y (t)}}_{H^{-2}_B, H^2_B}
    \big)
  \end{align*}
  that coincides with \eqref{eq: energy} for a strong solution.
  The stabilization result we are interested in only concerns strong
  solutions (see Section~\ref{sec: Stabilization result}). Thus, we shall
  not mention weak solutions in what follows.

 \section{Resolvent estimates and applications to
   stabilization}
 \label{sec:resolvent, stab}
 Here we use the observability inequality of Theorem~\ref{th: observation}
to obtain a resolvent estimate for the plate
 semigroup generator that allows one to deduce a stabilization result
 for the damped plate equation. This a sequence of argument comes from
 the seminal works of Lebeau~\cite{Lebeau:96} and
 Lebeau-Robbiano~\cite{LR:97}.

\subsection{Resolvent estimate}
\label{sec: Resolvent estimate}
We prove a resolvent estimate for the unbounded operator $(\dot{A}, D (\dot{A}))$ that
acts on $\dot{\H}$. First, we establish that $\{ \Re z \leq 0\}$ lies in the
resolvent set of $\dot{A}$.
\begin{proposition}
  \label{prop: position spectrum boundary damping}
  The spectrum of $(\dot{A}, D (\dot{A}))$ is contained in $\{ z \in \C; \Re (z) >0\}$.
\end{proposition}
The proof of this proposition is rather classical based on a unique
continuation argument and a Fredholm index argument for a compact
perturbation. It is given in Appendix~\ref{sec: prop:
  position spectrum boundary damping}.

\begin{theorem}\label{res}
Let $\scrO$ be an open subset of $\Omega$ such that $\alpha\geq
\delta>0$ on $\scrO$. Then, for $\sigma \in \R$ the unbounded
operator $i\sigma\id -\dot {A}$ is invertible on $\dot \H$ and
for there exist $C>0$ such  that 
\begin{align}
  \label{eq: precise resolvent estimate}
  \Norm{ (i\sigma\id-\dot{A})^{-1}}{\mathcal{L}(\dot{\H})}
  \leq C e^{C |\sigma|^{1/2}}, \qquad \sigma \in \R.
  \end{align}
\end{theorem}
\begin{proof}
By Proposition~\ref{prop: position spectrum boundary damping}
$i\sigma\id -\dot {A}$ is indeed invertible. Observe that it then
suffices to prove the resolvent estimate~\eqref{eq: precise resolvent
  estimate} for $|\sigma| \geq \sigma_0$ for some $\sigma_0>0$. 
  
Let $U= \transp (u^0,u^1)\in D(\dot A)$ and $F= \transp (f^0,f^1)\in \dot{\H}$ be such that 
$(i\sigma\id-\dot{A}) U=F$. This reads
\begin{align*}
  f^0=i\sigma u^0+u^1,
  \qquad
  f^1=-\Pell u^0+(i\sigma-\alpha)u^1.
\end{align*}
which gives
\begin{equation*}
  (\Pell -\sigma^2-i\sigma\alpha)u^0=f 
\end{equation*}
with $f=(i\sigma-\alpha)f^0-f^1$. Computing the $L^2$-inner product
with $u^0$ one finds
\begin{align*}
  \inp{(\Pell - \sigma^2) u^0}{u^0}_{L^2(\Omega)}
  - i \sigma  \inp{\alpha u^0}{u^0}_{L^2(\Omega)}
  =  \inp{f}{u^0}_{L^2(\Omega)}.
\end{align*}
As $\alpha \geq 0$,  computing the imaginary part one obtains
\begin{align*}
  \sigma \Norm{\alpha^{1/2} u^0}{L^2(\Omega)}^2
  =-\Im  \inp{f}{u^0}_{L^2(\Omega)}.
  \end{align*}
  Since $\alpha\geq \delta>0$ in $\scrO$ by assumption and since we consider
  $|\sigma| \geq \sigma_0$ one has
\begin{align*}
  \delta \sigma^0 \Norm{u^0}{L^2(\scrO)}^2
  \leq \Norm{f}{L^2(\Omega)}
  \Norm{u^0}{L^2(\Omega)}.
\end{align*}
Applying Theorem~\ref{th: observation} (with Remark~\ref{remark: observability third-order perturbation}) one has
\begin{align*}
  \Norm{u^0}{H^4(\Omega)}
  \lesssim e^{C|\sigma|^{1/2}}
  \big(
  \Norm{f}{L^2(\Omega)}+\Norm{u^0}{L^2(\scrO)}
  \big),
\end{align*}
replacing $|\sigma|$ by $|\sigma|^2$ therein.
 Thus, we obtain
  \begin{align*}
  \Norm{u^0}{H^4(\Omega)}
  \lesssim e^{C|\sigma|^{1/2}}
  \big(
  \Norm{f}{L^2(\Omega)}+\Norm{f}{L^2(\Omega)}^{1/2}
  \Norm{u^0}{L^2(\Omega)}^{1/2}
  \big),
  \end{align*}
  for $|\sigma| \geq \sigma_0$. 
  With Young inequality we write, for $\eps>0$, 
  \begin{align*}
    e^{C|\sigma|^{1/2}}\Norm{f}{L^2(\Omega)}^{1/2}
    \Norm{u^0}{L^2(\Omega)}^{1/2}
    \lesssim  \eps^{-1}e^{2C|\sigma|^{1/2}}\Norm{f}{L^2(\Omega)}
    + \eps \Norm{u^0}{L^2(\Omega)}.
  \end{align*}
  Thus, with $\eps$ chosen \suff small one obtains
  \begin{align*}
  \Norm{u^0}{H^4(\Omega)}
  \lesssim e^{C|\sigma|^{1/2}}
  \Norm{f}{L^2(\Omega)}.
  \end{align*}
Since $u^1=f^0-i\sigma u^0$  and  $f=(i\sigma-\alpha)f^0-f^1$ we finally obtain that 
 \begin{align*}
   \Norm{u^0}{H^4(\Omega)}
   +\Norm{u^1}{L^2(\Omega)}
   &\lesssim  e^{C|\sigma|^{1/2}}
     \big( \Norm{f^0}{L^2(\Omega)}+\Norm{f^1}{L^2(\Omega)}\big) \\
   &\lesssim e^{C|\sigma|^{1/2}} \Norm{F}{\dot \H}.
 \end{align*}
 Since $u^0 \in H^4(\Omega)$ one has
 \begin{align*}
   \big|\inp{\Pell u^0}{u^0}_{L^2(\Omega)}\big|
   \leq \Norm{u^0}{H^4(\Omega)}\Norm{u^0}{L^2(\Omega)}
   \leq \Norm{u^0}{H^4(\Omega)}^2
 \end{align*}
 and thus one finally obtains
 \begin{align*}
   \Norm{U}{\dot \H}^2  = \inp{\Pell u^0}{u^0}_{L^2(\Omega)}
   +\Norm{u^1}{L^2(\Omega)}^2
   &\lesssim e^{C|\sigma|^{1/2}} \Norm{F}{\dot \H}^2,
 \end{align*}
 which concludes the proof of the resolvent estimate~\eqref{eq: precise resolvent estimate}.
\end{proof}

\subsection{Stabilization result}
\label{sec: Stabilization result}
As an application of the resolvent estimate of
Theorem~\ref{res}, we give a logarithmic stabilization result of the damped
plate equation \eqref{eq: damped plate equation - intro}.

For the plate generator $(A,D(A))$ its iterated domains are
inductively given by
\begin{align*}
  D(A^{n+1}) = \{ U \in D(A^n); A U \in D(A^n)\}.
\end{align*}
With Proposition~\ref{prop: strong solutions plate equation}, for $Y^0 =\transp(y^0,y^1) \in D(A^n)$ then the first
component of $Y(t) = S(t) Y^0$ is precisely the solution to \eqref{eq:
  damped-plate equation}.
One has $Y(t) = \dot{Y}(t) + \Pi_{\mathcal N} Y^0$ with $\dot{Y}(t)
=\dot{S}(t) \Pi_{\dot{\H}} Y^0$ with the semigroup $\dot{S}(t)$
defined in Section~\ref{sec: The plate semigroup generator}. Moreover,
by \eqref{eq: connection energy norm H dot} the energy of $y(t)$ is
given by the square of the $\dot{\H}$-norm of $\dot{Y}(t)$.

With the resolvent estimate of Theorem~\ref{res}, with the result of
Theorem 1.5 in \cite{BD} one obtains the following bound for the
energy of $y(t)$:
\begin{align}
  \label{eq: logarithmic energy decay}
  \E(y)(t) &= \Norm{\dot{Y} (t) }{\dot{\H}}^2
             \leq  \frac{C}{\big(\log(2+t)\big)^{4 n}}
             \Norm{A^n Y^0}{\dot{\H}}^2.
\end{align}  
We have thus obtain the following theorem.
\begin{theorem}[logarithmic stabilization for the damped plate equation]
  \label{theorem: stabilisation theorem}
  Assume that conditions (i) to (iv) of Section~\ref{sec: Solutions to
    the damped plate equations} hold. Let $n \in \N$, $n \geq 1$. Then, there exists $C>0$ such
  that for any $Y^0 =\transp(y^0,y^1) \in D(A^n)$ the associated
  solution $y(t)$ of the damped plate equation~\eqref{eq: damped plate equation - intro} has the logarithmic energy decay given by \eqref{eq: logarithmic energy decay}.
\end{theorem}
Note that for $n=1$ using the form of $A$ and \eqref{eq: norm H2B-bis}
one recovers the statement of Theorem~\ref{theorem: stabilisation theorem-intro} in the introductory section.

\appendix

\section{Some technical results and proofs}

\subsection{A perfect elliptic estimate}
\label{SB}
Here we consider $a(\y',\xi_d)$ polynomial in the $\xi_d$ variable and
such that its root have negative imaginary parts microlocally.
\begin{lemma}\label{el1}
  Let $\k_0 >0$.
Let $a(\y',\xi_d)\in \Ssc^{k,0}$, with $\y'=(x,\xi',\tau,\sigma)$ and with
$k\geq 1$, 
that is,  $a(\y',\xi_d) = \sum_{j=0}^k a_j (\y') \xi_d^{k-j}$, 
and where the coefficients $a_j$ are homogeneous in
$(\xi',\tau,\sigma)$. Moreover, assume that $a_0(\y') =1$.
Set $A=\Op(a)$.

Let $\U$ be a conic open subset of  $W \times\R^{d-1}\times
[0,+\infty) \times[0,+\infty)$ where $\tau \geq \k_0 \sigma$ and such
that all the roots of $a(\y',\xi_d)$ have a negative imaginary part for $\y'\in\U$.

Let $\chi(\y')\in \Ssct^{0}$ be homogeneous of degree zero and such
that $\supp(\chi)\subset\U$ and $N\in\N$.  Then there exist $C>0$, $C_N>0$, and $\tau_0>0$ such that 
\begin{align*}
  \Normsc{\Opt(\chi)v}{k}
  + \normsc{\trace(\Opt(\chi)v)}{k-1,1/2}
  \leq C
  \Norm{A \Opt(\chi) v}{+}
  + \Normsc{v}{k,-N},
  \end{align*}
for $w\in\Sbarp$ and $\tau\geq \max (\tau_0,\k_0 \sigma)$.
\end{lemma}
We refer to \cite{ML} for a proof  (see Lemma~4.1 therein and its
proof that adapts to the presence of the parameter $\sigma$ with
$\sigma\lesssim \tau$ in a straightforward manner). 

\subsection{Basic resolvent estimation}
\label{sec: lemma: a priori resolvent estimate}
Here we provide a proof of Lemma~\ref{lemma: a priori resolvent
  estimate}

Let $U= \transp (u^0, u^1) \in D(\dot{A})$.  With \eqref{eq: inner product dot H} We write
      \begin{align*} 
        \inp{(z \id_{\dot{\H}} - \dot{A})U}{U}_{\dot{\H}} 
        &= \left( 
        \begin{pmatrix}z u^0 + u^1\\zu^1 - \Pell u^0 - \alpha u^1 \end{pmatrix}, 
        \begin{pmatrix} u^0 \\u^1\end{pmatrix}
          \right)_{\dot{\H}}\\
        &= z  \Norm{U}{\dot{\H}}^2 
          + \dup{\Pell u^1}{\ovl{u^0}}_{H^{-2}_B,H^2_B} 
           - \inp{\Pell u^0}{u^1}_{L^2(\Omega)}
          - \inp{\alpha u^1}{u^1}_{L^2 (\Omega)}
         \\
        &= z \Norm{U}{\dot{\H}}^2 
          + 2 i \Im \inp{u^1}{\Pell u^0}_{L^2 (\Omega)} 
         - \inp{\alpha u^1}{u^1}_{L^2 (\Omega)}.
      \end{align*}
      Computing the real part one obtains
      \begin{align}
        \label{eq: scalar product sigma id -A boundary damping}
        - \Re \inp{(z \id_{\dot{\H}}- \dot{A})U}{U}_{\dot{\H}} 
        = - \Re (z) \Norm{U}{\dot{\H}}^2 
        + \inp{\alpha u^1}{u^1}_{L^2 (\Omega)}.
      \end{align}
      As $\alpha \geq 0$ and $\Re z<0$, this gives
       \begin{align*}
        |\Re \inp{(z \id_{\dot{\H}} - \dot{A})U}{U}_{\dot{\H}} |
        \geq |\Re (z) |\, \Norm{U}{\dot{\H}}^2 ,
      \end{align*}
      which yields the conclusion of Lemma~\ref{lemma: a priori resolvent estimate}.  \hfill \qedsymbol \endproof

\subsection{Basic estimation for the resolvent set}
\label{sec: prop: position spectrum boundary damping}

Here we provide a proof of Proposition~\ref{prop: position spectrum
  boundary damping}.

Let $z \in \C$. We consider the two cases.

\medskip\noindent
  \paragraph{{\bfseries Case 1: $\bld{\Re z <0}$.}}
  By Lemma~\ref{lemma: a priori resolvent estimate}
      $z \id_{\dot{\H}} - \dot{A}$ is injective. Moreover, as its adjoint 
      $\ovl{z} \id_{\dot{\H}} - \dot{A}^*$ is injective and satisfies $\Norm{(\ovl{z} \id_{\dot{\H}} - \dot{A}^*) U}{\dot{\H}} \gtrsim \Norm{U}{\dot{\H}}$ for $U \in
      D(\dot{A})$ by Lemma~\ref{lemma: a priori resolvent estimate
        bis} the map
      $z \id_{\dot{\H}} - \dot{A}$ is surjective (see for instance \cite[Theorem~2.20]{Brezis:11}). The estimation of
      Lemma~\ref{lemma: a priori resolvent estimate}  then gives the
      continuity of the operator $(z \id_{\dot{\H}} - \dot{A})^{-1}$ on $\dot{\H}$. 
      
     \medskip\noindent
    \paragraph{{\bfseries Case 2: $\bld{\Re z =0}$.}}
    We start by proving the injectivity of $z \id_{\dot{\H}} - \dot{A}$. 
      Let thus $U =\transp (u^0,u^1) \in D(\dot{A})$ be such that $z U - \dot{A}U=0$.
      This gives
      \begin{align}
        \label{eq: injectivity sigma - A boundary damping}
        z u^0 +u^1 =0, \quad 
        - \Pell u^0  + (z-\alpha) u^1 =0.
      \end{align}
      First, if $z=0$ one has $u^1=0$, and then $\Pell u^0=0$. Thus,
      $u^0 \in \ker(\Pell)$ given $U \in \mathcal N = \ker(A)$. From
      the definition of $\dot \H$ this gives $U=0$.

  Second,
      if now $z \neq 0$, using \eqref{eq: scalar product sigma id -A
        boundary damping} we obtain
      \begin{align*}
        0 =  \Re \inp{(z \id_{\dot{\H}} - \dot{A})U}{U}_{\dot{\H}} 
        = -  \inp{\alpha u^1}{u^1}_{L^2 (\Omega)}.
      \end{align*}
      As $\alpha\geq 0$, this implies that $u^0$ vanishes {a.e.\@\xspace} on $\supp(\alpha)$.
      Observe that 
      \begin{align*}
          \Pell u^0 = z u^1 = -z^2 u^0. 
      \end{align*}
      The function $u^0$ is thus an eigenfunction for $\Pell$ that
      vanishes on an open set. 
      With the unique continuation property we obtain that
      $u^0$ vanishes in $\Omega$ and $u^1$ as well.

      \medskip If we now prove that $z \id_{\dot{\H}} - \dot{A}$ is
      surjective, the result then follows from the closed graph
      theorem as $\dot{A}$ is a closed operator. We write
      $z \id_{\dot\H} - \dot{A} = T +\id_{\dot\H}$ with
      $T = (z- 1) \id_{\dot\H} - \dot{A}$. By the first part of the
      proof, $T$ is invertible with a bounded inverse. The operator
      $T$ is unbounded on $\dot{\H}$. We denote by $\tilde{T}$ the
      restriction of $T$ to $D(\dot{A})$ equipped with the graph-norm
      associated with $\dot{A}$. The operator $\tilde{T}$ is bounded. It is also
      invertible.  It is thus a bounded
      Fredholm operator of index $\ind \tilde{T}=0$.  Similarly, we denote
      by $\iota$ the injection of $D(\dot{A})$ into $\dot{\H}$ and
      $\tilde{A}$ the restriction of $\dot{A}$ on $D(\dot{A})$
      viewed as a bounded operator. We have
      $z \iota - \tilde{A} = \tilde{T} + \iota$.  Since $\iota$
      is a compact operator, we obtain that $z \iota - \tilde{A}$ is
      also a bounded Fredholm operator of index $0$.
      Hence, $z \iota - \tilde{A}$ is surjective since
      $z \id_{\H} - \dot{A}$ is injective as proven above. Consequently,
      $z \id_\H - \dot{A}$ is surjective.
      This concludes the proof of Proposition~\ref{prop: position
        spectrum boundary damping}.
      \hfill \qedsymbol \endproof

\end{document}